\documentclass[11pt, twoside, leqno]{article}

\usepackage{amssymb}
\usepackage{amsmath}
\usepackage{amsthm}
\usepackage{color}
\usepackage{mathrsfs}

\usepackage{indentfirst}

\usepackage{txfonts}

\allowdisplaybreaks

\def\red{\color{red}}

\def\rr{{\mathbb R}}
\def\rn{{\mathbb{R}^n}}

\def\zz{{\mathbb Z}}
\def\cc{{\mathbb C}}

\def\nn{{\mathbb N}}

\def\ca{{\mathcal A}}

\def\cf{{\mathcal F}}

\def\cm{{\mathcal M}}

\def\cq{{\mathcal Q}}

\def\dsup{\displaystyle\sup}

\def\cs{{\mathcal S}}

\def\fz{\infty }

\def\gz{{\gamma}}

\def\lz{\lambda}

\def\lf{\left}
\def\r{\right}

\def\la{\langle}
\def\ra{\rangle}
\def\hs{\hspace{0.25cm}}
\def\ls{\lesssim}

\def\noz{\nonumber}

\def\wh{\widehat}

\def\gfz{\genfrac{}{}{0pt}{}}

\def\dist{\mathop\mathrm{\,dist\,}}

\def\BMO{\mathop\mathrm{\,BMO\,}}

\def\loc{{\mathop\mathrm{\,loc\,}}}
\def\supp{\mathop\mathrm{\,supp\,}}

\def\XXint#1#2#3{{\setbox0=\hbox{$#1{#2#3}{\int}$ }
\vcenter{\hbox{$#2#3$ }}\kern-.6\wd0}}

\DeclareMathOperator{\esssup}{ess\,sup}
\DeclareMathOperator{\essinf}{ess\,inf}

\def\la{{\langle}}
\def\ra{{\rangle}}
\def\({\left(}
\def \){ \right)}

\def\lz{{\lambda}}

\def\BB{{\mathbb B}}

\newtheorem{theorem}{Theorem}[section]
\newtheorem{lemma}[theorem]{Lemma}

\newtheorem{assumption}[theorem]{Assumption}
\newtheorem{proposition}[theorem]{Proposition}

\theoremstyle{definition}
\newtheorem{remark}[theorem]{Remark}
\newtheorem{definition}[theorem]{Definition}
\renewcommand{\appendix}{\par
   \setcounter{section}{0}%
   \setcounter{subsection}{0}%
   \setcounter{subsubsection}{0}%
   \gdef\thesection{\@Alph\c@section}%
   \gdef\thesubsection{\@Alph\c@section.\@arabic\c@subsection}%
   \gdef\theHsection{\@Alph\c@section.}%
   \gdef\theHsubsection{\@Alph\c@section.\@arabic\c@subsection}%
   \csname appendixmore\endcsname
 }

\numberwithin{equation}{section}

\begin{document}

\title{\bf\Large Weak Hardy-Type Spaces Associated with Ball Quasi-Banach Function Spaces
II: Littlewood--Paley Characterizations and Real Interpolation\footnotetext{\hspace{-0.35cm} 2010 {\it
Mathematics Subject Classification}. Primary 42B30;
Secondary 42B25, 42B20, 42B35, 46E30.
\endgraf {\it Key words and phrases.} ball quasi-Banach function space, weak Hardy space, weak tent space,
Orlicz-slice space, maximal function, Littlewood--Paley function, real interpolation.
\endgraf This project is supported
by the National Natural Science Foundation of China (Grant Nos.
11726621, 11726622, 11761131002, 11571039, 11671185 and 11871100).}}
\author{Songbai Wang, Dachun Yang\,\footnote{Corresponding author,
E-mail: \texttt{dcyang@bnu.edu.cn}/{\red June 28, 2019}/Final version.},
\  Wen Yuan and Yangyang Zhang}
\date{}
\maketitle

\vspace{-0.7cm}

\begin{center}

\begin{minipage}{13cm}
{\small {\bf Abstract}\quad Let $X$ be a ball quasi-Banach function space on ${\mathbb R}^n$.
In this article, assuming that the powered Hardy--Littlewood maximal operator
satisfies some Fefferman--Stein vector-valued maximal inequality on $X$ as well as it is
bounded on both the weak ball quasi-Banach function space $WX$ and the associated space,
the authors establish various Littlewood--Paley
function characterizations of $WH_X({\mathbb R}^n)$ under some weak assumptions on the Littlewood--Paley
functions. The authors also prove that the real interpolation intermediate space
$(H_{X}({\mathbb R}^n),L^\infty({\mathbb R}^n))_{\theta,\infty}$, between the Hardy
space associated with $X$, $H_{X}({\mathbb R}^n)$, and the Lebesgue space
$L^\infty({\mathbb R}^n)$, is $WH_{X^{{1}/{(1-\theta)}}}({\mathbb R}^n)$, where $\theta\in (0, 1)$.
All these results are of wide applications. Particularly,
when $X:=M_q^p({\mathbb R}^n)$ (the Morrey space), $X:=L^{\vec{p}}({\mathbb R}^n)$ (the mixed-norm Lebesgue space)
and $X:=(E_\Phi^q)_t({\mathbb R}^n)$ (the Orlicz-slice space), all these results are even new;
when $X:=L_\omega^\Phi({\mathbb R}^n)$ (the weighted Orlicz space),
the result on the real interpolation is new
and, when  $X:=L^{p(\cdot)}({\mathbb R}^n)$ (the variable Lebesgue space) and
$X:=L_\omega^\Phi({\mathbb R}^n)$,
the Littlewood--Paley function characterizations of $WH_X({\mathbb R}^n)$ obtained in this
article improves the existing results via weakening the assumptions on the Littlewood--Paley
functions.
}
\end{minipage}
\end{center}

\vspace{0.2cm}

\tableofcontents

\vspace{0.2cm}

\section{Introduction\label{s1}}

This is the second part of the research of the authors on weak Hardy-type spaces associated with
ball quasi-Banach function spaces. The first part was presented in \cite{zwyy}.

It is well known that the classical Hardy space $H^p(\rn)$ with $p\in (0,1]$ was originally introduced
by Stein and Weiss \cite{SW} and further exploited by Fefferman and Stein \cite{FS0},
which plays a key role in harmonic analysis and partial differential equations. The articles
\cite{FS0,SW} provided many new ideas for the real-variable theory of function spaces,
which reveal the intrinsic connections among some important notions in
harmonic analysis, such as harmonic functions, maximal functions and square functions. In
recent decades, various variants of classical Hardy spaces have been introduced
and their real-variable theories have been well developed; these variants include weighted Hardy
spaces (see \cite{ST}), (weighted) Herz--Hardy spaces (see, for instance,
\cite{CL,GC1,GCH,LY1,LY2}),  (weighted) Hardy--Morrey spaces (see, for instance
\cite{JW,Sa,H}), Hardy--Orlicz spaces (see, for instance, \cite{IV,Se,V,NS1,YY}),
Lorentz Hardy spaces (see, for instance, \cite{AST}), Musielak--Orlicz Hardy spaces
(see, for instance, \cite{K,YLK}) and variable Hardy spaces (see, for instance, \cite{CUW,NS,YZN}).
Observe that these elementary spaces on which the aforementioned Hardy spaces were built,
such as (weighted) Lebesgue spaces, (weighted) Herz spaces, (weighted) Morrey spaces,
mixed-norm Lebesgue spaces,
Orlicz spaces, Lorentz spaces, Musielak--Orlicz spaces and variable Lebesgue spaces,
are all included in a generalized framework called ball quasi-Banach function spaces
which were introduced, very recently, by Sawano et al. \cite{SHYY}.
Moreover, Sawano et al. \cite{SHYY} and Wang et al. \cite{wyy19} established a unified real-variable theory
for Hardy spaces associated with ball quasi-Banach function spaces on ${\mathbb R}^n$
and gave some applications of these Hardy-type spaces to the boundedness of
Calder\'on--Zygmund operators and pseudo-differential operators.

Recall that ball quasi-Banach function spaces are a generalization of
quasi-Banach function spaces. Compared with quasi-Banach function spaces,
ball quasi-Banach function spaces contain more function spaces.
For instance, Morrey spaces, mixed-norm Lebesgue spaces, weighted Lebesgue
spaces and Orlicz-slice spaces are all ball quasi-Banach function spaces,
which are not quasi-Banach function spaces and hence the class
of quasi-Banach function spaces is a proper subclass of ball quasi-Banach function spaces
(see \cite{SHYY,zwyy} for more details). Let $X$ be a ball quasi-Banach function space
(see \cite{SHYY} or Definition \ref{Debqfs} below). Sawano et al. \cite{SHYY} introduced
the Hardy space associated with $X$, $H_X(\rn)$, via the grand maximal function
(see \cite{SHYY} or Definition \ref{DeHX} below). Assuming that the Hardy--Littlewood maximal function is bounded on
the $p$-convexification of $X$, Sawano et al. \cite{SHYY} established
several different  maximal function characterizations of $H_X(\rn)$
and, assuming the Fefferman--Stein vector-valued maximal inequality on $X$ and the boundedness on the associated
space of the powered Hardy--Littlewood maximal operator, Sawano et al. \cite{SHYY}
then obtained the atomic characterization of $H_X(\rn)$, which when $X:=L^p(\rn)$
with $p\in (0,1]$ was originally obtained by Coifman \cite{C} and Latter \cite{La}.

Recall that, to find the biggest function space $\ca$ such that Calder\'on--Zygmund operators
are bounded from $\ca$ to $WL^1(\rn)$, Fefferman and Soria \cite{FSo} originally introduced the weak Hardy space
$WH^1(\rn)$ and did obtain the boundedness of the convolutional Calder\'on--Zygmund operator
with a kernel satisfying the Dini condition
from $WH^1(\rn)$ to $WL^1(\rn)$ by using the $\infty$-atomic characterization of $WH^1(\rn)$.
It is well known that the classic Hardy spaces  $H^p(\rn)$, with $p\in(0,1]$,
are good substitutes of Lebesgue spaces $L^p(\rn)$  when studying the boundedness
of some Calder\'on--Zygmund operators. For instance, if $\delta\in(0,1]$ and $T$ is
a convolutional $\delta$-type Calder\'on--Zygmund operator, then $T$ is bounded on
$H^p(\rn)$ for any given $p\in(n/(n+\delta),1]$ (see \cite{AM1986}).  However, this is not
true when $p=n/(n+\delta)$ which is called the \emph{critical case} or the \emph{endpoint case}.
Liu \cite{L} introduced the weak Hardy space $WH^p(\rn)$ with any given $p\in(0,1]$ and
proved that the aforementioned operator $T$ is bounded from $H^{n/(n+\delta)}(\rn)$ to $WH^{n/(n+\delta)}(\rn)$
via first establishing the $\infty$-atomic characterization of the weak Hardy space $WH^p(\rn)$.
Thus, the classical weak Hardy space $WH^p(\rn)$ plays an irreplaceable role in the study
of the boundedness of operators in the critical case.
Another main motivation to develop a real-variable
theory of $WH^p(\rn)$ is that the weak Hardy space $WH^p(\rn)$ naturally appears as
the intermediate spaces of the real interpolation between the Hardy space $H^p(\rn)$
and the Lebesgue space $L^\infty(\rn)$ (see \cite{FRS}).
Recently, He \cite{He} and Grafakos and He \cite{GH} further studied the vector-valued weak Hardy space
$H^{p,\infty}(\rn,\ell^2)$ with $p\in(0,\infty)$. In 2016, Liang et al. \cite{LYJ} (see also \cite{YLK})
considered the weak Musielak--Orlicz type Hardy space $WH^\varphi(\rn)$, which covers
both the weak Hardy space $WH^p(\rn)$ and the weighted weak Hardy space
$WH^p_\omega(\rn)$ from \cite{QY},  and obtained various equivalent characterizations of
$WH^\varphi(\rn)$ in terms of maximal functions, atoms, molecules and Littlewood--Paley functions,
as well as the boundedness of Calder\'on--Zygmund operators in the critical case. Meanwhile,
Yan et al. \cite{YYYZ} developed a real-variable theory of variable weak Hardy spaces
$WH^{p(\cdot)}(\rn)$ with $p(\cdot)\in C^{\log}(\rn)$.

Let $X$ be a ball quasi-Banach function space on ${\mathbb R}^n$ introduced by Sawano et al.
in \cite{SHYY}.
In \cite{zwyy}, we introduced the weak Hardy-type
space $WH_X({\mathbb R}^n)$, associated with $X$.
Assuming that the powered Hardy--Littlewood maximal operator
satisfies some Fefferman--Stein vector-valued maximal inequality on $X$ as well as it is
bounded on both the weak ball quasi-Banach function space $WX$ and the associated space,
we also established several real-variable characterizations of $WH_X({\mathbb R}^n)$, respectively, in terms of
various maximal functions, atoms and molecules in \cite{zwyy}.
In this article, applying those characterizations of
$WH_X({\mathbb R}^n)$ obtained in \cite{zwyy},
we establish various Littlewood--Paley
function characterizations of $WH_X(\rn)$.
Particularly, our assumptions on the Littlewood--Paley
functions are much weaker than the corresponding assumptions in \cite{SHYY,LYJ,YYYZ}
(see Remark \ref{wa} below).
We point out that, in the proof of the Lusin area function characterization
of the weak Musielak--Orlicz type Hardy space $WH^\varphi(\rn)$ in
\cite{LYJ}, there exists a gap in lines 14 and 16 of \cite[p.\,662]{LYJ}.
The same gap also exists in \cite[(6.5)]{YYYZ}, \cite[(4.24)]{hyy} and the proof
of \cite[Theorem 4.1]{hlyy}. Using some ideas from Sawano et al. \cite{SHYY} and Liu \cite{L},
in the proof of the Lusin area function characterization of $WH_X(\rn)$ of this article,
we seal this gap under some even weaker assumptions on the Lusin area function.
We also prove that the real interpolation intermediate space
$(H_{X}({\mathbb R}^n),L^\infty({\mathbb R}^n))_{\theta,\infty}$ is
$WH_{X^{{1}/{(1-\theta)}}}({\mathbb R}^n)$, where $\theta\in (0, 1)$.
All these results are of wide applications.
when  $X:=L^{p(\cdot)}({\mathbb R}^n)$ (the variable Lebesgue space) and
$X:=L_\omega^\Phi({\mathbb R}^n)$ (the weighted Orlicz space),
the Littlewood--Paley function characterizations of $WH_X({\mathbb R}^n)$ obtained in this
article improve the existing results in \cite{LYJ,YYYZ} via weakening
the assumptions on the Littlewood--Paley functions.
When $X:=L^{p(\cdot)}({\mathbb R}^n)$  and
$$p_-:=\mathop{\mathrm{ess\,sup}}_{x\in\rn}\,p(x)\in(0,1],$$
the $g_\lambda^\ast$-function characterization of $WH_X({\mathbb R}^n)$ obtained in this
article improves the existing results via widening
the range of $\lambda\in(1+\frac{2}{\min\{2,p_-\}},\infty)$ in \cite{YYYZ} to
$$\lz\in\left(\max\lf\{\frac{2}{\min\{1,p_-\}},
1-\frac2{\max\{1,p_+\}}+\frac{2}{\min\{1,p_-\}}\r\},\infty\r),$$
where $p_+:={\esssup}_{x\in\rn}\,p(x)$.
When $X:=L_\omega^\Phi({\mathbb R}^n)$,
Liang et al. \cite{LYJ} obtained all main theorems obtained in Sections \ref{s3}
of this article for the weighted weak Orlicz--Hardy space $WH_\omega^\Phi(\rn)$ as the special case,
however, the result for the real interpolation in Section \ref{s4}
(see Theorem \ref{th5} below),
$(H_\omega^\Phi(\rn),L^\infty(\rn))_{\theta,\infty}=WH_\omega^{\Phi_{1/(1-\theta)}}(\rn)
$ with $\theta\in(0,1)$, is new.
Moreover, when $X:=M_q^p({\mathbb R}^n)$ (the Morrey space), $X:=L^{\vec{p}}(\rn)$ (the mixed-norm Lebesgue space)
and $X:=(E_\Phi^q)_t({\mathbb R}^n)$ (the Orlicz-slice space), all these results
obtained in this article are new.

To be precise, this article is organized as follows.

In Section \ref{s2}, we recall some notions concerning the ball (quasi)-Banach function space $X$ and
the weak ball (quasi)-Banach function space $WX$. Then we state the assumptions of the Fefferman--Stein vector-valued
maximal inequality on $X$ (see Assumption \ref{a2.15} below)
and the boundedness on the $p$-convexification of $WX$ for the Hardy--Littlewood maximal operator
(see Assumption \ref{a2.17} below).

Section \ref{s3} contains some square function characterizations
of $WH_X(\rn)$, including its characterizations via the Lusin area function, the Littlewood--Paley
$g$-function and the Littlewood--Paley $g_\lambda^\ast$-function, respectively,
in Theorems \ref{Tharea}, \ref{Thgf} and \ref{Thgx} below.
We first prove Theorem \ref{Tharea}, the Lusin area function characterization of $WH_X(\rn)$.
In lines 14 and 16 of \cite[p.\,662]{LYJ},
when they used Calder\'on reproducing
formula to decompose a distribution $f$ into a sequence of atoms, Liang et al. did not prove
that the sum of the sequence of atoms converges to $f$ in the sense of distribution.
To seal this gap, we employ a different method from \cite{LYJ}.
Via borrowing some ideas from \cite{SHYY}, we first introduce the weak tent space
associated to $X$ and establish its atomic characterization (see Theorem \ref{That} below),
which is a key tool to decompose a distribution $f$ into a sequence of molecules.
Using the some ideas from Liu \cite{L}, we prove that the sum of the sequence of molecules
indeed converges in the sense of distribution [see \eqref{kkk} below].
Then, applying some ideas used in the proof of \cite[Theorem 3.21]{SHYY},
we prove that, under some even weaker assumptions on the Lusin area function, the distribution
which the sum of the aforementioned sequence of molecules converges to equals $f$
in the sense of distribution, and hence seal the aforementioned gap existed in the proof of \cite[Theorem 4.5]{LYJ}.
To obtain the Littlewood--Paley $g$-function characterization of $WH_X(\rn)$, we
first establish a discrete Calder\'on reproducing formula
under some weak assumptions (see Lemma \ref{Le49}(i) below), and then we apply
an approach initiated
by Ullrich \cite{U} and further developed by Liang et al. \cite{LSUYY}, which provides one way to control
the $WX$ norm of the Lusin area function by the corresponding norm of
the Littlewood--Paley $g$-function, via a key technical lemma
(see Lemma \ref{Le67} below) and an auxiliary Peetre type square function $g_{a,\ast}(f)$.
We point out that, applying the discrete Calder\'on reproducing formula established
in this article, we obtain the same estimate \cite[(2.66)]{U}
under weaker assumptions on the Schwartz function $\phi$ (see Lemma \ref{Le67} below),
which is the key tool to obtain the Littlewood--Paley $g$-function characterization of $WH_X(\rn)$.
To obtain the Littlewood--Paley $g_\lambda^\ast$-function characterization
of $WH_X(\rn)$,
via borrowing some ideas from \cite{AJ},
we use the atomic characterization of the weak tent space (see Theorem \ref{That} below), which is the
key tool to obtain a very useful estimate on
the change of angles in weak tent spaces (see Theorem \ref{Thca} below).
Using Theorem \ref{Thca} and the Lusin area function characterization of $WH_X(\rn)$, we then establish
the Littlewood--Paley $g_\lambda^\ast$-function characterization
of $WH_X(\rn)$.

The main purpose of Section \ref{s4} is to give the real interpolation
intermediate space between the
Hardy space $H_{X}(\rn)$, associated with $X$, and the Lebesgue space $L^\infty(\rn)$ for any given ball quasi-Banach
function space $X$.
We first show a real interpolation theorem between $X$ and $L^\infty(\rn)$ (see Theorem \ref{de1} below),
which plays a key role in proving the real interpolation theorem between
the Hardy space associated with $X$, $H_{X}(\rn)$, and the Lebesgue space $L^\infty(\rn)$
(see Theorem \ref{inte} below).
Then, by borrowing some ideas from \cite{C1977} and \cite{zyy},
we decompose any distribution $f\in WH_{X^{{1}/{(1-\theta)}}}(\rn)$ into ``good" and ``bad" parts
in Theorem \ref{de} below and prove that the following real interpolation intermediate space
$(H_{X}(\rn),L^\infty(\rn))_{\theta,\infty}$ is just $WH_{X^{{1}/{(1-\theta)}}}(\rn)
$, where $\theta\in (0, 1)$.

In Section \ref{s5}, we apply the above results, respectively, to the Morrey space,
the mixed-norm Lebesgue space, the variable Lebesgue space, the weighted Orlicz space
and the Orlicz-slice space. Recall that, in these five examples, only variable Lebesgue spaces
are quasi-Banach function spaces and the others are only ball quasi-Banach function spaces.

Finally, we make some conventions on notation. Let $\nn:=\{1,2,\ldots\}$, $\zz_+:=\nn\cup\{0\}$
and $\zz_+^n:=(\zz_+)^n$.
We always denote by $C$ a \emph{positive constant} which is independent of the main parameters,
but it may vary from line to line. We also use $C_{(\alpha,\beta,\ldots)}$ to denote a positive constant depending
on the indicated parameters $\alpha,\beta,\ldots.$ The \emph{symbol} $f\lesssim g$ means that $f\le Cg$.
If $f\lesssim g$ and $g\lesssim f$, we then write $f\sim g$. We also use the following
convention: If $f\le Cg$ and $g=h$ or $g\le h$, we then write $f\ls g\sim h$
or $f\ls g\ls h$, \emph{rather than} $f\ls g=h$
or $f\ls g\le h$. The \emph{symbol} $\lfloor s\rfloor$ for any $s\in\mathbb{R}$
denotes the maximal integer not greater
than $s$. We use $\vec0_n$ to denote the \emph{origin} of $\rn$ and let
$\mathbb{R}^{n+1}_+:=\rn\times(0,\infty)$.
If $E$ is a subset of $\rn$, we denote by $\mathbf{1}_E$ its
\emph{characteristic function} and by $E^\complement$ the set $\rn\setminus E$.
For any cube $Q:=Q(x_Q,l_Q)\subset\rn$,
with center $x_Q\in\rn$ and side length $l_Q\in(0,\infty)$,
and $\alpha\in(0,\infty)$, let $\alpha Q:=Q(x_Q,\alpha l_Q)$.
Denote by $\cq$ the set of all cubes having their edges parallel to the coordinate axes. For any
$\theta:=(\theta_1,\ldots,\theta_n)\in\zz_+^n$, let $|\theta|:=\theta_1+\cdots+\theta_n$. Furthermore,
for any cube $Q$ in $\rn$ and $j\in\zz_+$, let $S_j(Q):=(2^{j+1}Q)\setminus(2^jQ)$ with $j\in\nn$
and $S_0(Q):=2Q$. Finally, for any $q\in[1,\infty]$, we denote by $q'$ its \emph{conjugate exponent},
namely, $1/q+1/q'=1$.

\section{Preliminaries\label{s2}}

In this section, we recall some notions and preliminary facts on ball quasi-Banach function spaces defined in \cite{SHYY}.

Denote by the \emph{symbol} $\mathscr M(\rn)$ the set of
all measurable functions on $\rn$.
For any $x\in\rn$ and $r\in(0,\infty)$, let $B(x,r):=\{y\in\rn:\ |x-y|<r\}$ and
\begin{equation}\label{Eqball}
\BB:=\lf\{B(x,r):\ x\in\rn\quad\text{and}\quad r\in(0,\infty)\r\}.
\end{equation}

\begin{definition}\label{Debqfs}
A quasi-Banach space $X\subset\mathscr M(\rn)$ is called a \emph{ball quasi-Banach function space} if it satisfies
\begin{itemize}
\item[(i)] $\|f\|_X=0$ implies that $f=0$ almost everywhere;
\item[(ii)] $|g|\le |f|$ almost everywhere implies that $\|g\|_X\le\|f\|_X$;
\item[(iii)] $0\le f_m\uparrow f$ almost everywhere implies that $\|f_m\|_X\uparrow\|f\|_X$;
\item[(iv)] $B\in\BB$ implies that $\mathbf{1}_B\in X$, where $\BB$ is as in \eqref{Eqball}.
\end{itemize}

Moreover, a ball quasi-Banach function space $X$ is called a
\emph{ball Banach function space} if the norm of $X$
satisfies the triangle inequality: for any $f,\ g\in X$,
\begin{equation*}
\|f+g\|_X\le \|f\|_X+\|g\|_X
\end{equation*}
and, for any $B\in \BB$, there exists a positive constant $C_{(B)}$, depending on $B$, such that, for any $f\in X$,
\begin{equation*}
\int_B|f(x)|\,dx\le C_{(B)}\|f\|_X.
\end{equation*}
For any ball Banach function space $X$, the \emph{associate space} (\emph{K\"othe dual}) $X'$ is defined by setting
\begin{equation*}
X':=\lf\{f\in\mathscr M(\rn):\ \|f\|_{X'}:=\sup\lf\{\|fg\|_{L^1(\rn)}:\ g\in X,\ \|g\|_X=1\r\}<\infty\r\},
\end{equation*}
where $\|\cdot\|_{X'}$ is called the \emph{associate norm} of $\|\cdot\|_X$
(see, for instance, \cite[Chapter 1, Definitions 2.1 and 2.3]{BS}).
\end{definition}

\begin{remark}
\begin{itemize}
\item[(i)] By \cite[Proposition 2.3]{SHYY}, we know that, if $X$ is a ball Banach function space,
then its associate space $X'$ is also a ball Banach function space.
\item[(ii)] Recall that a quasi-Banach space $X\subset\mathscr M(\rn)$ is called a \emph{quasi-Banach function space} if it
is a ball quasi-Banach function space and it satisfies Definition \ref{Debqfs}(iv) with ball
replaced by any measurable set of finite measure.
\end{itemize}
\end{remark}

We still need to recall the notions of the convexity and the concavity of ball quasi-Banach spaces,
which come from, for instance, \cite[Definition 1.d.3]{LT}.
\begin{definition}\label{Debf}
Let $X$ be a ball quasi-Banach function space and $p\in(0,\infty)$.
\begin{itemize}
\item[(i)] The $p$-\emph{convexification} $X^p$ of $X$ is defined by setting $X^p:=\{f\in\mathscr M(\rn):\ |f|^p\in X\}$
equipped with the quasi-norm $\|f\|_{X^p}:=\||f|^p\|_X^{1/p}$.
\item[(ii)] The space $X$ is said to be $p$-\emph{concave} if there exists a positive constant $C$ such that,
for any sequence $\{f_j\}_{j\in\nn}$ of $X^{1/p}$,
$$
\sum_{j\in\nn}\|f_{j}\|_{X^{1/p}}\le C\lf\|\sum_{j\in\nn}|f_j|\r\|_{X^{1/p}}.
$$
Particularly, $X$ is said to be \emph{strictly $p$-concave} when $C=1$.
\end{itemize}
\end{definition}

Now we recall the notion of weak ball quasi-Banach function spaces as follows.

\begin{definition}\label{2.8}
Let $X$ be a ball quasi-Banach function space. The \emph{weak ball quasi-Banach function
space} $WX$ is defined to be the set of all measurable functions $f$ such that
\begin{equation*}
\|f\|_{WX}:=\sup_{\alpha\in(0,\infty)}
\lf\{\alpha\lf\|\mathbf{1}_{\{x\in\rn:\ |f(x)|>\alpha\}}\r\|_X\r\}<\infty.
\end{equation*}
\end{definition}

\begin{remark}\label{Rews}
\begin{itemize}
\item[(i)] Let $X$ be a ball quasi-Banach function space. For any $f\in X$ and $\alpha\in(0,\infty)$, we have
$\mathbf{1}_{\{x\in\rn:\ |f(x)|>\alpha\}}(x)\le|f(x)|/\alpha$ for any $x\in\rn$, which, together with Definition \ref{Debqfs}(ii),
further implies that $\sup_{\alpha\in(0,\infty)}\lf\{\alpha\|\mathbf{1}_{\{x\in\rn:
\ |f(x)|>\alpha\}}\|_X\r\}\le\|f\|_X$. This shows
that $X\subset WX$.

\item[(ii)] $WX$ is also a ball quasi-Banach function space (see \cite[Lemma 2.13]{zwyy}).
\end{itemize}
\end{remark}

\begin{remark}\label{Re213}
Let $X$ be a ball quasi-Banach function space. Then, by the Aoki--Rolewicz theorem
(see, for instance, \cite[Exercise 1.4.6]{G1}), one finds a positive constant $\nu\in(0,1)$ such that,
for any $N\in\nn$ and $\{f_j\}_{j=1}^N\subset\mathscr M(\rn)$,
\begin{align*}
\lf\|\sum_{j=1}^N|f_j|\r\|_{WX(\rn)}^\nu\le4\sum_{j=1}^N\lf\||f_j|\r\|_{WX}^\nu.
\end{align*}
\end{remark}

Denote by the \emph{symbol $L_{\loc}^1(\rn)$} the set of all locally integrable functions on $\rn$.
The \emph{Hardy--Littlewood maximal operator} $\cm$
is defined by setting, for any $f\in L_{\loc}^1(\rn)$ and $x\in\rn$,
\begin{equation}\label{mm}
\cm(f)(x):=\sup_{B\in\BB,B\ni x}\frac1{|B|}\int_B|f(y)|\,dy,
\end{equation}
where the supremum is taken over all balls $B\in\BB$ containing $x$.

To establish Littlewood--Paley characterizations of weak Hardy spaces associated with ball quasi-Banach
function spaces,
the approach used in this article heavily depends on the following both assumptions
on the Fefferman--Stein vector-valued inequality, respectively, on $X^{1/p}$ and $(WX)^{1/p}$.

\begin{assumption}\label{a2.15}
Let $X$ be a ball quasi-Banach function space and there exists a $p_-\in(0,\infty)$ such that,
for any given $p\in(0,p_-)$ and $s\in(1,\infty)$, there exists a positive constant $C$ such that,
for any $\{f_j\}_{j=1}^\infty\subset\mathscr M(\rn)$,
\begin{equation*}
\lf\|\lf\{\sum_{j\in\nn}\lf[\cm(f_j)\r]^s\r\}^{1/s}\r\|_{X^{1/p}}
\le C\lf\|\lf\{\sum_{j\in\nn}|f_j|^s\r\}^{1/s}\r\|_{X^{1/p}},
\end{equation*}
where $\cm$ is as in $\eqref{mm}$.
\end{assumption}

\begin{remark}\label{Refs}
Let $X$ and $p_-$ be the same as in Assumption \ref{a2.15}. Let
\begin{equation}\label{Eqpll}
\underline{p}:=\min\{p_-,\ 1\}.
\end{equation}
Then, for any given $r\in(0,\underline{p})$ and for any sequence $\{B_j\}_{j\in\nn}\subset\BB$
and $\beta\in[1,\infty)$, by Definition \ref{Debqfs}(ii),
the fact that $\mathbf{1}_{\beta B_j}\le[\beta^n\cm(\mathbf{1}_{B_j})]^{1/r}$
almost everywhere on $\rn$ for any $j\in\nn$, Definition \ref{Debf}(i) and
Assumption \ref{a2.15}, we have
\begin{align}\label{EqHLMS}
\lf\|\sum_{j\in\nn}\mathbf{1}_{\beta B_j}\r\|_{X}&\le\lf\|\sum_{j\in\nn}
\lf[\beta^n\cm(\mathbf{1}_{B_j})\r]^\frac1r\r\|_{X}
=\beta^\frac{n}{r}\lf\|\lf\{\sum_{j\in\nn}\lf[\cm(\mathbf{1}_{B_j})\r]^\frac1r\r\}^r\r\|_{X^{1/r}}^{1/r}\\\noz
&\le C\beta^\frac{n}{r}\lf\|\lf[\sum_{j\in\nn}\mathbf{1}_{B_j}\r]^r\r\|_{X^{1/r}}^{1/r}
=C\beta^\frac{n}{r}\lf\|\sum_{j\in\nn}\mathbf{1}_{B_j}\r\|_{X},
\end{align}
where the positive constant $C$ is independent of $\{B_j\}_{j\in\nn}$ and $\beta$.
\end{remark}

\begin{assumption}\label{a2.17}
Let $X$ be a ball quasi-Banach function space satisfying Assumption \ref{a2.15} for some $p_-\in(0,\infty)$.
Assume that, for any given $r\in(1,\infty)$ and $p\in(0,p_-)$,
there exists a positive constant $C$ such that, for any $\{f_j\}_{j\in\nn}\subset\mathscr M(\rn)$,
\begin{equation*}
\lf\|\lf\{\sum_{j\in\nn}\lf[\cm(f_j)\r]^r\r\}^{1/r}\r\|_{(WX)^{1/p}}
\le C\lf\|\lf\{\sum_{j\in\nn}|f_j|^r\r\}^{1/r}\r\|_{(WX)^{1/p}},
\end{equation*}
where $\cm$ is as in $\eqref{mm}$.
\end{assumption}

In what follows, we denote by $\cs(\rn)$ the \emph{space of all Schwartz functions},
equipped with the well-known topology determined by a countable family of
seminorms, and by $\cs'(\rn)$ its \emph{topological dual space},
equipped with the weak-$\ast$ topology. For any $N\in\nn$, let
\begin{equation*}
\cf_N(\rn):=\lf\{\varphi\in\cs(\rn):\sum_{\beta\in\zz_+^n,|\beta|\le N}
\sup_{x\in\rn}\lf[\lf(1+|x|\r)^{N+n}\lf|\partial_x^\beta\varphi(x)\r|\r]\le1\r\},
\end{equation*}
here and hereafter, for any $\beta:=(\beta_1,\ldots,\beta_n)\in\zz_+^n$
and $x\in\rn$,
$|\beta|:=\beta_1+\cdots+\beta_n$ and
$\partial_x^\beta:=(\frac{\partial}{\partial x_1})^{\beta_1}
\cdots(\frac{\partial}{\partial x_n})^{\beta_n}$.
For any given $f\in\cs'(\rn)$, the \emph{radial grand maximal function} $M_N^0(f)$
of $f$ is defined by setting, for any $x\in\rn$,
\begin{equation}\label{EqMN0}
M_N^0(f)(x):=\sup\lf\{|f\ast\varphi_t(x)|:\ t\in(0,\infty)\ \text{and}\ \varphi\in\cf_N(\rn)\r\},
\end{equation}
where, for any $t\in(0,\infty)$ and $\xi\in\mathbb R^n,\varphi_t(\xi):=t^{-n}\varphi(\xi/t)$.

\begin{definition}\label{DewSH}
Let $X$ be a ball quasi-Banach function space.
Then the \emph{weak Hardy-type space}
$WH_X(\rn)$ associated with $X$ is defined by setting
$$
WH_X(\rn):=\lf\{f\in\cs'(\rn):\ \|f\|_{WH_X(\rn)}:=\lf\|M_N^0(f)\r\|_{WX}<\infty\r\},
$$
where $M_N^0(f)$ is as in \eqref{EqMN0} with $N\in\nn$ sufficiently large.
\end{definition}

\begin{remark}
\begin{itemize}
\item[(i)] When $X:=L^p(\rn)$ with $p\in(0,1]$, the weak Hardy-type space $WH_X(\rn)$
coincides with the classical weak Hardy space $WH^p(\rn)$ (see, for instance, \cite[p.\,114]{L}).
\item[(ii)] By \cite[Theorem 3.2(ii)]{zwyy}, we know that, if
the Hardy-Littlewood maximal operator $\cm$ in \eqref{mm} is bounded on $(WX)^{1/r}$ and
$N\in[\lfloor\frac nr\rfloor+1,\infty)\cap\nn$,
then $WH_X(\rn)$ in Definition \ref{DewSH} is independent of
the choice of $N$.
\end{itemize}
\end{remark}

The following lemma is just \cite[Lemma 4.9]{zwyy}.

\begin{lemma}\label{Le45}
Let $r\in(0,\infty)$, $q\in(r,\infty]$ and $X$ be a ball quasi-Banach function space.
Assume that $X^{1/r}$ is a ball Banach function space and there exists a positive constant $C$
such that, for any $f\in(X^{1/r})'$,
$\|\cm^{((q/r)')}(f)\|_{(X^{1/r})'}\le C\lf\|f\r\|_{(X^{1/r})'}$, where $\cm$ is as in $\eqref{mm}$.
Then there exists a positive constant $C$
such that, for any sequence $\{B_j\}_{j\in\nn}$ of balls,
numbers $\{\lambda_j\}_{j\in\nn}\subset\cc$ and measurable functions
$\{a_j\}_{j\in\nn}$ satisfying that, for any $j\in\nn$, $\supp(a_j)\subset B_j$ and
$\|a_j\|_{L^q(\rn)}\le|B_j|^{1/q}$,
$$
\lf\|\lf(\sum_{j\in\nn}|\lambda_ja_j|^r\r)^\frac1r\r\|_{X}
\le C\lf\|\lf(\sum_{j\in\nn}|\lambda_j\mathbf{1}_{B_j}|^r\r)^\frac1r\r\|_{X}.
$$
\end{lemma}

Now we recall the notion of atoms associated with $X$, which origins from \cite[Definition 3.5]{SHYY}.

\begin{definition}\label{Deatom}
Let $X$ be a ball quasi-Banach function space, $q\in(1,\infty]$ and $d\in\zz_+$.
Then a measurable function $a$ on $\rn$ is called an $(X,\,q,\,d)$-\emph{atom}
if there exists a ball $B\in\BB$ such that
\begin{itemize}
\item[(i)] $\supp a:=\{x\in\rn:\ a(x)\neq0\}\subset B$;
\item[(ii)] $\|a\|_{L^q(\rn)}\le\frac{|B|^{1/q}}{\|\mathbf{1}_B\|_X}$;
\item[(iii)] $\int_{\rn}a(x)x^\alpha\,dx=0$ for any
$\alpha:=(\alpha_1,\ldots,\alpha_n)\in\zz_+^n$ with $|\alpha|\le d$,
here and hereafter, for any $x:=(x_1,\ldots,x_n)\in\rn$,
$x^\alpha:=x_1^{\alpha_1}\cdots x_n^{\alpha_n}$.
\end{itemize}
\end{definition}

The following atomic characterization of $WH_X(\rn)$ is just \cite[Theorems 4.2 and 4.8]{zwyy}.

\begin{lemma}\label{Thad}
Let $X$ be a ball quasi-Banach function space satisfying that, for some
given $r\in(0,1)$ and for any $\{f_j\}_{j\in\nn}\subset\mathscr M(\rn)$,
\begin{equation*}
\lf\|\lf\{\sum_{j\in\nn}\lf[\cm(f_j)\r]^{1/r}\r\}^{r}\r\|_{X^{1/r}}
\le C\lf\|\lf\{\sum_{j\in\nn}|f_j|^{1/r}\r\}^{r}\r\|_{X^{1/r}},
\end{equation*}
where $\cm$ is as in \eqref{mm} and the positive constant $C$ is independent of $\{f_j\}_{j\in\nn}$.
Assume that there exist both $\vartheta_0\in(1,\infty)$
and $p\in(0,\infty)$ such that $X$ is $\vartheta_0$-concave and
$\cm$ in \eqref{mm} is bounded on
$X^{1/(\vartheta_0 p)}$, and that there exists a $\gamma_0\in(0,\infty)$
such that $\cm$ is bounded on $(WX)^{1/\gamma_0}$.
Let $d\geq \lfloor n(1/p-1)\rfloor$ be a fixed nonnegative integer and $f\in WH_X(\rn)$.
Then there exists a sequence $\{a_{i,j}\}_{i\in\zz,j\in\nn}$ of
$(X,\,\infty,\,d)$-atoms supported, respectively, in balls $\{B_{i,j}\}_{i\in\zz,j\in\nn}$
satisfying that,
for any $i\in\zz$,
$\sum_{j\in\nn}\mathbf{1}_{cB_{i,j}}\le A$ with $c\in(0,1]$ and $A$ being
a positive constant independent of $f$ and $i$,
such that $f=\sum_{i\in\zz}\sum_{j\in\nn}\lambda_{i,j}a_{i,j}$ in $\cs'(\rn)$,
where $\lambda_{i,j}:=\widetilde A2^i\|\mathbf{1}_{B_{i,j}}\|_{X}$ for any $i\in\zz$ and $j\in\nn$,
with $\widetilde A$ being a positive constant independent of $i$, $j$ and $f$, and
$$
\sup_{i\in\zz}\lf\|\sum_{j\in\nn}
\frac{\lambda_{i,j}\mathbf{1}_{B_{i,j}}}{\|\mathbf{1}_{B_{i,j}}\|_{X}}\r\|_{X}\lesssim\|f\|_{WH_X(\rn)},
$$
where the implicit positive constant is independent of $f$.
\end{lemma}

\begin{lemma}\label{Thar}
Let $X$ be a ball quasi-Banach function space satisfying
Assumption \ref{a2.15} for some $p_-\in(0,\infty)$.
Assume that, for any given $r\in(0,\underline{p})$
with $\underline{p}$ as in \eqref{Eqpll},
$X^{1/r}$ is a ball Banach function space.
Assume that there exist $r_0\in(0,\underline{p})$ and $p_0\in(r_0,\infty)$ such that
\begin{equation*}
\lf\|\cm^{((p_0/r_0)')}(f)\r\|_{(X^{1/r_0})'}\le C\lf\|f\r\|_{(X^{1/r_0})'},
\end{equation*}
where $\cm$ is as in \eqref{mm} and the positive constant $C$ is independent of $f$.
Let $d\in\zz_+$ with $d\geq \lfloor n(1/\underline{p}-1)\rfloor$,
$c\in(0,1]$, $q\in(\max\{1,p_0\},\infty]$
and $A,\ \widetilde A\in(0,\infty)$ and let
$\{a_{i,j}\}_{i\in\zz,j\in\nn}$ be a sequence of $(X,\,q,\,d)$-atoms supported, respectively, in balls
$\{B_{i,j}\}_{i\in\zz,j\in\nn}$
satisfying that $\sum_{j\in\nn}\mathbf{1}_{cB_{i,j}}\le A$ for any $i\in\zz$,
$\lambda_{i,j}:=\widetilde A2^i\|\mathbf{1}_{B_{i,j}}\|_{X}$ for any $i\in\zz$ and $j\in\nn$,
the series
$
f:=\sum_{i\in\zz}\sum_{j\in\nn}\lambda_{i,j}a_{i,j}
$
converges in $\cs'(\rn)$
and
$$
\sup_{i\in\zz}\lf\|\sum_{j\in\nn}
\frac{\lambda_{i,j}\mathbf{1}_{B_{i,j}}}{\|\mathbf{1}_{B_{i,j}}\|_{X}}\r\|_{X}<\infty.
$$
Then $f\in WH_X(\rn)$ and
$$
\|f\|_{WH_X(\rn)}\lesssim\sup_{i\in\zz}\lf\|\sum_{j\in\nn}
\frac{\lambda_{i,j}\mathbf{1}_{B_{i,j}}}{\|\mathbf{1}_{B_{i,j}}\|_{X}}\r\|_{X},
$$
where the implicit positive constant is independent of $f$.
\end{lemma}

We also need the following molecular construction of $WH_{X}(\rn)$ from \cite[Theorem 5.3]{zwyy}.

\begin{definition}\label{Demol}
Let $X$ be a ball quasi-Banach function space, $\epsilon\in(0,\infty)$, $q\in[1,\infty]$ and $d\in\zz_+$.
A measurable function $m$ is called an $(X,\,q,\,d,\,\epsilon)$-\emph{molecule} associated with some ball $B\subset\rn$ if
\begin{enumerate}
\item[(i)] for any $j\in\nn$, $\|m\|_{L^q(S_j(B))}\le2^{-j\epsilon}|S_j(B)|^\frac1q\|\mathbf{1}_B\|_{X}^{-1}$,
where $S_0:=B$ and, for any $j\in\nn$, $S_j(B):=(2^jB)\setminus(2^{j-1}B)$;
\item[(ii)] $\int_\rn m(x)x^\beta\,dx=0$ for any $\beta\in\zz_+^n$ with $|\beta|\le d$.
\end{enumerate}
\end{definition}

\begin{lemma}\label{Thmolcha}
Let $X$ be a ball quasi-Banach function space satisfying
Assumption \ref{a2.15} for some $p_-\in(0,\infty)$.
Assume that, for any given $r\in(0,\underline{p})$
with $\underline{p}$ as in \eqref{Eqpll},
$X^{1/r}$ is a ball Banach function space and
assume that there exists a $p_+\in[p_-,\infty)$ such that, for any given $r\in(0,\underline{p})$
and $p\in(p_+,\infty)$, and for any $f\in(X^{1/r})'$,
\begin{equation*}
\lf\|\cm^{((p/r)')}(f)\r\|_{(X^{1/r})'}\le C\lf\|f\r\|_{(X^{1/r})'},
\end{equation*}
where $\cm$ is as in \eqref{mm} and the positive constant $C$ is independent of $f$.
Let $d\in\zz_+$ with $d\geq \lfloor n(1/\underline{p}-1)\rfloor$.
Let $q\in(\max\{p_+,1\},\infty]$, $\epsilon\in(n+d+1,\infty)$,
$A,\ \widetilde A\in(0,\infty)$ and $c\in(0,1]$, and let
$\{m_{i,j}\}_{i\in\zz,j\in\nn}$ be a
sequence of $(X,\,q,\,d,\,\epsilon)$-molecules
associated, respectively, with balls $\{B_{i,j}\}_{i\in\zz,j\in\nn}$
satisfying that $\sum_{j\in\nn}\mathbf{1}_{cB_{i,j}}\le A$ for any $i\in\zz$, $\{\lambda_{i,j}\}_{i\in\zz,j\in\nn}:=\{\widetilde
A2^i\|\mathbf{1}_{B_{i,j}}\|_{X}\}_{i\in\zz,j\in\nn}$,
$$
\sup_{i\in\zz}\lf\|\sum_{j\in\nn}
\frac{\lambda_{i,j}\mathbf{1}_{B_{i,j}}}{\|\mathbf{1}_{B_{i,j}}\|_{X}}\r\|_{X}<\infty
$$
and the series
$f:=\sum_{i\in \zz}\sum_{j\in\nn}\lambda_{i,j}m_{i,j}$
converges in $\cs'(\rn)$.
Then $f\in WH_X(\rn)$ and
$$
\|f\|_{WH_X(\rn)}\lesssim\sup_{i\in\zz}\lf\|\sum_{j\in\nn}
\frac{\lambda_{i,j}\mathbf{1}_{B_{i,j}}}{\|\mathbf{1}_{B_{i,j}}\|_{X}}\r\|_{X},
$$
where the implicit positive constant is independent of $f$.
\end{lemma}

\section{Littlewood--Paley function characterizations\label{s3}}

In this section, applying the atomic and the molecular characterizations
of $WH_X(\rn)$ (see \cite[Theorems 4.2 and 4.8]{zwyy}),
we establish various Littlewood--Paley function characterizations of $WH_X(\rn)$.

\begin{definition}\label{cone}
For any $\alpha\in(0,\infty)$ and $x\in\rn$, let
$\Gamma_\alpha(x):=\{(y,t)\in\mathbb{R}_+^{n+1}:\ |x-y|<\alpha t\}$, which is
called the \emph{cone} of aperture $\alpha$
with vertex $x\in\rn$. When $\alpha:=1$, we denote $\Gamma_\alpha(x)$ simply by $\Gamma(x)$.
\end{definition}

In what follows, the symbol $\vec 0_n$ denotes the \emph{origin} of $\rn$ and,
for any $\phi\in\cs(\rn)$, $\widehat\phi$ denotes its \emph{Fourier transform}
which is defined by setting, for any $\xi\in\rn$,
$$
\widehat\phi(\xi):=\int_\rn e^{-2\pi ix\xi}\phi(x)\,dx.
$$
For any $f\in\cs'(\rn)$, $\widehat f$ is defined by setting, for any $\varphi\in\mathcal{S}(\rn)$,
$\la\widehat f,\varphi\ra:=\la f,\widehat\varphi\ra$; also, for any $f\in\mathcal{S}(\rn)$ [or $\mathcal{S}'(\rn)$],
$f^{\vee}$ denotes its \emph{inverse Fourier transform} which is defined by setting,
for any $\xi\in\rn$, $f^{\vee}(\xi):=\widehat{f}(-\xi)$.

Let $\phi\in\mathcal{S}(\rn)$ satisfy $\widehat{\phi}(\vec 0_n)=0$
and, for any $x\in\rn\setminus\{\vec 0_n\}$,
there exists a $t\in (0,\fz)$ such that $\widehat\phi(tx)\not=0$.
From this, we easily deduce that
\begin{equation}\label{Eq62}
\int_\rn\phi(x)\,dx=0.
\end{equation}
Recall that, for any $f\in\cs'(\rn)$,
the \emph{Lusin area function $S(f)$}
and the \emph{Littlewood--Paley $g_\lambda^\ast$-function $g_\lambda^\ast(f)$}
of $f$ with any given $\lambda\in(0,\infty)$ are defined,
respectively, by setting, for any $x\in\rn$,
\begin{equation}\label{eq61}
S(f)(x):=\lf[\int_{\Gamma(x)}|f\ast\phi_t(y)|^2\,\frac{dy\,dt}{t^{n+1}}\r]^{1/2}
\end{equation}
and
\begin{equation}\label{eq62}
g_\lambda^\ast(f)(x):=\lf[\int_0^\infty\int_\rn\lf(\frac{t}{t+|x-y|}\r)^{\lambda n}|f\ast\phi_t(y)|^2\,\frac{dy\,dt}{t^{n+1}}\r]^{1/2},
\end{equation}
where, for any $x\in\rn$,
$\Gamma(x)$ is as in Definition \ref{cone} and, for any $t\in(0,\infty)$ and $x\in\rn$,
$\phi_t(x):=t^{-n}\phi(x/t)$.

Let $\phi\in\mathcal{S}(\rn)$ satisfy $\widehat{\phi}(\vec 0_n)=0$
and, for any $x\in\rn\setminus\{\vec 0_n\}$,
there exists a $j\in \zz$ such that $\widehat\phi(2^{j}x)\not=0$.
Recall that, for any $f\in\cs'(\rn)$, the \emph{Littlewood--Paley $g$-function $g(f)$}
is defined by setting, for any $x\in\rn$,
\begin{equation}\label{eq63}
g(f)(x):=\lf[\int_0^\infty|f\ast\phi_t(x)|^2\,\frac{dt}{t}\r]^{1/2}.
\end{equation}

\begin{remark}\label{wa}
\begin{itemize}
\item[(i)] Observe that $\phi$ appearing in the definition of the Littlewood--Paley $g$-function
satisfies some stronger assumptions than the corresponding assumptions on $\phi$ appearing in the
definitions of both the Lusin area function and the Littlewood--Paley $g_\lambda^\ast$-function.
This is because, when we try to establish the Littlewood--Paley $g$-function characterization of
$WH_X(\rn)$ (see Theorem \ref{Tharea} below), we need to use the discrete
Calder\'on reproducing formula (see Lemma \ref{Le49}(i) below)
which requires that $\phi$ satisfies these stronger assumptions, while, when we establish the Lusin area
function and the Littlewood--Paley $g_\lambda^\ast$-function characterizations of $WH_X(\rn)$
(see, respectively, Theorems \ref{Tharea} and \ref{Thgx}), we only need to use the continuous Calder\'on
reproducing formula (see Lemma \ref{Le47} below) which requires that $\phi$ satisfies only these slightly
weak assumptions. However, in both cases, we did not assume that $\phi$ is radial and
has compact support and hence, compared with the assumptions required in \cite{LYJ,YYYZ},
our assumptions in both cases here are quite weaker.

\item[(ii)] In all these Littlewood--Paley function characterizations of $WH_X(\rn)$,
we only need $\widehat{\phi}(\vec 0_n)=0$, namely, $\phi$ has zero order vanishing moment.
The zero order vanishing moment of $\phi$ is used to guarantee the boundedness of all
Littlewood--Paley functions on $L^q(\rn)$ for any $q\in(1,\infty)$, which is also necessary.
Since we need to use this boundedness of all the Littlewood--Paley functions on $L^q(\rn)$ for
any $q\in(1,\infty)$, in this sense, the assumption on the vanishing moment of $\phi$
appearing in Theorems \ref{Tharea}, \ref{Thgf} and \ref{Thgx} is necessary. Compared with
all the known results on the Littlewood--Paley function characterizations on function
spaces (see, for instance, \cite{LYJ,YYYZ}), this assumption on the vanishing moment of $\phi$ is also minimal.
\end{itemize}
\end{remark}

\subsection{Characterization by the Lusin area function}\label{s3.1}

In this subsection, borrowing some ideas from the proof of \cite[Lemma 4.1]{SHYY},
by first introducing the weak tent space associated with the ball quasi-Banach function space
$X$
(see Definition \ref{Dewt} below),
we then characterize the weak Hardy space $WH_X(\rn)$ by the Lusin-area function.

Now we introduce the notion of weak tent spaces associated with $X$.

Let $\alpha\in(0,\infty)$. For any measurable function
$F:\ \rr_+^{n+1}:=\rn\times(0,\infty)\to\cc$ and
$x\in\rn$, define
\begin{equation}\label{aa}
\ca^{(\alpha)}(F)(x):=\lf[\int_{\Gamma_\alpha(x)}|F(y,t)|^2\,\frac{dy\,dt}{t^{n+1}}\r]^\frac12,
\end{equation}
where $\Gamma_\alpha(x)$ is as in Definition \ref{cone}.
Recall that a measurable function $F$ is said to belong to the \emph{tent space}
$T_2^{p,\alpha}(\rr_+^{n+1})$, with $p\in(0,\infty)$, if $\|F\|_{T_2^{p,\alpha}(\rr_+^{n+1})}:=\|\ca^{(\alpha)}(F)\|_{L^p(\rn)}<\infty$.
For any given ball quasi-Banach function space $X$,
the $X$-\emph{tent space}, $T_X^\alpha(\rr_+^{n+1})$, with aperture $\alpha$,
is defined to be the set of all measurable
functions $F$ such that $\ca^{(\alpha)}(F)\in X$
and naturally equipped with the quasi-norm
$\|F\|_{T_X^{\alpha}(\rr_+^{n+1})}:=\|\ca^{(\alpha)}(F)\|_{X}$.

For any ball $B(x,r)\subset\rn$ with $x\in\rn$ and $r\in(0,\infty)$, let
$$
T_\alpha(B):=\lf\{(y,t)\in\rr^{n+1}_+:\ 0<t<\frac{r}{\alpha},\ |y-x|<r-\alpha t\r\}.
$$
When $\alpha=1$, we denote $T_\alpha(B)$ simply by $T(B)$.

\begin{definition}\label{1q1}
Let $X$ be a ball quasi-Banach function space, $p\in(1,\infty)$ and $\alpha\in(0,\infty)$.
A measurable function $a:\ \rr_+^{n+1}\to\cc$ is called a
\emph{$(T_X,p)$-atom}, associated with a ball, if there
exists a ball $B\subset\rn$ such that
\begin{itemize}
\item[(i)] $\supp(a):=\{(x,t)\in\rr_+^{n+1}:\ a(x,t)\neq0\}\subset T(B)$,
\item[(ii)] $\|a\|_{T_2^{p,1}(\rr_+^{n+1})}\le|B|^{1/p}/\|\mathbf{1}_B\|_{X}$.
\end{itemize}
Moreover, if $a$ is a $(T_X,p)$-atom for any $p\in(1,\infty)$,
then $a$ is called a \emph{$(T_X,\infty)$-atom}.
\end{definition}

The following lemma is a direct corollary of
the definition of $(T_X, p)$-atoms, and we omit the details
here.

\begin{lemma}\label{Lezj}
Let $X$ be a ball quasi-Banach function space and $p\in(1,\infty)$. Then,
for any $(T_X,p)$-atom
$a$ supported in $T(B)$, $\ca^{(1)}(a)$
is supported in $B$ and $\|\ca^{(1)}(a)\|_{L^p(\rn)}
\le | B|^{1/p}\|\mathbf{1}_{ B}\|_X^{-1}$.
\end{lemma}

\begin{definition}\label{Dewt}
Let $\alpha\in(0,\infty)$ and $X$ be a ball quasi-Banach space. The $X$-\emph{weak tent space} $WT_X^{\alpha}(\rr_+^{n+1})$
is defined to be the set of all measurable functions $f$ on $\rr_+^{n+1}$ such that
$$
\|f\|_{WT_X^{\alpha}(\rr_+^{n+1})}:=\|\ca^{(\alpha)}(f)\|_{WX}<\infty.
$$
\end{definition}

To establish the atomic characterization of
$WT_X^1(\rr_+^{n+1})$, we first introduce some notation.
Let $f\in WT_X^1(\rr_+^{n+1})$. For any $i\in\zz$,
let
\begin{equation}\label{ly1}
O_i:=\{x\in\rn:\ \ca^{(1)}(f)(x)>2^i\}\quad\text{and}\quad F_i:=\rn\backslash O_i.
\end{equation}
Moreover, for any given $\gamma\in(0,1)$ and $i\in\zz$, define
\begin{equation}\label{eqq1}
(O_i)^*_\gamma:=\{x\in\rn:\ \cm(\mathbf1_{O_i})(x)>1-\gamma\}\quad\text{and}
\quad(F_i)_\gamma^*:=\rn\backslash (O_i)_\gamma^*,
\end{equation}
where $\cm$ is as in \eqref{mm}.
From Definition \ref{Debqfs}(i), it is easy to see that
$\ca^{(1)}(f)(x)<\infty$ for almost every $x\in\rn$,
which implies that $f\in L_{\loc}^2(\rr_+^{n+1})$.
Therefore, the set of all Lebesgue points of $f$ is almost
equal to $\rr_+^{n+1}$ except for a set of Lebesgue measure zero.
By this and the proof of \cite[Lemma 4.7]{SHYY}, we have the following conclusion.

\begin{lemma}\label{lem47}
Let $f\in WT^1_X(\rr^{n+1}_+)$ and $\gz\in(0,1)$.
Then
\begin{equation*}
\supp(f)\subset\bigcup_{i\in{\mathbb Z}}\widehat{(O_i)_\gz^\ast}\cup E,
\end{equation*}
where, for any $i\in\zz$, $(O_i)_\gz^\ast$ is as in \eqref{eqq1} and $E\subset\rr^{n+1}_+$
satisfies $\int_E\frac{dy\,dt}{t}=0$.
\end{lemma}

For a closed set $F\subset\rn$, denote by $\mathcal R_\alpha(F)$
the union of all cones with aperture $\alpha$ and vertices in $F$, namely,
$$
\mathcal R_\alpha(F):=\bigcup_{x\in F}\Gamma_\alpha(x).
$$
For an open set $O\subset \rn$, define the tent $\widehat O$ over $O$ by
$$
\widehat O:=\{(x,t)\in\rr_+^{n+1}:\ B(x,t)\subset O\}.
$$
It is easy to see that $\widehat O=[\mathcal R(O^\complement)]^\complement$.

\begin{theorem}\label{That}
Let $f:\ \rr_+^{n+1}\to\cc$ be a measurable function.
Let $X$ be a ball quasi-Banach function space
satisfying Assumption \ref{a2.15} for some $p_-\in(0,\infty)$.
Assume that, for any given $r\in(0,\underline{p})$, $X^{1/r}$ is a ball
Banach function space and assume that
there exist $r_0\in(0,\underline{p})$, $p_0\in(r_0,\infty)$ and a positive constant $C$ such that,
for any $f\in(X^{1/r_0})'$,
\begin{equation}\label{Eqdm}
\lf\|\cm^{((p_0/r_0)')}(f)\r\|_{(X^{1/r_0})'}\le C\lf\|f\r\|_{(X^{1/r_0})'},
\end{equation}
where $\cm$ is as in \eqref{mm}.
Then $f\in WT_X^1(\rr_+^{n+1})$ if and only if there exists a
sequence $\{a_{i,j}\}_{i\in\zz,j\in\nn}$ of $(T_X,\infty)$-atoms
associated, respectively, to balls $\{B_{i,j}\}_{i\in\zz,j\in\nn}$ such that
\begin{itemize}
\item[{\rm(i)}] $f:=\sum_{i\in\zz}\sum_{j\in\nn}\lambda_{i,j}a_{i,j}$
almost everywhere on $\rr_+^{n+1}$, where $\lambda_{i,j}:=2^i\|\mathbf{1}_{B_{i,j}}\|_X$;
\item[{\rm(ii)}]
$$
\sup_{i\in\zz}2^i\lf\|\sum_{j\in\nn}\mathbf1_{B_{i,j}}\r\|_X<\infty;
$$
\item[{\rm(iii)}] there exist $c\in(0,1]$ and $M_0\in\nn$ such that,
for any $i\in\zz$, $\sum_{j\in\nn}\mathbf1_{cB_{i,j}}\le M_0$,
where $c$ and $M_0$ are independent of $i\in\zz$, $j\in\nn$ and $f$.
\end{itemize}
Moreover,
$$
\lf\|f\r\|_{WT_X^1(\rr_+^{n+1})}\sim\inf\lf[\sup_{i\in\zz}2^i
\lf\|\sum_{j\in\nn}\mathbf1_{B_{i,j}}\r\|_X\r],
$$
where the infimum is taken over all decompositions of $f$ as above
and the positive equivalence constants are independent of $f$.
\end{theorem}

To show Theorem \ref{That}, we need the
following lemma which was essentially proved by Coifman et al. in \cite[Lemma 2]{CMS}.

\begin{lemma}\label{LeCMS}
Let $\alpha\in(0,\infty)$. Then there exist positive constants
$\gamma\in(0,1)$ and $C_{(\alpha,\gamma)}$ such that, for any closed subset
$F$ of $\rn$ and any  non-negative measurable function $H$ on $\rr_+^{n+1}$,
$$
\int_{\mathcal R_\alpha(F_\gamma^*)}H(y,t)t^n\,dy\,dt\le C_{(\alpha,\gamma)}
\int_F\lf\{\int_{\Gamma(x)}H(y,t)\,dy\,dt\r\}\,dx,
$$
where $F_\gamma^*:=\{x\in\rn:\ \cm(\mathbf1_{F^\complement})(x)>1-\gamma\}^\complement$
with $\cm$ as in \eqref{mm}.
\end{lemma}
\begin{remark}
In Coifman et al. \cite[Lemma 2]{CMS}, the set $F_\gamma^*$ is defined to be the
\emph{set of all points of the global $\gamma$-density with
respect to $F$} (see, \cite[p.\, 310]{CMS}) and
then Coifman et al. \cite[Lemma 2]{CMS} proved
$$F_\gamma^*=\{x\in\rn:\
\cm(\mathbf1_{F^\complement})(x)>1-\gamma\}^\complement,$$
under the additional assumption that
the measure of $F^\complement$ is finite. Then, using the fact that $F_\gamma^*=\{x\in\rn:\ \cm(\mathbf1_{F^\complement})(x)>1-\gamma\}^\complement$, they obtain the estimate of Lemma
\ref{LeCMS}.
If we just let $F_\gamma^*:=\{x\in\rn:\
\cm(\mathbf1_{F^\complement})(x)>1-\gamma\}^\complement$,
then, similarly to the proof of \cite[Lemma 2]{CMS},
we can indeed prove Lemma \ref{LeCMS} without the additional assumption
$|F^\complement|<\infty$; we omit the details here.
\end{remark}

\begin{proof}[Proof of Theorem \ref{That}]
We first show the necessity. Let $f\in WT_X^1(\rr_+^{n+1})$.
By \cite[Lemma 2.14]{SHYY}, we know that $1\notin X$.
For any $i\in\zz$ and
$f\in WT_X^1(\rr_+^{n+1})$, from this and the fact that
$\|{\mathbf 1}_{O_i}\|_X=\|{\mathbf 1}_{O_i}\|_{WX}$,
it follows that $O_i$ is a proper subset of $\rn$. For any given $\gamma\in(0,1)$
and any $i\in\zz$,
since $(O_i)_\gamma^*$ is open, from the well-known Whitney decomposition theorem
(see, for instance \cite[p.\,463]{G1}), we deduce that
there exists a family of cubes, $\{Q_{i,j}\}_{j\in\nn}$, with disjoint interiors
such that
$(O_i)_\gamma^*:=\bigcup_{j\in\nn}Q_{i,j}$,
$$
\sqrt{n}l_{Q_{i,j}}\le\mathrm{dist}\lf(Q_{i,j},[(O_i)_\gamma^*]
^\complement\r):=\inf_{\{x\in Q_{i,j},y\in[(O_i)_\gamma^*]
^\complement\}}|x-y|\le4\sqrt{n}l_{Q_{i,j}},
$$
where $l_{Q_{i,j}}$ denotes the side length of the cube $Q_{i,j}$,
and, for any given $j\in\nn$, there exist at most $12^n$
different cubes $\{Q_{i,k}\}_k$ that touch $Q_{i,j}$.

For any $i\in\zz$ and $j\in\nn$, let $B_{i,j}$
be the ball with the same center as $Q_{i,j}$ and with
radius $\frac{11\sqrt{n}}{2}l(Q_{i,j})$.
Moreover, let
$$
A_{i,j}:=\widehat{B_{i,j}}\cap[Q_{i,j}\times
(0,\infty)]\cap\widehat{(O_i)_\gamma^*}\cap[\widehat{(O_{i+1})_\gamma^*}]^\complement,
$$
$$
a_{i,j}:=2^{-i}\|{\mathbf 1}_{B_{i,j}}\|_X^{-1}f{\mathbf 1}_{A_{i,j}}\quad\text{and}
\quad\lambda_{i,j}:=2^i\|{\mathbf 1}_{B_{i,j}}\|_X.
$$
By Lemma \ref{lem47}, we find that
$$
\supp(f)\subset\bigcup_{i\in\zz}\widehat{(O_i)_\gamma^*}\cup E,
$$
where $E\subset\rr_+^{n+1}$ satisfies $\int_E\frac{dy\,dt}{t}=0$.
Then, from this and the fact that $\{[Q_{i,j}\times(0,\infty)]\cap\widehat{(O_i)_\gamma^*}
\cap[\widehat{(O_{i+1})_\gamma^*}]^\complement\}\subset \widehat{B_{i,j}}$, it follows that
\begin{equation}\label{Eqft}
f=\sum_{i\in\zz}\sum_{j\in\nn}\lambda_{i,j}a_{i,j}
\end{equation}
almost everywhere on $\rr_+^{n+1}$.

Next we show that, for any $i\in\zz$ and $j\in\nn$, $a_{i,j}$ is
a harmless constant multiple of a $(T_X,\infty)$-atom supported in $\widehat{B_{i,j}}$.
Let $p\in(1,\infty)$ and $h\in T_2^{p'}(\rr_+^{n+1})$
with $\|h\|_{T_2^{p'}(\rr_+^{n+1})}\le1$.
Since $A_{i,j}\subset[\widehat{(O_{i+1})_\gamma^*}]^\complement=\mathcal R_1([(O_{i+1})_\gamma^*)]^\complement)$, from
Lemma \ref{LeCMS}, we deduce that there exists a $\gamma\in(0,1)$
sufficiently close to $1$ such that
\begin{align*}
\lf|\langle a_{i,j},h\rangle\r|:&=\lf|\int_{\rr_+^{n+1}}
a_{i,j}(y,t){\mathbf 1}_{A_{i,j}}(y,t)h(y,t)\,\frac{dy\,dt}{t}\r|
\le\int_{\mathcal R_1([(O_{i+1})_\gamma^*)]^\complement)}
\lf|a_{i,j}(y,t)h(y,t)\r|\,\frac{dy\,dt}{t}\\
&\lesssim\int_{(O_{i+1})^\complement}\int_{\Gamma(x)}
\lf|a_{i,j}(y,t)h(y,t)\r|\,\frac{dy\,dt}{t^{n+1}}\,dx.
\end{align*}
Then, by the H\"older inequality, Lemma \ref{Lezj}
and the fact that $\ca^{(1)}(f)(x)\le2^{i+1}$ for any
$x\in(O_{i+1})^\complement$, we further obtain
\begin{align*}
\lf|\langle a_{i,j},h\rangle\r|&\lesssim
\int_{(O_{i+1})^\complement}\ca^{(1)}(a_{i,j})(x)\ca^{(1)}(h)(x)\,dx\\
&\lesssim2^{-i}\|\mathbf 1_{B_{i,j}}\|_X^{-1}\lf\{\int_{B_{i,j}\cap(O_{i+1})
^\complement}[\ca^{(1)}(f)(x)]^p\,dx\r\}^\frac1p\|h\|_{T_2^{p'}(\rr_+^{n+1})}\\
&\lesssim|B_{i,j}|^\frac1p\|\mathbf 1_{B_{i,j}}\|_X^{-1}.
\end{align*}
From this and $(T_2^p(\rr_+^{n+1}))'=T_2^{p'}(\rr_+^{n+1})$ (see \cite[Theorem 2]{CMS}),
where $(T_2^p(\rr_+^{n+1}))'$ denotes the \emph{dual
space} of $T_2^p(\rr_+^{n+1})$, it follows that
$$
\|a_{i,j}\|_{T_2^p(\rr_+^{n+1})}\lesssim|B_{i,j}|^{1/p}\|\mathbf1_{B_{i,j}}\|_X^{-1}.
$$
Thus, $a_{i,j}$ is a constant multiple of a $(T_X,p)$-atom
supported in $\widehat{B_{i,j}}$
for any given $p\in(1,\infty)$,
and hence a constant multiple of a $(T_X,\infty)$-atom,
which, together with \eqref{Eqft}, implies (i) of the necessity.

To show (ii), let $\theta\in(0,p_-)$. For any $i\in\zz$, we know that
\begin{align*}
\sum_{j\in\nn}\mathbf1_{Q_{i,j}}\lesssim\mathbf1_{(O_i)_\gamma^*}
\sim\mathbf1_{\{x\in\rn:\ \cm(\mathbf1_{O_i})(x)>1-\gamma\}}
\sim\mathbf1_{\{x\in\rn:\ \cm^{(\theta)}(\mathbf1_{O_i})(x)>\sqrt[\theta]{1-\gamma}\}},
\end{align*}
which, combined with \eqref{EqHLMS}, further implies that
\begin{align}\label{lyy}
\lf\|\sum_{j\in\nn}\mathbf1_{B_{i,j}}\r\|_X\lesssim\lf\|\mathbf1_{\{x\in\rn:\ \cm^{(\theta)}(\mathbf1_{O_i})(x)>\sqrt[\theta]{1-\gamma}\}}\r\|_X
\lesssim\lf\|\cm^{(\theta)}(\mathbf1_{O_i})\r\|_X
\sim\lf\|\cm(\mathbf1_{O_i})\r\|_{X^{1/\theta}}^{1/\theta}\lesssim\lf\|\mathbf1_{O_i}\r\|_X.
\end{align}
Thus, by the definition of $WX$, we obtain
$$
\sup_{i\in\zz}2^i\lf\|\sum_{j\in\nn}
\mathbf1_{B_{i,j}}\r\|_X\lesssim\sup_{i\in\zz}2^i\lf\|\mathbf1_{O_i}\r\|_X
\lesssim\lf\|\ca^{(1)}(f)\r\|_{WX}\sim\lf\|f\r\|_{WT_X^1(\rr_+^{n+1})},
$$
which completes the proof of (ii) of the necessity.

The conclusion of (iii) of the necessity
is a direct consequence of the property of $\{Q_{i,j}\}_{i\in\zz,j\in\nn}$,
which completes the proof of the necessity.

Now, we show the sufficiency. Assume that
there exist $\{\lambda_{i,j}\}_{i\in\zz,j\in\nn}\subset[0,\infty)$ and
a sequence $\{a_{i,j}\}_{i\in\zz,j\in\nn}$
of $(T_X,\infty)$-atoms associated, respectively, to balls $\{B_{i,j}\}_{i\in\zz,j\in\nn}$
such that (i), (ii) and (iii) of Theorem \ref{That} hold true.
To prove $f\in WT_X^1(\rr_+^{n+1})$, by the
definition of $WT_X^1(\rr_+^{n+1})$, it suffices to show that
\begin{equation}\label{Eqwt}
\sup_{\alpha\in(0,\infty)}\lf\{
\alpha\lf\|\mathbf1_{\{x\in\rn:\ \ca^{(1)}(f)(x)>\alpha\}}\r\|_X\r\}\lesssim \sup_{i\in\zz}2^i\lf\|\sum_{j\in\nn}\mathbf1_{B_{i,j}}\r\|_X.
\end{equation}
For any given $\alpha\in(0,\infty)$, let $i_0\in\zz$
be such that $2^{i_0}\le\alpha<2^{i_0+1}$. Then we write
$$
f=\sum_{i=-\infty}^{i_0-1}\sum_{j\in\nn}\lambda_{i,j}a_{i,j}
+\sum_{i=i_0}^{\infty}\sum_{j\in\nn}\cdots=:f_1+f_2.
$$
By Definition \ref{Debqfs}(ii), we obtain
\begin{align*}
\lf\|\mathbf{1}_{\{x\in\rn:\ \ca^{(1)}(f)(x)>\alpha\}}\r\|_X&\lesssim
\lf\|\mathbf{1}_{\{x\in\rn:\ \ca^{(1)}(f_1)(x)>\alpha/2\}}\r\|_X+
\lf\|\mathbf{1}_{\{x\in\rn:\ \ca^{(1)}(f_2)(x)>\alpha/2\}}\r\|_X\\\noz
&=:\mathrm{I_1}+\mathrm{I_2}.
\end{align*}

We first estimate $\mathrm{I_{1}}$. Let $\widetilde q\in(\max\{1/p_0,1\},1/{r_0}]$
and $a\in(0,1-1/{\widetilde q})$. Then, from the H\"older inequality, we deduce that
\begin{align*}
&\sum_{i=-\infty}^{i_0-1}\sum_{j\in\nn}\lambda_{i,j}\ca^{(1)}(a_{i,j})
\le\frac{2^{i_0a}}{(2^{a\widetilde q'}-1)^{1/\widetilde q'}}
\lf\{\sum_{i=-\infty}^{i_0-1}2^{-ia\widetilde q}\lf[\sum_{j\in\nn}
\lambda_{i,j}\ca^{(1)}(a_{i,j})\r]^{\widetilde q}\r\}^{1/\widetilde q},
\end{align*}
where $\widetilde q':={\widetilde q}/{(\widetilde q-1)}$. By this, Definition
 \ref{Debf}(i), $\widetilde qr_0\in(0,1]$, Lemma \ref{Le65}
and the assumption that $X^{1/r_0}$ is a ball Banach function space,
we conclude that
\begin{align*}
\mathrm{I_{1}}
&\lesssim\lf\|\mathbf{1}_{\{x\in\rn:\ 2^{i_0a}
\{\sum_{i=-\infty}^{i_0-1}2^{-ia\widetilde q}[\sum_{j\in\nn}
\lambda_{i,j}\ca^{(1)}(a_{i,j})(x)]^
{\widetilde q}\}^{1/\widetilde q}>2^{i_0-2}\}}\r\|_{X}\\
&\lesssim2^{-i_0\widetilde q(1-a)}\lf\|\sum_{i=-\infty}^
{i_0-1}2^{-ia\widetilde q}\lf[\sum_{j\in\nn}
\lambda_{i,j}\ca^{(1)}(a_{i,j})\r]^{\widetilde q}\r\|_{X}\\
&\lesssim2^{-i_0\widetilde q(1-a)}
\lf\|\sum_{i=-\infty}^{i_0-1}2^{(1-a)i\widetilde qr_0}\sum_{j\in\nn}\lf[
\lf\|\mathbf{1}_{B_{i,j}}\r\|_{X}\ca^{(1)}(a_{i,j})\r]
^{\widetilde qr_0}\r\|_{X^{1/r_0}}^\frac{1}{r_0}\\
&\lesssim2^{-i_0\widetilde q(1-a)}\lf[\sum_{i=-\infty}^{i_0-1}2^{(1-a)i\widetilde qr_0}
\lf\|\lf\{\sum_{j\in\nn}\lf[
\lf\|\mathbf{1}_{B_{i,j}}\r\|_{X}\ca^{(1)}(a_{i,j})\r]^{\widetilde qr_0}\r\}^\frac1{r_0}\r\|_{X}^{r_0}\r]^\frac{1}{r_0}.
\end{align*}

Let $q:=p_0\widetilde q\in(1,\infty)$. From Lemma \ref{Lezj} and the fact that
$\{a_{i,j}\}_{i\in\zz,j\in\nn}$ is a sequence
of $(T_X,\infty)$-atoms associated, respectively, to balls $\{B_{i,j}\}_{i\in\zz,j\in\nn}$,
we deduce that, for any $i\in\zz$ and $j\in\nn$, $\supp(\ca^{(1)}(a_{i,j}))\subset B_{i,j}$ and $$\lf\|\ca^{(1)}(a_{i,j})\r\|_{L^q(\rn)}\le|B_{i,j}|^{1/q}
\|\mathbf1_{B_{i,j}}\|_X^{-1}.$$
Then, by this, we conclude that, for any $i\in\zz$ and $j\in\nn$,
$$
\lf\|\lf[\lf\|\mathbf1_{B_{i,j}}\r\|_X\ca^{(1)}(a_{i,j})\r]^{\widetilde q}\r\|_{L^{p_0}(\rn)}
\lesssim\lf\|\mathbf1_{B_{i,j}}\r\|_X^{\widetilde q}
\lf\|\ca^{(1)}(a_{i,j})\r\|_{L^q(\rn)}^{\widetilde q}\lesssim|B_{i,j}|^{1/p_0},
$$
which, combined with Lemma \ref{Le45}, \eqref{EqHLMS} and $(1-a)\widetilde q>1$, further implies that
\begin{align*}
\mathrm{I_{1}}&\lesssim2^{-i_0\widetilde q(1-a)}
\lf[\sum_{i=-\infty}^{i_0-1}2^{(1-a)i\widetilde qr_0}\lf\|\lf(\sum_{j\in\nn}\mathbf{1}_{B_{i,j}}\r)^\frac1{r_0}\r\|_{X}^{r_0}\r]^{\frac1{r_0}}
\lesssim2^{-i_0\widetilde q(1-a)}
\lf[\sum_{i=-\infty}^{i_0-1}2^{(1-a)i\widetilde qr_0}\lf\|\lf(\sum_{j\in\nn}\mathbf{1}_{cB_{i,j}}\r)^\frac1{r_0}\r\|_{X}^{r_0}\r]^{\frac1{r_0}}\\
&\lesssim2^{-i_0\widetilde q(1-a)}
\lf[\sum_{i=-\infty}^{i_0-1}2^{[(1-a)\widetilde
q-1]ir_0}\r]^{\frac1{r_0}}\sup_{i\in\zz}2^i\lf\|\sum_{j\in\nn}
\mathbf{1}_{B_{i,j}}\r\|_{X}\lesssim\alpha^{-1}
\sup_{i\in\zz}2^i\lf\|\sum_{j\in\nn}\mathbf{1}_{B_{i,j}}\r\|_{X}.
\end{align*}
This shows that
\begin{equation*}
\alpha\mathrm{I_{1}}\lesssim\sup_{i\in\zz}2^i\lf\|\sum_{j\in\nn}\mathbf{1}_{B_{i,j}}\r\|_{X}.
\end{equation*}

Next we deal with $\mathrm{I_2}$. Let $r_2\in(0,\underline p)$.
Then, by \eqref{EqHLMS}, Definition \ref{Debf}(i),
the assumption that $X^{1/r_2}$ is a ball Banach
function space and $\sum_{j\in\nn}\mathbf{1}_{cB_{i,j}}\le A$,
we conclude that
\begin{align*}
\mathrm{I_2}&\lesssim
\lf\|\sum_{i=i_0}^\infty\sum_{j\in\nn}\mathbf{1}_{B_{i,j}}\r\|_{X}
\lesssim\lf\|\sum_{i=i_0}^\infty\sum_{j\in\nn}\mathbf{1}_{cB_{i,j}}\r\|_{X}
\sim\lf\|\lf\{\sum_{i=i_0}^\infty\sum_{j\in\nn}
\mathbf{1}_{cB_{i,j}}\r\}^{r_2}\r\|_{X^\frac1{r_2}}^\frac1{r_2}\\
&\lesssim\lf[\sum_{i=i_0}^\infty\lf\|\sum_{j\in\nn}
\mathbf{1}_{cB_{i,j}}\r\|_{X^\frac1{r_2}}\r]^\frac1{r_2}
\lesssim\lf[\sum_{i=i_0}^\infty\lf\|\sum_{j\in\nn}
\mathbf{1}_{cB_{i,j}}\r\|_{X}^{r_2}\r]^\frac1{r_2}
\lesssim\lf\{\sum_{i=i_0}^\infty2^{-ir_2}\lf[2^i
\lf\|\sum_{j\in\nn}\mathbf{1}_{B_{i,j}}\r\|_{X}\r]^{r_2}\r\}^{\frac1{r_2}}\\
&\lesssim\sup_{i\in\nn}2^i\lf\|\sum_{j\in\nn}
\mathbf{1}_{B_{i,j}}\r\|_{X}\lf(\sum_{i=i_0}^\infty2^{-ir_2}\r)^\frac1{r_2}
\lesssim\alpha^{-1}\sup_{i\in\zz}2^i\lf\|\sum_{j\in\nn}\mathbf{1}_{B_{i,j}}\r\|_{X},
\end{align*}
which implies that
\begin{equation*}
\alpha\mathrm{I_2}\lesssim\sup_{i\in\zz}2^i\lf\|\sum_{j\in\nn}\mathbf{1}_{B_{i,j}}\r\|_{X}.
\end{equation*}

Combining the estimates for $\mathrm{I_1}$ and
$\mathrm{I_2}$, we obtain \eqref{Eqwt} and hence complete the proof of Theorem \ref{That}.
\end{proof}

\begin{remark}\label{buji}
Let the sequence $\{a_{i,j}\}_{i\in\zz,j\in\nn}$ be as in the proof of Theorem \ref{That}. Then
we claim that
$\{\supp(a_{i,j})\}_{i\in\zz,j\in\nn}$
have pairwise disjoint interior. Indeed, let
$\{A_{i,j}\}_{i\in\zz,j\in\nn}$, $\{Q_{i,j}\}_{i\in\zz,j\in\nn}$ and
$\{\widehat{(O_i)_\gamma^*}\}_{i\in\zz}$
be as in the proof of Theorem \ref{That}. Then,
by the definition of the set $\{A_{i,j}\}_{i\in\zz,j\in\nn}$ , the fact that
$\{Q_{i,j}\}_{j\in\nn}$ for any given $i\in\zz$, and $\{\widehat{(O_i)_\gamma^*}
\cap[\widehat{(O_{i+1})_\gamma^*}]^\complement\}\}_{i\in\zz}$
have pairwise disjoint interior, we conclude that the collection of sets, $\{A_{i,j}\}_{i\in\zz,j\in\nn}$,
is pairwise disjoint, up to sets of measure zero. From this and the definitions of
$\{a_{i,j}\}_{i\in\zz,j\in\nn}$, it follows that this claim holds true.
\end{remark}

In what follows, we use the \emph{symbol $\epsilon\to0^+$} to denote
$\epsilon\in(0,\infty)$ and $\epsilon\to0$, and we also use
the \emph{symbol} $C_c^\fz(\rn)$ to denote the set of all infinitely differentiable
functions on $\rn$ with compact supports.

Combining Calder\'on \cite[Lemma 4.1]{C1975} and Folland and Stein \cite[Theorem 1.64]{FoS}
(see also \cite[p.\,219]{C1977} and \cite[Lemma 4.6]{YYYZ}), we immediately obtain
the following Calder\'on reproducing formula; see also \cite[Lemma 4.4]{zwyy}.

\begin{lemma}\label{Le47}
Let $\phi$ be a Schwartz function and, for any $x\in\rn\setminus\{\vec 0_n\}$,
there exists a $t\in (0,\fz)$ such that $\widehat\phi(tx)\not=0$. Then there
exists a $\psi\in\cs(\rn)$ such that $\wh\psi\in C^\fz_c(\rn)$ with its support
away from $\vec 0_n$, $\wh\phi\wh\psi\ge 0$ and, for any $x\in\rn\setminus\{\vec 0_n\}$,
$$\int^\fz_0\wh\phi(tx)\wh\psi(tx)\,\frac {dt}t=1.$$
Moreover, for any $f\in\cs'(\rn)$, if $f$ vanishes weakly at infinity,
then
$$
f=\int_0^\infty f\ast\phi_t\ast\psi_t\,\frac{dt}{t}\quad\text{in}\quad\cs'(\rn),
$$
namely,
$$
f=\lim_{\substack{\epsilon\to0^+\\ A\to\infty}}
\int_\epsilon^A f\ast\phi_t\ast\psi_t\,\frac{dt}{t}\quad\text{in}\quad\cs'(\rn).
$$
\end{lemma}

To establish the Lusin area function characterization of $WH_X(\rn)$,
we also need the following notion of absolutely continuous quasi-norms.

\begin{definition}\label{abo}
A ball quasi-Banach function space $X$ is said to have an \emph{absolutely continuous quasi-norm} if
$\|\mathbf{1}_{E_j}\|_{X}\downarrow0$ as $j\to\infty$ whenever $\{E_j\}_{j=1}^\infty$ is a sequence of measurable sets
in $\rn$ satisfying that $E_j\supset E_{j+1}$ for any $j\in\mathbb{N}$, $\mathbf{1}_{E_1}\in X$
and $\cap_{j=1}^\infty E_j=\emptyset$.
\end{definition}

\begin{theorem}\label{Tharea}
Let $X$ be a ball quasi-Banach function space having an absolutely continuous quasi-norm
and satisfying
Assumption \ref{a2.15} for some $p_-\in(0,\infty)$.
Assume that there exist both $\vartheta_0\in(1,\infty)$ such that $X$ is $\vartheta_0$-concave
and $\gamma_0\in(0,\infty)$ such that $\cm$ in \eqref{mm} is bounded on $(WX)^{1/\gamma_0}$.
Assume that, for any $r\in(0,\underline{p})$
with $\underline{p}$ as in \eqref{Eqpll}, $X^{1/r}$ is a ball Banach function space and
there exists a $p_+\in[p_-,\infty)$ such that, for any given  $r\in(0,\underline{p})$
and $p\in(p_+,\infty)$, and for any $f\in(X^{1/r})'$,
\begin{equation*}
\lf\|\cm^{((p/r)')}(f)\r\|_{(X^{1/r})'}\lesssim\lf\|f\r\|_{(X^{1/r})'},
\end{equation*}
where the implicit positive constant is independent of $f$.
Then $f\in WH_X(\rn)$ if and only if $f\in\cs'(\rn)$,
$f$ vanishes weakly at infinity and $S(f)\in WX$, where $S(f)$ is as in \eqref{eq61}. Moreover,
there exists a positive constant $C$ such that,
for any $f\in WH_X(\rn)$,
$$
C^{-1}\lf\|S(f)\r\|_{WX}\le\|f\|_{WH_X(\rn)}\le C\lf\|S(f)\r\|_{WX}.
$$
\end{theorem}

To show Theorem \ref{Tharea}, we still need to recall some necessary notions and
lemmas. Recall that $f\in\cs'(\rn)$ is said to \emph{vanish weakly at infinity} if, for any $\phi\in\cs(\rn)$,
$f\ast\phi_t\to0$ in $\cs'(\rn)$ as $t\to\infty$ (see, for instance, \cite[p.\,50]{FoS}).
The following lemma is just \cite[Lemma 4.3]{zwyy}.

\begin{lemma}\label{Le64}
Let $X$ be a ball quasi-Banach function space. If $f\in WH_X(\rn)$,
then $f$ vanishes weakly at infinity.
\end{lemma}

\begin{lemma}\label{lyyy}
Let $X$ be a ball quasi-Banach function space. Assume that
$\{f_j\}_{j\in\nn}\subset WH_X(\rn)$ and
$f\in WH_X(\rn)$ satisfy $\lim_{j\to\infty}\|f_j-f\|_{WH_X(\rn)}=0$.
Then, for any $\varphi\in\cs(\rn)$,
$$
\lim_{j\to\infty}\la f_j-f,\varphi\ra=0.
$$
\end{lemma}

\begin{proof}
By \cite[Proposition 3.10]{Bo}, we know that, for any given $\varphi\in\cs(\rn)$
and for any $f\in WH_X(\rn)$,
$t\in(0,\infty)$, $x\in\rn$ and $y\in B(x,t)$, $|f\ast\varphi_t(x)|\lesssim M_N^0(f)(y)$,
where $M_N^0(f)$ is as in \eqref{EqMN0} and $N\in\nn$.
From this and Definition \ref{Debqfs}(ii), it follows that, if $N$ is sufficiently large as in Definition \ref{DewSH},
then, for any given $\varphi\in\cs(\rn)$,
\begin{align*}
\lf|\la f_j-f,\varphi\ra\r|&=\lf|(f_j-f)\ast[\varphi(-\cdot)](\vec 0_n)\r|\lesssim\inf_{y\in B(0,1)}M_N^0(f_j-f)(y)\lesssim
\frac{\|\mathbf{1}_{B(\vec 0_n,1)}M_N^0(f_j-f)\|_{WX}}{\|\mathbf{1}_{B(\vec 0_n,1)}\|_{WX}}\\
&\lesssim \frac{\|f_j-f\|_{WH_X(\rn)}}{\|\mathbf{1}_{B(\vec 0_n,1)}\|_{WX}}
\to0\quad\text{as}\quad j\to\infty,
\end{align*}
which further implies the desired conclusion of the lemma and
hence completes the proof of Lemma \ref{lyyy}.
\end{proof}

The following inequality of the Lusin area function on the classical Lebesgue space is well known,
whose proof can be found, for instance, in \cite[Chapter 7]{FoS}.

\begin{lemma}\label{Le65}
Let $q\in(1,\infty)$. Then there exists a positive constant $C$ such that, for any $f\in L^q(\rn)$,
$$
\|S(f)\|_{L^q(\rn)}\le C\|f\|_{L^q(\rn)},
$$
where $S(f)$ is as in \eqref{eq61}.
\end{lemma}

In what follows, for any $f\in\cs'(\rn)$, we use the \emph{symbol} $\widetilde{\supp}(f)$
to denote the intersection of all closed sets $K$
satisfying that, for any $\varphi\in C_c^\infty(\rn)$ with
$\supp(\varphi):=\{x\in\rn:\ \varphi(x)\neq0\}\subset(\rn\setminus K)$,
$\langle f,\varphi\rangle=0$.

\begin{proof}[Proof of Theorem \ref{Tharea}]
Let $\vartheta_0$, $\underline{p}$ and $p_-$ be as in
this theorem. We first prove the necessity. Let $f\in WH_X(\rn)$ and
$d\geq\lfloor n(1/\min\{\frac{p_-}{\vartheta_0},\underline{p}\}-1)\rfloor$.
Then, by Lemma \ref{Le64}, we know that $f$ vanishes weakly at infinity.
From $d\geq\lfloor n(\vartheta_0/p_--1)\rfloor$ and the fact that
$X$ satisfies Assumption \ref{a2.15} for some $p_-\in(0,\infty)$, we deduce that there exists
a $p\in(0,p_-)$ such that $\cm$ in \eqref{mm} is bounded on
$X^{1/(\vartheta_0 p)}$ and $d\geq \lfloor n(1/p-1)\rfloor$.
By this and Lemma \ref{Thad}, we know that there exists a sequence $\{a_{i,j}\}_{i\in\zz,j\in\nn}$ of
$(X,\,\infty,\,d)$-atoms supported, respectively, in balls $\{B_{i,j}\}_{i\in\zz,j\in\nn}$
satisfying that,
for any $i\in\zz$,
$\sum_{j\in\nn}\mathbf{1}_{cB_{i,j}}\le A$ with $c\in(0,1]$ and $A$ being
a positive constant independent of $f$ and $i$
such that $f=\sum_{i\in\zz}\sum_{j\in\nn}\lambda_{i,j}a_{i,j}$ in $\mathcal{S}'(\rn)$,
where $\lambda_{i,j}:=\widetilde A2^i\|\mathbf{1}_{B_{i,j}}\|_{X}$ for any $i\in\zz$ and $j\in\nn$
with $\widetilde A$ being a positive constant independent of $f$, $i\in\zz$ and $j\in\nn$, and
$$
\sup_{i\in\zz}\lf\|\sum_{j\in\nn}
\frac{\lambda_{i,j}\mathbf{1}_{B_{i,j}}}{\|\mathbf{1}_{B_{i,j}}\|_{X}}\r\|_{X}\lesssim\|f\|_{WH_X(\rn)},
$$
where the implicit positive constant is independent of $f$. For any $\alpha\in(0,\infty)$,
let $i_0\in\zz$ be such that $2^{i_0}\le\alpha<2^{i_0+1}$.
Then we decompose $f$ into
$$
f=\sum_{i=-\infty}^{i_0-1}\sum_{j\in\nn}\lambda_{i,j}a_{i,j}
+\sum_{i=i_0}^{\infty}\sum_{j\in\nn}\cdots
=:f_1+f_2.
$$
To prove $S(f)\in WX$, it suffices to show that, for any $\alpha\in(0,\infty)$,
$$
\alpha\lf\|\mathbf{1}_{\{x\in\rn:\ S(f)(x)>\alpha\}}\r\|_X\lesssim\sup_{i\in\zz}\lf\|\sum_{j\in\nn}
\frac{\lambda_{i,j}\mathbf{1}_{B_{i,j}}}{\|\mathbf{1}_{B_{i,j}}\|_{X}}\r\|_{X}.
$$
By Definition \ref{Debqfs}(ii), we obtain
\begin{align}\label{EqI6123}
\lf\|\mathbf{1}_{\{x\in\rn:\ S(f)(x)>\alpha\}}\r\|_X&\lesssim
\lf\|\mathbf{1}_{\{x\in\rn:\ S(f_1)(x)>\alpha/2\}}\r\|_X+
\lf\|\mathbf{1}_{\{x\in A_{i_0}:\ S(f_2)(x)>\alpha/2\}}\r\|_X\\\noz
&\quad+\lf\|\mathbf{1}_{\{x\in (A_{i_0})^\complement:\ S(f_2)(x)>\alpha/2\}}\r\|_X\\\noz
&=:\mathrm{I_1}+\mathrm{I_2}+\mathrm{I_3},
\end{align}
where $A_0:=\cup_{i=i_0}^\infty\cup_{j\in\nn}(4B_{i,j})$.

It is easy to see that
\begin{align}\label{Eq6I1}
\mathrm{I_1}&\lesssim\lf\|\mathbf{1}_{\{x\in\rn:\
\sum_{i=-\infty}^{i_0-1}\sum_{j\in\nn}\lambda_{i,j}
S(a_{i,j})(x)\mathbf{1}_{4B_{i,j}}(x)>\frac\alpha4\}}\r\|_{X}
+\lf\|\mathbf{1}_{\{x\in\rn:\ \sum_{i=-\infty}^{i_0-1}\sum_{j\in\nn}
\lambda_{i,j}S(a_{i,j})(x)\mathbf{1}_{(4B_{i,j})^\complement}(x)>\frac\alpha4\}}\r\|_{X}\\\noz
&=:\mathrm{I_{1,1}}+\mathrm{I_{1,2}}.
\end{align}

We first estimate $\mathrm{I_{1,1}}$. Let $r_0\in(0,\underline p)$,
$p_0\in(p_+,\infty)$, $\widetilde q\in(\max\{1/p_0,1\},1/{r_0}]$
and $a\in(0,1-1/{\widetilde q})$. Then, by the H\"older inequality, we obtain
\begin{align*}
&\sum_{i=-\infty}^{i_0-1}\sum_{j\in\nn}\lambda_{i,j}S(a_{i,j})\mathbf{1}_{4B_{i,j}}
\le\frac{2^{i_0a}}{(2^{a\widetilde q'}-1)^{1/\widetilde q'}}
\lf\{\sum_{i=-\infty}^{i_0-1}2^{-ia\widetilde q}\lf[\sum_{j\in\nn}
\lambda_{i,j}S(a_{i,j})\mathbf{1}_{4B_{i,j}}\r]^{\widetilde q}\r\}^{1/\widetilde q},
\end{align*}
where $\widetilde q':={\widetilde q}/{(\widetilde q-1)}$. From this, Definition
 \ref{Debf}(i), $\widetilde qr_0\in(0,1]$, Lemma \ref{Le65}
and the assumption that $X^{1/r_0}$ is a ball Banach function space,
we deduce that
\begin{align*}
\mathrm{I_{1,1}}
&\lesssim\lf\|\mathbf{1}_{\{x\in\rn:\ 2^{i_0a}
\{\sum_{i=-\infty}^{i_0-1}2^{-ia\widetilde q}[\sum_{j\in\nn}
\lambda_{i,j}S(a_{i,j})(x)\mathbf{1}_{4B_{i,j}}(x)]^
{\widetilde q}\}^{1/\widetilde q}>2^{i_0-2}\}}\r\|_{X}\\
&\lesssim2^{-i_0\widetilde q(1-a)}\lf\|\sum_{i=-\infty}^
{i_0-1}2^{-ia\widetilde q}\lf[\sum_{j\in\nn}
\lambda_{i,j}S(a_{i,j})\mathbf{1}_{4B_{i,j}}\r]^{\widetilde q}\r\|_{X}\\
&\lesssim2^{-i_0\widetilde q(1-a)}
\lf\|\sum_{i=-\infty}^{i_0-1}2^{(1-a)i\widetilde qr_0}\sum_{j\in\nn}\lf[
\lf\|\mathbf{1}_{B_{i,j}}\r\|_{X}S(a_{i,j})\mathbf{1}_{4B_{i,j}}\r]
^{\widetilde qr_0}\r\|_{X^{1/r_0}}^\frac{1}{r_0}\\
&\lesssim2^{-i_0\widetilde q(1-a)}\lf[\sum_{i=-\infty}^{i_0-1}2^{(1-a)i\widetilde qr_0}
\lf\|\lf\{\sum_{j\in\nn}\lf[
\lf\|\mathbf{1}_{B_{i,j}}\r\|_{X}S(a_{i,j})\mathbf{1}_{4B_{i,j}}\r]^{\widetilde qr_0}\r\}^\frac1{r_0}\r\|_{X}^{r_0}\r]^\frac{1}{r_0}.
\end{align*}
Let $q:=p_0\widetilde q$. Then $q\in(1,\infty)$ and hence, by the boundedness
of $\cm$ on $L^q(\rn)$ and Definition \ref{Deatom}(ii),
we conclude that, for any $i\in\zz$ and $j\in\nn$,
\begin{align*}
\lf\|\lf[
\lf\|\mathbf{1}_{B_{i,j}}\r\|_{X}S(a_{i,j})\r]^{\widetilde q}\mathbf{1}_{4B_{i,j}}\r\|_{L^{p_0}(\rn)}
\lesssim\lf\|\mathbf{1}_{B_{i,j}}\r\|_{X}^{\widetilde q}\lf\|S(a_{i,j})\mathbf{1}_{4B_{i,j}}\r\|_{L^q(\rn)}^{\widetilde q}
\lesssim\lf\|\mathbf{1}_{B_{i,j}}\r\|_{X}^{\widetilde q}\lf\|a_{i,j}\r\|_{L^q(\rn)}^{\widetilde q}
\lesssim\lf|B_{i,j}\r|^\frac1{p_0},
\end{align*}
which, combined with Lemma \ref{Le45}, \eqref{EqHLMS} and $(1-a)\widetilde q>1$, further implies that
\begin{align*}
\mathrm{I_{1,1}}&\lesssim2^{-i_0\widetilde q(1-a)}
\lf[\sum_{i=-\infty}^{i_0-1}2^{(1-a)i\widetilde qr_0}\lf\|\lf(\sum_{j\in\nn}\mathbf{1}_{4B_{i,j}}\r)^\frac1{r_0}\r\|_{X}^{r_0}\r]^{\frac1{r_0}}
\lesssim2^{-i_0\widetilde q(1-a)}
\lf[\sum_{i=-\infty}^{i_0-1}2^{(1-a)i\widetilde qr_0}\lf\|\lf(\sum_{j\in\nn}\mathbf{1}_{cB_{i,j}}\r)^\frac1{r_0}\r\|_{X}^{r_0}\r]^{\frac1{r_0}}\\
&\lesssim2^{-i_0\widetilde q(1-a)}
\lf[\sum_{i=-\infty}^{i_0-1}2^{[(1-a)\widetilde
q-1]ir_0}\r]^{\frac1{r_0}}\sup_{i\in\zz}2^i\lf\|\sum_{j\in\nn}
\mathbf{1}_{B_{i,j}}\r\|_{X}\lesssim\alpha^{-1}
\sup_{i\in\zz}2^i\lf\|\sum_{j\in\nn}\mathbf{1}_{B_{i,j}}\r\|_{X}.
\end{align*}
This shows that
\begin{equation}\label{Eq411}
\alpha\mathrm{I_{1,1}}\lesssim\sup_{i\in\zz}2^i\lf\|\sum_{j\in\nn}\mathbf{1}_{B_{i,j}}\r\|_{X}.
\end{equation}

As for $\mathrm{I_{1,2}}$,
for any $i\in\zz$, $j\in\nn$ and
$x\in\rn$, we first write
\begin{equation}\label{eq611}
[(a_{i,j})(x)]^2=\int_{0}^{|x-x_{i,j}|/4}\int_{|y-x|<t}|a_{i,j}\ast\phi_t(y)|^2\,\frac{dydt}{t^{n+1}}
+\int_{|x-x_{i,j}|/4}^{\infty}\int_{|y-x|<t}\cdots=:\mathrm{J}_1+\mathrm{J}_2,
\end{equation}
where $x_{i,j}$ denotes the center of $B_{i,j}$ and $\phi$ is as in \eqref{Eq62}.

We first deal with $\mathrm{J}_1$.
To do so, using the vanishing moments of $a_{i,j}$ and the Taylor remainder theorem, we have,
for any $i\in\zz$, $j\in\nn$,
$t\in(0,\infty)$ and $y\in\rn$,
\begin{align}\label{EqMNa0}
|a_{i,j}\ast\phi_t(y)|&=\lf|\int_{B_{i,j}}a_{i,j}(z)\lf[\phi\lf(\frac{y-z}{t}\r)-\sum_{|\beta|\le d}\frac{\partial^\beta\phi(\frac{y-x_{i,j}}{t})}
{\beta!}\lf(\frac{x_{i,j}-z}{t}\r)^\beta\r]\,\frac{dz}{t^n}\r|\\\noz
&\lesssim\int_{B_{i,j}}\lf|a_{i,j}(z)\r|\sum_{|\beta|= d+1}\lf|\partial^\beta\phi\lf(\frac{\xi}{t}\r)\r|
\lf|\frac{x_{i,j}-z}{t}\r|^{d+1}\,\frac{dz}{t^n},
\end{align}
where $\xi:=(x-x_{i,j})+\theta(x_{i,j}-y)$ for some $\theta\in[0,1]$.

For any $i\in\zz$, $j\in\nn$, $x\in(4B_{i,j})^\complement$, $|y-x|<t\le\frac{|x-x_{i,j}|}{4}$ and
$z\in B_{i,j}$,
it is easy to see that $|y-x_{i,j}|\sim|x-x_{i,j}|$
and $|\xi|\geq|x-x_{i,j}|-|x_{i,j}-z|\geq\frac12|x-x_{i,j}|$. By this, \eqref{EqMNa0}, the fact that $\phi\in\cs(\rn)$
and the H\"older inequality, we conclude that, for any $i\in\zz$, $j\in\nn$,
$x\in(4B_{i,j})^\complement$ and $|y-x|<t\le\frac{|x-x_{i,j}|}{4}$,
\begin{align*}
|a_{i,j}\ast\phi_t(y)|
&\lesssim t\int_{B_{i,j}}|a_{i,j}(z)|\frac{|z-x_{i,j}|^{d+1}}{|x-x_{i,j}|^{n+d+2}}\,dz
\lesssim\frac{t(r_{i,j})^{d+1}}{|x-x_{i,j}|^{n+d+2}}
\lf[\int_{B_{i,j}}|a_{i,j}(z)|^q\,dz\r]^{1/q}|B_{i,j}|^{1/q'}\\\noz
&\lesssim\frac{t}{|x-x_{i,j}|}\lf\|\mathbf{1}_{B_{i,j}}\r\|_{X}^{-1}\lf(\frac{r_{i,j}}{|x-x_{i,j}|}\r)^{n+d+1},
\end{align*}
which implies that, for any $i\in\zz$, $j\in\nn$ and
$x\in(4B_{i,j})^\complement$,
\begin{equation}\label{eq612}
\mathrm{J}_1\lesssim
\frac{1}{|x-x_{i,j}|^2}\lf\|\mathbf{1}_{B_{i,j}}\r\|_{X}^{-2}\lf(\frac{r_{i,j}}{|x-x_{i,j}|}\r)^{2(n+d+1)}
\int_{0}^{|x-x_{i,j}|/4}t\,dt
\lesssim
\lf\|\mathbf{1}_{B_{i,j}}\r\|_{X}^{-2}\lf[\cm(\mathbf{1}_{B_{i,j}})(x)\r]^\frac{2(n+d+1)}{n}.
\end{equation}

As for $\mathrm{J_{2}}$, by \eqref{EqMNa0}, the fact that $t\geq\frac{|x-x_{i,j}|}{4}$
and the H\"older inequality, we conclude that, for any $i\in\zz$, $j\in\nn$,
$x\in(4B_{i,j})^\complement$ and $|y-x|<t$,
\begin{align*}
|a_{i,j}\ast\phi_t(y)|
&\lesssim \int_{B_{i,j}}|a_{i,j}(z)|\frac{|z-x_{i,j}|^{d+1}}{t^{n+d+1}}\,dz
\lesssim\lf\|\mathbf{1}_{B_{i,j}}\r\|_{X}^{-1}\frac{t(r_{i,j})^{d+1}}{|x-x_{i,j}|^{n+d+2}}
\lesssim\lf\|\mathbf{1}_{B_{i,j}}\r\|_{X}^{-1}\lf(\frac{r_{i,j}}{t}\r)^{n+d+1},
\end{align*}
which implies that, for any $i\in\zz$, $j\in\nn$ and
$x\in(4B_{i,j})^\complement$,
\begin{equation}\label{eq614}
\mathrm{J}_2\lesssim
\lf\|\mathbf{1}_{B_{i,j}}\r\|_{X}^{-2}r_{i,j}^{2(n+d+1)}
\int_{|x-x_{i,j}|/4}^{\infty}t^{-2(n+d+1)-1}\,dt
\lesssim
\lf\|\mathbf{1}_{B_{i,j}}\r\|_{X}^{-2}\lf[\cm(\mathbf{1}_{B_{i,j}})(x)\r]^\frac{2(n+d+1)}{n}.
\end{equation}
Thus, from \eqref{eq611}, \eqref{eq612} and \eqref{eq614}, we deduce that,
for any $i\in\zz$, $j\in\nn$ and
$x\in(4B_{i,j})^\complement$,
\begin{equation}\label{Eq69}
\lf|S(a_{i,j})(x)\r|\lesssim\lf\|\mathbf{1}_{B_{i,j}}\r\|_{X}^{-1}
\lf[\cm(\mathbf{1}_{B_{i,j}})(x)\r]^{\frac{(n+d+1)}{n}}.
\end{equation}
Observe that $d\geq\lfloor n(\frac1{\underline{p}}-1)\rfloor$ implies that $\underline{p}\in(\frac{n}{n+d+1},1]$.
Let $r_1\in(0,\frac{n}{n+d+1})\subset(0,\underline{p})$, $q_1\in(\frac{n}{(n+d+1)r_1},\frac1{r_1})\subset(1,\infty)$ and $a\in(0,1-\frac1{q_1})$.
By the H\"older inequality, we obtain
\begin{align*}
&\sum_{i=-\infty}^{i_0-1}\sum_{j\in\nn}\lambda_{i,j}S(a_{i,j})\mathbf{1}_{(4B_{i,j})^\complement}
\le\frac{2^{i_0a}}{(2^{aq_1'}-1)^{1/q_1'}}
\lf\{\sum_{i=-\infty}^{i_0-1}2^{-iaq_1}\lf[\sum_{j\in\nn}
\lambda_{i,j}S(a_{i,j})\mathbf{1}_{(4B_{i,j})^\complement}\r]^{q_1}\r\}^{1/q_1},
\end{align*}
where $q_1':={q_1}/{(q_1-1)}$. From this, Definition \ref{Debqfs}(ii), \eqref{Eq69}, the definition of $\lambda_{i,j}$
and the assumption that $X^{1/r_1}$ is a ball Banach function space, we deduce that
\begin{align*}
\mathrm{I_{1,2}}&\lesssim
\lf\|\mathbf{1}_{\{x\in\rn:\ 2^{i_0a}
\{\sum_{i=-\infty}^{i_0-1}2^{-iaq_1}[\sum_{j\in\nn}
\lambda_{i,j}S(a_{i,j})(x)\mathbf{1}_{(4B_{i,j})^\complement}(x)]^{q_1}\}^{1/q_1}>2^{i_0-2}\}}\r\|_{X}\\
&\lesssim2^{-i_0q_1(1-a)}\lf\|\sum_{i=-\infty}^{i_0-1}2^{-iaq_1}\lf[\sum_{j\in\nn}
\lambda_{i,j}S(a_{i,j})\mathbf{1}_{(4B_{i,j})^\complement}\r]^{q_1}\r\|_{X}\\
&\lesssim2^{-i_0q_1(1-a)}\lf\{\sum_{i=-\infty}^{i_0-1}2^{(1-a)iq_1r_1}\lf\|\sum_{j\in\nn}\lf[
\cm(\mathbf{1}_{B_{i,j}})\r]^{\frac{(n+d+1)q_1r_1}{n}}\r\|_{X^{\frac1{r_1}}}\r\}^\frac1{r_1}.
\end{align*}
It is easy to see that $\frac{(n+d+1)q_1r_1}{n}\in(1,\infty)$, $\frac{n}{(n+d+1)q_1}\in(0,r_1)\subset(0,\underline{p})$
and $(1-a)q_1\in(1,\infty)$.
Then, by Definition \ref{Debf}(i), \eqref{EqHLMS} and $\sum_{j\in\nn}\mathbf{1}_{cB_{i,j}}\le A$, we further find that
\begin{align*}
\mathrm{I_{1,2}}&\lesssim2^{-i_0q_1(1-a)}
\lf[\sum_{i=-\infty}^{i_0-1}2^{(1-a)iq_1r_1}\lf\|\lf\{\sum_{j\in\nn}\lf[
\cm(\mathbf{1}_{B_{i,j}})\r]^{\frac{(n+d+1)q_1r_1}{n}}\r\}
^\frac{n}{(n+d+1)q_1r_1}\r\|_{X^{\frac{(n+d+1)q_1}{n}}}^\frac{(n+d+1)q_1r_1}{n}\r]^\frac1{r_1}\\
&\lesssim2^{-i_0q_1(1-a)}\lf\{\sum_{i=-\infty}^{i_0-1}2^{(1-a)iq_1{r_1}}\lf\|\sum_{j\in\nn}
\mathbf{1}_{B_{i,j}}\r\|_{X^{\frac1{r_1}}}\r\}^\frac1{r_1}\\
&\lesssim2^{-i_0q_1(1-a)}\lf\{\sum_{i=0}^{i_0-1}2^{[(1-a)q_1-1]ir_1}2^{ir_1}
\lf\|\lf(\sum_{j\in\nn}\mathbf{1}_{cB_{i,j}}\r)^\frac1{r_1}\r\|_{X}^{r_1}\r\}^\frac1{r_1}
\lesssim\alpha^{-1}\sup_{i\in\zz}2^i\lf\|\sum_{j\in\nn}\mathbf{1}_{B_{i,j}}\r\|_{X},
\end{align*}
which implies that
\begin{equation*}
\alpha\mathrm{I_{1,2}}\lesssim\sup_{i\in\zz}2^i\lf\|\sum_{j\in\nn}\mathbf{1}_{B_{i,j}}\r\|_{X}.
\end{equation*}
From this, \eqref{Eq6I1} and \eqref{Eq411}, it follows that
\begin{equation}\label{Eq612}
\alpha\mathrm{I_1}\lesssim\sup_{i\in\zz}\lf\|\sum_{j\in\nn}
\frac{\lambda_{i,j}\mathbf{1}_{B_{i,j}}}{\|\mathbf{1}_{B_{i,j}}\|_{X}}\r\|_{X}.
\end{equation}

Next we deal with $\mathrm{I_2}$. Let $r_2\in(0,\underline p)$.
Then, by \eqref{EqHLMS}, Definition \ref{Debf}(i),
the assumption that $X^{1/r_2}$ is a ball Banach
function space and $\sum_{j\in\nn}\mathbf{1}_{cB_{i,j}}\le A$,
we conclude that
\begin{align*}
\mathrm{I_2}&\lesssim\lf\|\mathbf{1}_{A_{i_0}}\r\|_{X}\lesssim
\lf\|\sum_{i=i_0}^\infty\sum_{j\in\nn}\mathbf{1}_{2B_{i,j}}\r\|_{X}
\lesssim\lf\|\sum_{i=i_0}^\infty\sum_{j\in\nn}\mathbf{1}_{cB_{i,j}}\r\|_{X}
\sim\lf\|\lf\{\sum_{i=i_0}^\infty\sum_{j\in\nn}
\mathbf{1}_{cB_{i,j}}\r\}^{r_2}\r\|_{X^\frac1{r_2}}^\frac1{r_2}\\
&\lesssim\lf[\sum_{i=i_0}^\infty\lf\|\sum_{j\in\nn}
\mathbf{1}_{cB_{i,j}}\r\|_{X^\frac1{r_2}}\r]^\frac1{r_2}
\lesssim\lf[\sum_{i=i_0}^\infty\lf\|\sum_{j\in\nn}
\mathbf{1}_{cB_{i,j}}\r\|_{X}^{r_2}\r]^\frac1{r_2}
\lesssim\lf\{\sum_{i=i_0}^\infty2^{-ir_2}\lf[2^i
\lf\|\sum_{j\in\nn}\mathbf{1}_{B_{i,j}}\r\|_{X}\r]^{r_2}\r\}^{\frac1{r_2}}\\
&\lesssim\sup_{i\in\nn}2^i\lf\|\sum_{j\in\nn}
\mathbf{1}_{B_{i,j}}\r\|_{X}\lf(\sum_{i=i_0}^\infty2^{-ir_2}\r)^\frac1{r_2}
\lesssim\alpha^{-1}\sup_{i\in\zz}2^i\lf\|\sum_{j\in\nn}\mathbf{1}_{B_{i,j}}\r\|_{X},
\end{align*}
which implies that
\begin{equation}\label{Eq415}
\alpha\mathrm{I_2}\lesssim\sup_{i\in\zz}2^i\lf\|\sum_{j\in\nn}\mathbf{1}_{B_{i,j}}\r\|_{X}.
\end{equation}

It remains to estimate $\mathrm{I_3}$. Recall that $\underline{p}\in(\frac{n}{n+d+1},1]$
and hence there exists a $r_3\in(\frac{n}{\underline{p}(n+d+1)},1)$.
By the assumption that $X^\frac{(n+d+1)r_3}{n}$ is a ball Banach function space
and \eqref{Eq69}, we conclude that
\begin{align*}
\mathrm{I_3}&\lesssim\lf\|\mathbf{1}_{\{x\in(A_{i_0})^\complement:\
\sum_{i=i_0}^\infty\sum_{j\in\nn}\lambda_{i,j}S(a_{i,j})(x)>\frac{\alpha}2\}}\r\|_{X}
\lesssim\alpha^{-r_3}\lf\|\sum_{i=i_0}^\infty\sum_{j\in\nn}
\lf[\lambda_{i,j}S(a_{i,j})\r]^{r_3}\mathbf{1}_{(A_{i_0})^\complement}\r\|_{X}\\
&\sim\alpha^{-r_3}\lf\|\lf\{\sum_{i=i_0}^\infty\sum_{j\in\nn}
\lf[\lambda_{i,j}S(a_{i,j})\r]^{r_3}
\mathbf{1}_{(A_{i_0})^\complement}\r\}^\frac{n}{(n+d+1)r_3}
\r\|_{X^\frac{(n+d+1)r_3}{n}}^\frac{(n+d+1)r_3}{n}\\
&\lesssim\alpha^{-r_3}\lf[\sum_{i=i_0}^\infty
\lf\|\lf\{\sum_{j\in\nn}\lf[\lambda_{i,j}S(a_{i,j})\r]^{r_3}
\mathbf{1}_{(A_{i_0})^\complement}\r\}^\frac{n}{(n+d+1)r_3}
\r\|_{X^\frac{(n+d+1)r_3}{n}}\r]^\frac{(n+d+1)r_3}{n}\\
&\lesssim\alpha^{-r_3}\lf\{\sum_{i=i_0}^\infty2^\frac{in}{(n+d+1)}\lf\|\sum_{j\in\nn}
\lf[\cm(\mathbf{1}_{B_{i,j}})\r]^\frac{(n+d+1)r_3}{n}
\r\|_{X}^\frac{n}{(n+d+1)r_3}\r\}^\frac{(n+d+1)r_3}{n}.
\end{align*}
Since $\frac{n}{(n+d+1)r_3}\in(0,\underline{p})\subset(0,1)$,
from Definition \ref{Debf}(i) and Assumption \ref{a2.15}, it follows that
\begin{align*}
\mathrm{I_3}&\lesssim\alpha^{-r_3}\lf[\sum_{i=i_0}^\infty2^\frac{in}{(n+d+1)}\lf\|\sum_{j\in\nn}
\mathbf{1}_{B_{i,j}}\r\|_{X}^\frac{n}{(n+d+1)r_3}\r]^\frac{(n+d+1)r_3}{n}
\lesssim\alpha^{-r_3}\sup_{i\in\zz}2^i\lf\|\sum_{j\in\nn}\mathbf{1}_{B_{i,j}}\r\|_{X}
\lf[\sum_{i=i_0}^\infty2^\frac{in(r_3-1)}{(n+d+1)r_3}\r]^\frac{(n+d+1)r_3}{n}\\
&\lesssim\alpha^{-1}\sup_{i\in\zz}2^i\lf\|\sum_{j\in\nn}\mathbf{1}_{B_{i,j}}\r\|_{X},
\end{align*}
namely,
\begin{equation}\label{Eq416}
\alpha\mathrm{I_3}\lesssim\sup_{i\in\zz}2^i\lf\|\sum_{j\in\nn}\mathbf{1}_{B_{i,j}}\r\|_{X}.
\end{equation}

Finally, combining \eqref{EqI6123}, \eqref{Eq612}, \eqref{Eq415} and \eqref{Eq416}, we obtain
$$
\lf\|S(f)\r\|_{WX}\lesssim\sup_{i\in\zz}\lf\|\sum_{j\in\nn}
\frac{\lambda_{i,j}\mathbf{1}_{B_{i,j}}}{\|\mathbf{1}_{B_{i,j}}\|_{X}}\r\|_{X},
$$
which completes the proof of the necessity of Theorem \ref{Tharea}.

Next, we prove the sufficiency. Let $f\in\cs'(\rn)$, $f$ vanish weakly at infinity and $S(f)\in WX$.
Then, by these and Lemma \ref{Le47}, we conclude that
there exists a $\psi\in\cs(\rn)$ such that
\begin{equation}\label{67mm}
\supp(\widehat{\psi})\subset B(\vec{0}_n,b) \setminus B(\vec{0}_n,a),
\end{equation}
where $b,\ a\in(0,\infty)$ and $b>a$, and
\begin{equation}\label{66mm}
f=\int_0^\infty f\ast\phi_t\ast\psi_t\,\frac{dt}{t}\quad\text{in}\quad\cs'(\rn)
\end{equation}
with $\phi$ as in \eqref{Eq62}.
It remains to prove that $f\in WH_X(\rn)$.
For any $(x,t)\in{\mathbb R}^{n+1}_+$, let $F(x,t):=f\ast\phi_t(x)$. Then
$F\in WT_X^1(\rr^{n+1}_+)$. From this and Theorem \ref{That}, we deduce that
there exists a
sequence $\{A_{i,j}\}_{i\in\zz,j\in\nn}$ of $(T_X,\infty)$-atoms
associated, respectively, with balls $\{B_{i,j}\}_{i\in\zz,j\in\nn}$
satisfying that, for any $i\in\zz$, $\sum_{j\in\nn}\mathbf1_{cB_{i,j}}\le M_0$ with
$c\in(0,1]$ and $M_0$
being positive constants independent of $i\in\zz$, $j\in\nn$ and $f$,
such that
\begin{align}\label{456}
F=\sum_{i\in\zz}\sum_{j\in\nn}\lambda_{i,j}A_{i,j}
\end{align}
pointwisely on $\rr_+^{n+1}$, where $\lambda_{i,j}:=2^i\|\mathbf{1}_{B_{i,j}}\|_X$,
and
\begin{align}\label{457}
\sup_{i\in\zz}2^i\lf\|\sum_{j\in\nn}\mathbf1_{B_{i,j}}\r\|_X
\lesssim\lf\|F\r\|_{WT_X^1(\rr_+^{n+1})}\sim\lf\|S(f)\r\|_{WX}.
\end{align}

For any $i\in\zz$, $j\in\nn$ and $x\in\rn$, let
\begin{align*}
a_{i,j}(x):
=\int_0^{\infty}
\int_\rn A_{i,j}(y)\psi_t(x-y)\,\frac{dy\,dt}{t}.
\end{align*}

We now show that
\begin{equation}\label{kkk}
\sum_{i\in\zz}\sum_{j\in\nn}\lambda_{i,j} a_{i,j}\quad \text{converges}\ \text{in}\quad \cs'(\rn).
\end{equation}
We first prove that, for any given $i\in\zz$,
\begin{equation}\label{qw11}
\sum_{j\in\nn}\lambda_{i,j} a_{i,j}\quad \text{converges}\ \text{in}\quad \cs'(\rn).
\end{equation}
Indeed, by Lemma \ref{lyyy}, to show \eqref{qw11}, it suffices to prove that
\begin{equation}\label{qw4}
\lim_{N\to\infty}\lf\|\sum_{\gfz{N\le j\le N+k}{j\in\nn}}
\lambda_{i,j} a_{i,j}\r\|_{WH_X(\rn)}=0
\end{equation}
uniformly in $k\in\nn$. Similarly to the proof of \cite[Lemma 4.8]{hyy}, we conclude that, for any given $i\in\zz$,
up to a harmless constant multiple, $\{a_{i,j}\}_{j\in\nn}$ is a
sequence of $(X,\,q,\,d,\,\epsilon)$-molecules
associated, respectively, with balls $\{B_{i,j}\}_{j\in\nn}$,
where $q$, $d$ and $\epsilon$ are as in Lemma \ref{Thmolcha}.
Using this, \eqref{EqHLMS} and
Lemma \ref{Thmolcha}, we find that,
for any $N\in\nn$ and $k\in\nn$,
\begin{equation}\label{qw2}
\lf\|\sum_{\gfz{N\le j\le N+k}{j\in\nn}}
\lambda_{i,j} a_{i,j}\r\|_{WH_X(\rn)}\lesssim
2^i\lf\|
\sum_{\gfz{N\le j\le N+k}{j\in\nn}}\mathbf{1}_{B_{i,j}}\r\|_{X}
\lesssim
2^i\lf\|
\sum_{\gfz{N\le j}{j\in\nn}}\mathbf{1}_{cB_{i,j}}\r\|_{X},
\end{equation}
which, combined with the fact that
$X$ has an absolutely continuous quasi-norm as in Definition \ref{abo}, implies that
\eqref{qw4} holds true and hence completes the proof of \eqref{qw11}.

By \eqref{qw11} and Lemma \ref{lyyy},
to show \eqref{kkk}, it suffices to prove that
\begin{equation}\label{455}
\lim_{N\to\infty}\lf\|\sum_{\gfz{-N-k\le i\le -N}{i\in\zz}}\sum_{j\in\nn}
\lambda_{i,j} a_{i,j}\r\|_{WH_X(\rn)}=0
\end{equation}
uniformly in $k\in\nn$, and, for some given $s\in(0,1)$,
\begin{equation}\label{12}
\lim_{N\to\infty}\lf\|\sum_{\gfz{N\le i\le N+k}{i\in\zz}}\sum_{j\in\nn}
\lambda_{i,j} a_{i,j}\r\|_{WH_{X^s}(\rn)}=0
\end{equation}
uniformly in $k\in\nn$. We first prove \eqref{455}. Let $\{O_i\}_{i\in\zz}$ be as in \eqref{ly1}.
Then, from the definition of $\{O_i\}_{i\in\zz}$
and the fact that $X$ has an absolutely continuous quasi-norm as in Definition \ref{abo}, we deduce that
\begin{equation}\label{lyy2}
\lim_{i\to-\infty}\|S(f)\mathbf1_{\rn\setminus O_i}\|_{X}=0.
\end{equation}
By the proof of \cite[Lemma 4.8]{hyy}, we conclude that, up to a harmless
constant multiple,
$\{a_{i,j}\}_{i\in\zz,j\in\nn}$ is a sequence of $(X,\,q,\,d,\,\epsilon)$-molecules
associated, respectively, with balls $\{B_{i,j}\}_{i\in\zz,j\in\nn}$,
where $q$, $d$ and $\epsilon$ are as in Lemma \ref{Thmolcha}.
Using this and
Lemma \ref{Thmolcha}, we find that,
for any $N,\ k\in\nn$,
\begin{equation}\label{450}
\lf\|\sum_{\gfz{-N-k\le i\le -N}{i\in\zz}}\sum_{j\in\nn}
\lambda_{i,j} a_{i,j}\r\|_{WH_X(\rn)}\lesssim
\sup_{\gfz{-N-k\le i\le -N}{i\in\zz}}2^i\lf\|
\sum_{j\in\nn}\mathbf{1}_{B_{i,j}}\r\|_{X},
\end{equation}
which, combined with \eqref{lyy}, implies that
\begin{equation}\label{Eq514}
\lf\|\sum_{\gfz{-N-k\le i\le -N}{i\in\zz}}\sum_{j\in\nn}\lambda_{i,j} a_{i,j}\r\|_{WH_X(\rn)}\lesssim
\sup_{\gfz{-N-k\le i\le -N}{i\in\zz}}2^i\lf\|\mathbf1_{O_i}\r\|_X.
\end{equation}
Observe that there exists a positive constant $\widetilde{C}$ such that, for any $m,\ N\in\nn$,
\begin{align}\label{447}
\sup_{\gfz{i\le -N}{i\in\zz}}2^i\lf\|\mathbf1_{O_i}\r\|_X&
\le\widetilde{C}\sup_{\gfz{i\le -N}{i\in\zz}}2^i\lf\|\mathbf1_{O_i\setminus O_{i+m}}\r\|_X+
\widetilde{C}\sup_{\gfz{i\le -N}{i\in\zz}}2^i\lf\|\mathbf1_{O_{i+m}}\r\|_X\\\noz
&\le\widetilde{C}\sup_{\gfz{i\le -N}{i\in\zz}}2^i\lf\|\mathbf1_{O_i\setminus O_{i+m}}\r\|_X+
\widetilde{C}2^{-m}\sup_{i\in\zz}2^i\lf\|\mathbf1_{O_i}\r\|_X.
\end{align}
By the fact that $S(f)\in WX$,
we know that, for any given $\epsilon\in(0,\infty)$,
there exists an $m_0\in\nn$ such that
\begin{align}\label{451}
\widetilde{C}2^{-m_0}\sup_{i\in\zz}2^i\lf\|\mathbf1_{O_i}\r\|_X
\le\widetilde{C}2^{-m_0}\lf\|S(f)\r\|_{WX}<\epsilon/2,
\end{align}
where $\widetilde{C}$
is as in \eqref{447}.
Then, from \eqref{lyy2}, \eqref{Eq514}, \eqref{447}, \eqref{451},
the definition of $O_i$ and Definition \ref{Debqfs}(ii), it follows that there exists an
$N_0\in\nn$ such that, when $N\in\nn\cap[N_0,\fz)$, for any $k\in\nn$,
\begin{align}\label{448}
\lf\|\sum_{\gfz{-N-k\le i\le -N}{i\in\zz}}\sum_{j\in\nn}\lambda_{i,j} a_{i,j}\r\|_{WH_X(\rn)}&\lesssim\sup_{\gfz{i\le -N}{i\in\zz}}2^i\lf\|\mathbf1_{O_i}\r\|_X
<\widetilde{C}\sup_{\gfz{i\le -N}{i\in\zz}}2^i\lf\|\mathbf1_{O_i\setminus O_{i+m_0}}\r\|_X+\epsilon/2\\\noz
&\le\widetilde{C}\lf\|S(f)\mathbf1_{\rn\setminus O_{-N+m_0}}\r\|_X+\epsilon/2<\epsilon,
\end{align}
which further implies that, for any $k\in\nn$,
\begin{align}\label{449}
\lim_{N\to\infty}\lf\|\sum_{\gfz{-N-k\le i\le -N}{i\in\zz}}
\sum_{j\in\nn}\lambda_{i,j} a_{i,j}\r\|_{WH_X(\rn)}=0.
\end{align}
This finishes the proof of \eqref{455}.

Now we prove \eqref{12}. Let $s\in(0,1)$. For any $i,\ j\in\nn$,
let $\mu_{i,j}:=2^i\|\mathbf{1}_{B_{i,j}}\|_{X^s}$ and
$$b_{i,j}:=\frac{\|\mathbf{1}_{B_{i,j}}\|_{X}}{\|\mathbf{1}_{B_{i,j}}\|_{X^s}}a_{i,j}.$$
Then we have, for any $i,\ j\in\nn$,
$\mu_{i,j}b_{i,j}=\lambda_{i,j} a_{i,j}$.
Moreover, similarly to the proof of \cite[Lemma 4.8]{hyy}, up to a harmless
constant multiple, we conclude that
$\{b_{i,j}\}_{i\in\zz,j\in\nn}$ is a
sequence of $(X^s,q,d,\epsilon)$-molecules
associated, respectively, with balls $\{B_{i,j}\}_{i\in\zz,j\in\nn}$,
where $q$ and $\epsilon$ are as in Lemma \ref{Thmolcha}, and $d\in\zz_+$ with $d\geq \lfloor n(1/s\underline{p}-1)\rfloor$.
From this, Lemma \ref{Thmolcha} and the fact that, for any $i\in\nn$ and $j\in\nn$,
$\mu_{i,j}b_{i,j}=\lambda_{i,j} a_{i,j}$,
we deduce that, for any $N,\ k\in\nn$,
\begin{equation*}
\lf\|\sum_{\gfz{N\le i\le N+k}{i\in\zz}}\sum_{j\in\nn}
\lambda_{i,j} a_{i,j}\r\|_{WH_{X^s}(\rn)}=
\lf\|\sum_{\gfz{N\le i\le N+k}{i\in\zz}}\sum_{j\in\nn}
\mu_{i,j}b_{i,j}\r\|_{WH_{X^s}(\rn)}\lesssim
\sup_{\gfz{N\le i\le N+k}{i\in\zz}}2^i\lf\|
\sum_{j\in\nn}\mathbf{1}_{B_{i,j}}\r\|_{X^s}.
\end{equation*}
By this, \eqref{457}, \eqref{EqHLMS} and the fact that $\{B_{i,j}\}_{i\in\zz,j\in\nn}$
satisfy that, for any $i\in\zz$, $\sum_{j\in\nn}\mathbf1_{cB_{i,j}}\le M_0$ with $c\in(0,1]$ and $M_0$
being positive constants independent of $i\in\zz$, $j\in\nn$ and $f$, we conclude that,
for any $N,\ k\in\nn$,
\begin{align*}
\lf\|\sum_{\gfz{N\le i\le N+k}{i\in\zz}}\sum_{j\in\nn}\lambda_{i,j} a_{i,j}\r\|_{WH_{X^s}(\rn)}&\lesssim
\sup_{\gfz{N\le i\le N+k}{i\in\zz}}2^i\lf\|
\sum_{j\in\nn}\mathbf{1}_{B_{i,j}}\r\|_{X^s}\lesssim
\sup_{\gfz{N\le i\le N+k}{i\in\zz}}2^i\lf\|
\sum_{j\in\nn}\mathbf{1}_{cB_{i,j}}\r\|_{X^s}\\\noz
&\lesssim\lf(\sup_{\gfz{N\le i\le N+k}{i\in\zz}}2^{is}\lf\|
\sum_{j\in\nn}\mathbf{1}_{cB_{i,j}}\r\|_{X}\r)^{1/s}
\sim\lf[\sup_{\gfz{N\le i\le N+k}{i\in\zz}}2^{(s-1)i}2^{i}\lf\|
\sum_{j\in\nn}\mathbf{1}_{B_{i,j}}\r\|_{X}\r]^{1/s}\\\noz
&\lesssim2^{\frac{(s-1)N}{s}},
\end{align*}
which implies that \eqref{12} holds true. Thus, \eqref{kkk} holds true.

Let $g:=\sum_{i\in\zz}\sum_{j\in\nn}\lambda_{i,j} a_{i,j}$ in $\cs'(\rn)$.
Next we show that
\begin{equation}\label{zz}
f=g\quad\text{in}\quad\cs'(\rn).
\end{equation}
Indeed, similarly to the proof of \cite[Lemma 4.8]{hyy}, we conclude that, up to a harmless
constant multiple, $\{a_{i,j}\}_{i\in\zz,j\in\nn}$ is a
sequence of $(X,\,q,\,d,\,\epsilon)$-molecules
associated, respectively, with balls $\{B_{i,j}\}_{i\in\zz,j\in\nn}$,
where $q$, $d$ and $\epsilon$ are as in Lemma \ref{Thmolcha}.
From this, \eqref{457} and Lemmas \ref{Thmolcha} and \ref{Le64}, we deduce that
$g\in WH_X(\rn)$ and hence $g$ vanishes weakly at infinity.
Let $\Phi\in\mathcal{S}(\rn)$ be such that
$$
\mathbf{1}_{B(\vec{0}_n,4) \setminus B(\vec{0}_n,2)}\le
\widehat\Phi\le\mathbf{1}_{B(\vec{0}_n,8) \setminus B(\vec{0}_n,1)}.
$$
We claim that, for any $t_1\in(0,\infty)$, $\Phi_{t_1}\ast g=\Phi_{t_1}\ast f$ in $\cs'(\rn)$.
Assuming that this claim holds true for the moment, then, by the fact that
$\mathbf{1}_{B(\vec{0}_n,4) \setminus B(\vec{0}_n,2)}\le
\widehat\Phi\le\mathbf{1}_{B(\vec{0}_n,8) \setminus B(\vec{0}_n,1)}$, we conclude that,
for any $x\in\rn\setminus \{\vec{0}_n\}$, there exists a $t_0\in(0,\infty)$ such that
$x\in\{y\in\rn:\ 2/t_0<|y|<4/t_0\}$ and hence, for any $\varphi\in C_c^\infty(\rn)$
with $\supp(\varphi)\subset\{y\in\rn:\ 2/t_0<|y|<4/t_0\}$,
$\varphi(\cdot)\widehat{\Phi}(t_0\cdot)=\varphi(\cdot)$.
From this and the above claim, we deduce that,
for any $\varphi\in C_c^\infty(\rn)$ with
$\supp(\varphi)\subset\{y\in\rn:\ 2/t_0<|y|<4/t_0\}$,
$$\langle \widehat{f-g},\varphi\rangle
=\langle \widehat{f-g},\widehat{\Phi}(t_0\cdot)\varphi\rangle=
\langle \Phi_{t_0}\ast(f-g),\widehat\varphi\rangle=0.$$
Thus, $\widetilde{\supp}(\widehat{f-g})\subset\{\vec{0}_n\}$.
From this, \cite[Corollary 2.4.2]{G1} and the fact that $f$ and $g$ vanish weakly at infinity,
it easily follows that $f=g$ in $\cs'(\rn)$, which is the desired conclusion.
Therefore, to show \eqref{zz}, it remains to prove the above claim.
Fix $t_1\in(0,\infty)$. By the fact that $f$ vanishes weakly at infinity and Lemma \ref{Le47},
we conclude that
\begin{equation}\label{659m}
\Phi_{t_1}\ast f=\int_0^\infty \Phi_{t_1}\ast\psi_t
\ast f\ast\phi_t\,\frac{dt}{t}\quad\text{in}\quad\cs'(\rn).
\end{equation}
Observe that, for any $t\in(0,\infty)$,
$$\Phi_{t_1}\ast f\ast\phi_t\ast\psi_t=
\lf[\widehat\Phi(t_1\cdot)\widehat\phi(t\cdot)\widehat\psi(t\cdot)\widehat f\,\r]^\vee,$$
which, together with \eqref{659m}, \eqref{67mm} and the fact that
$$
\mathbf{1}_{B(\vec{0}_n,4) \setminus B(\vec{0}_n,2)}\le
\widehat\Phi\le\mathbf{1}_{B(\vec{0}_n,8) \setminus B(\vec{0}_n,1)},
$$
further implies that
$$
\Phi_{t_1}\ast f=\int_{at_1/8}^{bt_1} \Phi_{t_1}\ast\psi_t
\ast f\ast\phi_t\,\frac{dt}{t}\quad\text{in}\quad\cs'(\rn).
$$
Moreover, from \cite[Theorem 2.3.20]{G1}, it follows that there exists an $N\in\nn$ such that,
for any $t\in[at_1/8,bt_1]$ and $x\in\rn$,
$$
|f\ast\phi_t(x)|\lesssim(1+|x|)^N,
$$
which, combined with \eqref{456} and the fact $\{\supp(A_{i,j})\}_{i\in\zz,j\in\nn}$
have pairwise disjoint interiors (see Remark \ref{buji}), implies that,
for any $t\in[at_1/8,bt_1]$ and $x\in\rn$,
$$
\sum_{i\in\zz}\sum_{j\in\nn}\lambda_{i,j}|A_{i,j}(x,t)|=|f\ast\phi_t(x)|\lesssim(1+|x|)^N.
$$
From this, \eqref{66mm} and the Lebesgue dominated convergence theorem,
we deduce that
$$
\Phi_{t_1}\ast f=\sum_{i\in\zz}\sum_{j\in\nn}\lambda_{i,j}\int_{at_1/8}^{bt_1} \Phi_{t_1}\ast\psi_t
\ast A_{i,j}\,\frac{dt}{t}=\Phi_{t_1}\ast g
$$
in $\mathcal{S}'(\rn)$.
Thus, we complete the proof of the above claim and hence of \eqref{zz}.

By \eqref{zz}, \eqref{457}, the fact that $\{a_{i,j}\}_{i\in\zz,j\in\nn}$ is a
sequence of $(X,\,q,\,d,\,\epsilon)$-molecules and Lemma \ref{Thmolcha},
we conclude that $f\in WH_X(\rn)$ and $\|f\|_{WH_X(\rn)}
\lesssim\|S(f)\|_{WX}$.
Thus, we complete the proof of the sufficiency of Theorem \ref{Tharea} and hence the proof
of Theorem \ref{Tharea}.
\end{proof}

\subsection{Characterization by the Littlewood--Paley $g$-function}\label{s3.2}

We have the following Littlewood--Paley $g$-function characterization of $WH_X(\rn)$.

\begin{theorem}\label{Thgf}
Let $X$ satisfy all assumptions in Theorem \ref{Tharea} and Assumption \ref{a2.17}.
Then $f\in WH_X(\rn)$ if and only if $f\in\cs'(\rn)$,
$f$ vanishes weakly at infinity and $g(f)\in WX$, where $g(f)$ is as in \eqref{eq63}. Moreover,
there exists a positive constant $C$ such that,
for any $f\in WH_X(\rn)$,
$$
C^{-1}\|g(f)\|_{WX}\le\|f\|_{WH_X(\rn)}\le C\|g(f)\|_{WX}.
$$
\end{theorem}

By an argument similar to that used in the proof of
the necessity of Theorem \ref{Tharea}, we can show that
the Littlewood--Paley $g$-function is bounded from $WH_X(\rn)$ to $WX$ and we
omit the details here.

\begin{proposition}\label{Pro66}
Let all the assumptions be the same as in Theorem \ref{Tharea}.
If $f\in WH_X(\rn)$, then $g(f)\in WX$ and
$$
\lf\|g(f)\r\|_{WX}\lesssim\lf\|f\r\|_{WH_X(\rn)},
$$
where $g(f)$ is as in \eqref{eq63} and the implicit positive constant is independent of $f$.
\end{proposition}

To prove Theorem \ref{Thgf}, we need the following Peetre type maximal function and square function.
For any $\phi\in\cs(\rn)$,
 $f\in\cs'(\rn)$ and $x\in\rn$, define
 $$
 (\phi_t^\ast f)_a(x):=\sup_{y\in\rn}\frac{|\phi_t\ast f(x+y)|}{(1+|y|/t)^a}
 $$
and
$$
g_{a,\ast}(f)(x):=\lf\{\int_0^\infty\lf[\lf(\phi_t^\ast f\r)_a(x)\r]^2\,\frac{dt}{t}\r\}^{1/2},
$$
where $a,\ t\in(0,\infty)$ and $\phi_t(\cdot):=\frac1{t^n}\phi(\frac{\cdot}{t})$.

Now we establish the following discrete Calder\'on reproducing formulae, which might be well known.
But, for the convenience of the reader, we present some details.

\begin{lemma}\label{Le49}
\begin{enumerate}
\item[\rm(i)] Let $\phi$ be a Schwartz function satisfying that, for any $x\in\rn\setminus\{\vec 0_n\}$,
there exists a $j\in \zz$ such that $\widehat\phi(2^{j}x)\not=0$. Then there
exists a $\psi\in\cs(\rn)$ such that $\wh\psi\in C^\fz_c(\rn)$ with its support
away from $\vec 0_n$, $\wh\phi\wh\psi\ge 0$ and, for any $x\in\rn\setminus\{\vec 0_n\}$,
$$\sum_{j\in\zz}\wh\phi(2^{-j}x)\wh\psi(2^{-j}x)=1.$$
Moreover, there exist $\phi_0,\ \psi_0\in\cs(\rn)$ with
$\wh\phi_0,\ \wh\psi_0\in C^\fz_c(\rn)$ such that, for some $\epsilon\in(0,\infty)$ and
for any $x\in B(\vec 0_n,\epsilon)$, $\wh\phi_0(x)\wh\psi_0(x)> 0$ and, for any $x\in\rn$,
$$\wh\phi_0(x)\wh\psi_0(x)+\sum_{j=1}^\infty\wh\phi(2^{-j}x)\wh\psi(2^{-j}x)=1.$$

\item[\rm(ii)] Let $\phi$ be a Schwartz function satisfying that, for any $x\in\rn\setminus\{\vec 0_n\}$,
there exists a $j\in \zz$ such that $\widehat\phi(2^{j}x)\not=0$ and,
for some given $a,\ b\in(0,\infty)$ with $a<b$, $\wh\phi(x)>0$ for any $x\in\{y\in\rn:\
a\le|y|\le b\}$. Then there exists a $\psi\in\cs(\rn)$ such that $\wh\psi\in C^\fz_c(\rn)$ with its support
away from $\vec 0_n$, $\wh\phi\wh\psi\ge 0$, $\wh\psi(x)>0$ for any $x\in\{y\in\rn:\
a\le|y|\le b\}$ and, for any $x\in\rn\setminus\{\vec 0_n\}$,
$$\sum_{j\in\zz}\wh\phi(2^{-j}x)\wh\psi(2^{-j}x)=1.$$
Moreover, there exist $\phi_0,\ \psi_0\in\cs(\rn)$ with
$\wh\phi_0,\ \wh\psi_0\in C^\fz_c(\rn)$ such that, for some $\epsilon\in(0,\infty)$ and
for any $x\in B(\vec 0_n,\epsilon)$, $\wh\phi_0(x)\wh\psi_0(x)> 0$ and, for any $x\in\rn$,
$$\wh\phi_0(x)\wh\psi_0(x)+\sum_{j=1}^\infty\wh\phi(2^{-j}x)\wh\psi(2^{-j}x)=1.$$
\end{enumerate}
\end{lemma}

\begin{proof}
We prove this lemma by borrowing some ideas from the proof of Calder\'on \cite[Lemma 4.1]{C1975}.

We first show (i). To this end, let $\{\theta_k\}_{k\in\nn}\subset C^\fz_c((0,\infty))$
be an increasing sequence of functions such that, for any $k\in\nn$,
$\theta_k\geq0$ and, for any $t\in(0,\infty)$,
$\lim_{k\to\infty}\theta_k(t)=1$.
We claim that there exists a $k\in\nn$ such that, for any $x\in\rn\setminus\{\vec{0}_n\}$,
$$
\sum_{j\in\zz}\theta_k(2^{-j}|x|)|\wh\phi(2^{-j}x)|^2\in(0,\infty).
$$
Assuming that this claim holds true for the moment, we define
\begin{equation}\label{crf}
\wh\psi(x):=\frac{\theta_k(|x|)\overline{\wh\phi(x)}}
{\sum_{j\in\zz}\theta_k(2^{-j}|x|)|\wh\phi(2^{-j}x)|^2}
\quad\text{for}\ \text{any}\quad x\in\rn\setminus\{\vec{0}_n\},
\quad\text{and}\quad\wh\psi(\vec{0}_n):=0.
\end{equation}
Then it is easy to see that $\psi\in\cs(\rn)$ has all the desired properties
stated in Lemma \ref{Le49}(i).
Therefore, to show such a function $\psi$ exists, it remains to prove the above claim.
Observe that, for any given $x\in\{y\in\rn:\ 1\le|y|\le2\}$, there exists an
$N_{(x)}\in\nn$ such that
$$
\sum_{j=-N_{(x)}}^{N_{(x)}}|\wh\phi(2^{-j}x)|^2>0,
$$
which, combined with the properties of $\{\theta_k\}_{k\in\nn}$, implies that there
exists an $m_{(x)}\in\nn$
such that
$$
\sum_{j=-N_{(x)}}^{N_{(x)}}\theta_{m_{(x)}}(2^{-j}|x|)|\wh\phi(2^{-j}x)|^2>0.
$$
Observe that $\sum_{j=-N_{(x)}}^{N_{(x)}}\theta_{m_{(x)}}(2^{-j}|\cdot|)|\wh\phi(2^{-j}\cdot)|^2
$ is a continuous function on $\rn$.
By this, we conclude that, for any given $x\in\{y\in\rn:\ 1\le|y|\le2\}$,
there exist
$N_{(x)},\ m_{(x)}\in\nn$  and $r_{(x)}\in(0,\infty)$ such that,
for any $\xi\in B(x,r_{(x)})$,
$$
\sum_{j=-N_{(x)}}^{N_{(x)}}\theta_{m_{(x)}}(2^{-j}|\xi|)|\wh\phi(2^{-j}\xi)|^2>0,
$$
which, together with the fact that $\{y\in\rn:\ 1\le|y|\le2\}$ is a compact set,
implies that
there exist $\ell\in\mathbb{N}$, $\{x_1,\ldots,x_\ell\}\subset\{y\in\rn:\ 1\le|y|\le2\}$,
$\{N_1,\ldots,N_\ell\},\ \{m_1,\ldots,m_\ell\}\subset\nn$
and $\{r_1,\ldots,r_\ell\}\subset(0,\infty)$ such that
$\{y\in\rn:\ 1\le|y|\le2\}\subset\cup_{i=1}^{\ell}B(x_i,r_i)$
and, for any $i\in\{1,\ldots,\ell\}$ and $\xi\in B(x_i,r_i)$,
$$
\sum_{j=-N_i}^{N_i}\theta_{m_i}(2^{-j}|\xi|)|\wh\phi(2^{-j}\xi)|^2>0.
$$
Let $m:=\max\{m_1,\ldots,m_\ell\}$. Then we have, for any
$\xi\in\{y\in\rn:\ 1\le|y|\le2\}$,
\begin{equation}\label{99}
\sum_{j\in\zz}\theta_{m}(2^{-j}|\xi|)|\wh\phi(2^{-j}\xi)|^2\in(0,\infty).
\end{equation}
Observe that, for any $x\in\rn\setminus\{\vec{0}_n\}$, there exists an $\ell\in\zz$
such that $2^{-\ell} x\in\{y\in\rn:\ 1\le|y|\le2\}$.
From this and \eqref{99}, we deduce that
$$
\sum_{j\in\zz}\theta_{m}(2^{-j}|x|)|\wh\phi(2^{-j}x)|^2=
\sum_{j\in\zz}\theta_{m}(2^{-j}2^{-\ell}|x|)|\wh\phi(2^{-j}2^{-\ell}x)|^2>0.
$$
Thus, we complete the proof of the above claim and hence the function $\psi$ indeed exists.

Let $a,\ b\in(0,\infty)$ be such that $\supp(\wh\psi)\subset\{y\in\rn:\ a<|y|< b\}$.
From this, it is easy to deduce that,
for any $x\in\{y\in\rn:\ |y|\geq b\}$,
\begin{equation}\label{95}
\sum_{j=1}^\infty\wh\phi(2^{-j}x)\wh\psi(2^{-j}x)=1.
\end{equation}
Choose $\phi_0\in\cs(\rn)$ satisfying that, for any
$x\in\{y\in\rn:\ |y|\geq b\}$, $\wh\phi_0(x)>0$.
For any $x\in\{y\in\rn:\ |y|>b\}$, let $\wh\psi_0(x):=0$ and,
for any $x\in\{y\in\rn:\ |y|\le b\}$, let $$\wh\psi_0(x):=\frac{1-\sum_{j=1}^\infty\wh\phi(2^{-j}x)\wh\psi(2^{-j}x)}{\wh\phi_0(x)}.$$
From this and \eqref{95}, we deduce that $\psi_0,\ \phi_0\in\cs(\rn)$ have all the desired properties
stated in (i) of this lemma, which then completes the proof of (i).

As for (ii), let $\phi$ and $a,\ b\in (0,\fz)$ be the same as in (ii).
Then it is easy to see that $\psi$ in \eqref{crf} has all the desired properties
stated in (ii) and, moreover, if we choose $\psi_0,\ \phi_0\in\cs(\rn)$
to be the same as in the proof of (i), then $\psi_0$ and $\phi_0$
also have all the desired properties stated in (ii) of this lemma,
which completes the proof of (ii) and hence of Lemma \ref{Le49}.
\end{proof}

\begin{remark}\label{dcrf} Observe that, in Lemma \ref{Le49}(ii), if $2a\le b$,
then $\phi$ automatically satisfies that, for any $x\in\rn\setminus\{\vec 0_n\}$,
there exists a $j\in \zz$ such that $\widehat\phi(2^{j}x)\not=0$ and hence, in this case,
this assumption is superfluous.
\end{remark}

Using Lemma \ref{Le49}(i) and \cite[Proposition 1.1.6(b)]{G2}, and repeating the proof of \cite[(2.66)]{U},
we then obtain
the following pointwise estimate and we omit the details.

\begin{lemma}\label{Le67}
Let $\phi$ be a Schwartz function and, for any $x\in\rn\setminus\{\vec 0_n\}$,
there exists a $j\in \zz$ such that $\widehat\phi(2^{j}x)\not=0$.
Then, for any given $N_0\in\nn$ and $\gamma\in(0,\infty)$,
there exists a positive constant $C_{(N_0,\gamma,\phi)}$ such that,
for any $s\in[1,2]$, $a\in(0,N_0]$, $l\in\zz$,
$f\in\cs'(\rn)$ and $x\in\rn$,
\begin{equation*}
\lf[\lf(\phi_{2^{-l}s}^\ast f\r)_a(x)\r]^\gamma\le C_{(N_0,\gamma,\phi
)}\sum_{k=0}^\infty2^{-kN_0\gamma}2^{(k+l)n}
\int_\rn\frac{|(\phi_{2^{-(k+l)}})_s\ast f(y)|^\gamma}{(1+2^l|x-y|)^{a\gamma}}\,dy.
\end{equation*}
\end{lemma}

Now, using Lemma \ref{Le67}, we can prove Theorem \ref{Thgf}.

\begin{proof}[Proof of Theorem \ref{Thgf}]
For any $f\in WH_X(\rn)$, from Lemma \ref{Le64} and Proposition \ref{Pro66}, it follows that
$f\in\cs'(\rn)$, $f$ vanishes weakly at infinity and $g(f)\in WX$. Thus, to finish the proof of Theorem \ref{Thgf},
by Theorem \ref{Tharea}, we only need to show that,
for any $f\in\cs'(\rn)$, which vanishes weakly at infinity,
such that $g(f)\in WX$, it holds true that
\begin{equation}\label{Eq620}
\|S(f)\|_{WX}\lesssim\|g(f)\|_{WX}.
\end{equation}
Observe that $S(f)(x)\lesssim g_{a,\ast}(f)(x)$ holds true for almost every $x\in\rn$
and any given $a\in(0,\infty)$.
Thus, to prove \eqref{Eq620}, it suffices to show
that there exists some $a\in(\frac{n}{\min\{p_-,2\}},\infty)$
such that, for any $f\in\cs'(\rn)$ vanishing weakly at infinity and satisfying $g(f)\in WX$,
\begin{equation}\label{Eq621}
\lf\|g_{a,\ast}(f)\r\|_{WX}\lesssim\|g(f)\|_{WX}.
\end{equation}
To this end, we first observe that, for any given $a\in(\frac{n}{\min\{p_-,2\}},\infty)$,
there exists a $\gamma\in(0,\min\{p_-,2\})$ such that
$a\in(\frac{n}{\gamma},\infty)$. Choosing $N_0$ sufficiently large, then, from Lemma \ref{Le67}
and the Minkowski integral inequality, we deduce that, for any $x\in\rn$,
\begin{align*}
g_{a,\ast}(f)(x)&=\lf\{\sum_{j\in\zz}\int_1^2\lf[\lf(\phi_{2^{-j}t}^\ast f\r)_a(x)\r]^2\,\frac{dt}{t}\r\}^{\frac12}\\
&\lesssim\lf\{\sum_{j\in\zz}\int_1^2\lf[\sum_{k=0}
^\infty2^{-kN_0\gamma}2^{(k+j)n}\int_\rn\frac{|(\phi_{2^{-(k+j)}})_t\ast f(y)|^\gamma}{(1+2^j|x-y|)^{a\gamma}}\,dy\r]
^\frac2{\gamma}\,\frac{dt}{t}\r\}^{\frac12}\\
&\lesssim\lf[\sum_{j\in\zz}\lf\{\sum_{k=0}^\infty2^{-kN_0\gamma}
2^{(k+j)n}\int_\rn\frac{[\int_1^2|(\phi_{2^{-(k+j)}})_t\ast f(y)|^2\,\frac{dt}{t}]^\frac{\gamma}{2}}{(1+2^{j}|x-y|)
^{a\gamma}}\,dy\r\}^\frac{2}{\gamma}\r]^\frac12,
\end{align*}
which, combined with Remarks \ref{Rews}(i) and \ref{Re213}, further implies that
\begin{align*}
\|g_{a,\ast}(f)\|_{WX}^{\gamma \nu}&\lesssim\lf\|\sum_{k=0}^\infty2^{-k(N_0\gamma-n)}\lf[\sum_{j\in\zz}2^{j\frac{2n}{\gamma}}
\lf\{\int_\rn\frac{[\int_1^2|(\phi_{2^{-(k+j)}})_t\ast f(y)|^2\,\frac{dt}{t}
]^\frac{\gamma}{2}}{(1+2^{j}|\cdot-y|)^{a\gamma}}\,dy\r\}^\frac{2}{\gamma}
\r]^\frac{\gamma}2\r\|_{(WX)^{1/\gamma}}^\nu\\
&\lesssim\sum_{k=0}^\infty2^{-k\nu(N_0\gamma-n)}\lf\|\lf[\sum_{j\in\zz}2^{j\frac{2n}{\gamma}}\lf\{
\int_\rn\frac{[\int_1^2|(\phi_{2^{-(k+j)}})_t\ast f(y)|^2\,\frac{dt}{t}
]^\frac{\gamma}{2}}{(1+2^{j}|\cdot-y|)^{a\gamma}}\,dy\r\}^\frac{2}{\gamma}
\r]^\frac{\gamma}2\r\|_{(WX)^{1/\gamma}}^\nu\\
&\lesssim\sum_{k=0}^\infty2^{-k\nu(N_0\gamma-n)}\\
&\hs\times\lf\|\lf[\sum_{j\in\zz}2^{j\frac{2n}{\gamma}}\lf\{\sum_{i=0}^\infty2^{-ia\gamma}
\int_{|\cdot-y|\sim2^{i-j}}\lf[\int_1^2|
(\phi_{2^{-(k+j)}})_t\ast f(y)|^2\,\frac{dt}{t}\r]^\frac{\gamma}{2}\,dy
\r\}^\frac2{\gamma}\r]^\frac{\gamma}{2}\r\|_{(WX)^{1/\gamma}}^\nu,
\end{align*}
where $\nu$ is as in Remark \ref{Re213} and $|\cdot-y|\sim2^{i-j}$ means that $|x-y|<2^{-j}$ when $i=0$,
or $2^{i-j-1}\le|x-y|<2^{i-j}$ when $i\in\nn$.
Using the Minkowski inequality on series, Assumption \ref{a2.17} and
Remark \ref{Re213}, we conclude that
\begin{align*}
\lf\|g_{a,\ast}(f)\r\|_{WX}^{\gamma \nu}&\lesssim\sum_{k=0}^\infty2^{-k\nu(N_0\gamma-n)}
\lf\|\sum_{i=0}^\infty2^{-ia\gamma+in}\lf\{\sum_{j\in\zz}
\lf[\cm\lf(\lf[\int_1^2|(\phi_{2^{-(k+j)}})_t\ast f(y)|^2
\,\frac{dt}{t}\r]^\frac{\gamma}{2}\r)\r]^\frac{2}
{\gamma}\r\}^\frac{\gamma}{2}\r\|_{(WX)^{1/\gamma}}^\nu\\
&\lesssim\sum_{k=0}^\infty2^{-k\nu(N_0\gamma-n)}
\sum_{i=0}^\infty2^{(-ia\gamma+in)\nu}\lf\|\lf\{\sum_{j\in\zz}\lf[\int_1^2
|(\phi_{2^{-(k+j)}})_t\ast f|^2\,\frac{dt}{t}\r]\r\}
^\frac{\gamma}{2}\r\|_{(WX)^{1/\gamma}}^\nu\\
&\lesssim\lf\|g(f)\r\|_{WX}^{\gamma \nu},
\end{align*}
which implies that \eqref{Eq621} holds true. This finishes the proof of Theorem \ref{Thgf}.
\end{proof}

\subsection{Characterization by the Littlewood--Paley $g_\lambda^\ast$-function}\label{s3.3}

In this subsection, we establish the Littlewood--Paley $g_\lambda^\ast$-function characterization
of $WH_X(\rn)$.

\begin{theorem}\label{Thgx}
Let $X$, $p_+$ and $p_-$ satisfy all assumptions in Theorem \ref{Tharea}
and
$$
\lambda\in\lf(\max\lf\{\frac{2}{\min\{1,p_-\}},
1-\frac2{\max\{1,p_+\}}+\frac{2}{\min\{1,p_-\}}\r\},\infty\r).
$$
Then $f\in WH_X(\rn)$ if and only if $f\in\cs'(\rn)$,
$f$ vanishes weakly at infinity and $g_\lambda^\ast(f)\in WX$, where $g_\lambda^\ast(f)$
is as in \eqref{eq62}.
Moreover, there exists a positive constant $C$ such that,
for any $f\in WH_X(\rn)$,
$$
C^{-1}\lf\|g_\lambda^\ast(f)\r\|_{WX}\le\|f\|_{WH_X(\rn)}\le C\lf\|g_\lambda^\ast(f)\r\|_{WX}.
$$
\end{theorem}

To prove Theorem \ref{Thgx}, we need the following inequality on the change of angles.

\begin{theorem}\label{Thca}
Let $X$ be a ball quasi-Banach function space
satisfying Assumption \ref{a2.15} for some $p_-\in(0,\infty)$.
Assume that, for any given $r\in(0,\underline{p})$,
$X^{1/r}$ is a ball Banach function space and assume that
there exist $r_0\in(0,\underline{p})$ and $p_0\in(r_0,\infty)$ such that \eqref{Eqdm} holds true.
Let $q\in(\max\{1,p_0\},p_0/r_0]$ and $r\in(0,1)$.
Then there exists a positive constant $C$ such that,
for any $\alpha\in[1,\infty)$ and $f\in WT_X^1(\rr_+^{n+1})$,
$$
\lf\|\ca^{(\alpha)}(f)\r\|_{WX}\le C\lf[\max\lf\{\alpha^{(\frac12-\frac1q)n},1\r\}\r]
^{\frac q{p_0}}\alpha^\frac{n}{r_0r}\lf\|\ca^{(1)}(f)\r\|_{WX},
$$
where $\ca^{(\alpha)}$
is as in \eqref{aa}.
\end{theorem}
\begin{proof}
Let $a$ be a $(T^1_X,\infty)$-atom supported in $T(B)$, where $B$ is a ball of $\rn$. Then,
from \cite[Theorem 1.1]{A}, we deduce that,
for any given $p\in (1,\infty)$, there exists a
positive constant $\widetilde C$, independent of $a$,
such that, for any $\alpha\in[1,\infty)$,
\begin{equation}\label{Eq617}
\lf\|\ca^{(\alpha)}(a)\r\|_{L^p(\rn)}\le\widetilde{C}
\max\lf\{\alpha^\frac{n}{2},\alpha^\frac{n}{p}\r\}\lf\|\ca^{(1)}(a)\r\|_{L^p(\rn)}.
\end{equation}

Since $f\in WT_X^1(\rr_+^{n+1})$, from Theorem \ref{That},
it follows that there exist $\{\lambda_{i,j}\}_{i\in\zz,j\in\nn}\subset[0,\infty)$ and a sequence $\{a_{i,j}\}_{i\in\zz,j\in\nn}$ of $(T_X,\infty)$-atoms associated, respectively, with balls $\{B_{i,j}\}_{i\in\zz,j\in\nn}$
such that (i), (ii) and (iii) of Theorem \ref{That} hold true.
To prove Theorem \ref{Thca}, by the definition of $WX$, Definition
\ref{Dewt} and Theorem \ref{That}, it suffices to show that
\begin{equation}\label{Eq619}
\sup_{\lambda\in(0,\infty)}
\lf\{\lambda\lf\|\mathbf1_{\{x\in\rn:\ \ca^{(\alpha)}(f)(x)>\lambda\}}\r\|_X\r\}
\lesssim \lf[\max\lf\{\alpha^{(\frac12-\frac1q)n},1\r\}\r]^{\frac q{p_0}}\alpha^\frac{n}{r_0r}
\sup_{i\in\zz}2^i\lf\|\sum_{j\in\nn}\mathbf1_{B_{i,j}}\r\|_X.
\end{equation}
Similarly to the proof of the sufficiency of Theorem \ref{That}, for any
given $\lambda\in(0,\infty)$, let $i_0\in\zz$
be such that $2^{i_0}\le\lambda<2^{i_0+1}$. Then we write
$$
f=\sum_{i=-\infty}^{i_0-1}\sum_{j\in\nn}\lambda_{i,j}a_{i,j}
+\sum_{i=i_0}^\infty\sum_{j\in\nn}\cdots=:f_1+f_2
$$
and, from Definition \ref{Debqfs}(ii), we deduce that
\begin{align*}
\lf\|\mathbf1_{\{x\in\rn:\ \ca^{(\alpha)}(f)(x)>\lambda\}}\r\|_X
\lesssim\lf\|\mathbf1_{\{x\in\rn:\ \ca^{(\alpha)}(f_1)(x)>\lambda/2\}}\r\|_X
+\lf\|\mathbf1_{\{x\in \rn:\ \ca^{(\alpha)}(f_2)(x)>\lambda/2\}}\r\|_X
=:\mathrm{II_1}+\mathrm{II_2}.
\end{align*}

For $\mathrm{II_1}$, let $\widetilde q:=q/p_0\in(\max\{1/p_0,1\},1/{r_0}]$
and $a\in(0,1-1/{\widetilde q})$. Then, by the H\"older inequality, we obtain
\begin{align*}
\sum_{i=-\infty}^{i_0-1}\sum_{j\in\nn}\lambda_{i,j}\ca^{(\alpha)}(a_{i,j})
\le\frac{2^{i_0a}}{(2^{a\widetilde q'}-1)^{1/\widetilde q'}}
\lf\{\sum_{i=-\infty}^{i_0-1}2^{-ia\widetilde q}\lf[\sum_{j\in\nn}
\lambda_{i,j}\ca^{(\alpha)}(a_{i,j})\r]^{\widetilde q}\r\}^{1/\widetilde q},
\end{align*}
where $\widetilde q':={\widetilde q}/{(\widetilde q-1)}$. From this, Definition
 \ref{Debf}(i), $\widetilde qr_0\in(0,1]$, Lemma \ref{Le65}
and the assumption that $X^{1/r_0}$ is a ball Banach function space,
we deduce that
\begin{align*}
\mathrm{II_1}
&\lesssim\lf\|\mathbf{1}_{\{x\in\rn:\ 2^{i_0a}
\{\sum_{i=-\infty}^{i_0-1}2^{-ia\widetilde q}[\sum_{j\in\nn}
\lambda_{i,j}\ca^{(\alpha)}(a_{i,j})]^
{\widetilde q}\}^{1/\widetilde q}>2^{i_0-2}\}}\r\|_{X}\\
&\lesssim2^{-i_0\widetilde q(1-a)}\lf\|\sum_{i=-\infty}^
{i_0-1}2^{-ia\widetilde q}\lf[\sum_{j\in\nn}
\lambda_{i,j}\ca^{(\alpha)}(a_{i,j})\r]^{\widetilde q}\r\|_{X}\\
&\lesssim2^{-i_0\widetilde q(1-a)}
\lf\|\sum_{i=-\infty}^{i_0-1}2^{(1-a)i\widetilde qr_0}\sum_{j\in\nn}\lf[
\lf\|\mathbf{1}_{B_{i,j}}\r\|_{X}\ca^{(\alpha)}(a_{i,j})\r]
^{\widetilde qr_0}\r\|_{X^{1/r_0}}^\frac{1}{r_0}\\
&\lesssim2^{-i_0\widetilde q(1-a)}\lf[\sum_{i=-\infty}^{i_0-1}2^{(1-a)i\widetilde qr_0}
\lf\|\lf\{\sum_{j\in\nn}\lf[
\lf\|\mathbf{1}_{B_{i,j}}\r\|_{X}\ca^{(\alpha)}(a_{i,j})\r]^{\widetilde qr_0}\r\}^\frac1{r_0}\r\|_{X}^{r_0}\r]^\frac{1}{r_0}.
\end{align*}

By $q=p_0\widetilde q\in(1,\infty)$, Lemma \ref{Lezj} and \eqref{Eq617},
we conclude that, for any $i\in\zz$ and $j\in\nn$, $\supp(\ca^{(\alpha)}(a_{i,j}))\subset \alpha B_{i,j}$ and $$\lf\|\ca^{(\alpha)}(a_{i,j})\r\|_{L^q(\rn)}\le\max\lf\{\alpha^{\frac{n}{2}-\frac{n}{q}},1\r\}
\frac{|\alpha B_{i,j}|^\frac1q}{\|\mathbf1_{B_{i,j}}\|_X}.$$
Then we find that, for any $i\in\zz$ and $j\in\nn$,
$$
\lf\|\lf[\lf\|\mathbf1_{B_{i,j}}\r\|_X\ca^{(\alpha)}(a_{i,j})\r]^{\widetilde q}\r\|_{L^{p_0}(\rn)}
\lesssim\lf\|\mathbf1_{B_{i,j}}\r\|_X^{\widetilde q}
\lf\|\ca^{(\alpha)}(a_{i,j})\r\|_{L^q(\rn)}^{\widetilde q}\lesssim
\lf[\max\lf\{\alpha^{\frac{n}{2}-\frac{n}{q}},1\r\}\r]^{\widetilde q}|\alpha B_{i,j}|^{1/p_0},
$$
which, combined with Lemma \ref{Le45}, $r\in(0,1)$,
\eqref{EqHLMS} and $(1-a)\widetilde q>1$, further implies that
\begin{align*}
\mathrm{II_{1}}&\lesssim\lf[\max\lf\{\alpha^{\frac{n}{2}-\frac{n}{q}},1\r\}
\r]^{\widetilde q}
2^{-i_0\widetilde q(1-a)}
\lf[\sum_{i=-\infty}^{i_0-1}2^{(1-a)i
\widetilde qr_0}\lf\|\lf(\sum_{j\in\nn}
\mathbf{1}_{\alpha B_{i,j}}\r)^\frac1{r_0}\r\|_{X}^{r_0}\r]^{\frac1{r_0}}\\
&\lesssim\lf[\max\lf\{\alpha^{\frac{n}{2}-\frac{n}{q}},1\r\}
\r]^{\widetilde q}
\alpha^{\frac{n}{r_0r}}2^{-i_0\widetilde q(1-a)}
\lf[\sum_{i=-\infty}^{i_0-1}2^{(1-a)i\widetilde qr_0}\lf\|\lf(\sum_{j\in\nn}\mathbf{1}_{cB_{i,j}}\r)^\frac1{r_0}\r\|_{X}^{r_0}\r]^{\frac1{r_0}}\\
&\lesssim\lf[\max\lf\{\alpha^{\frac{n}{2}-\frac{n}{q}},1\r\}
\r]^{\widetilde q}\alpha^{\frac{n}{r_0r}}2^{-i_0\widetilde q(1-a)}
\lf[\sum_{i=-\infty}^{i_0-1}2^{[(1-a)\widetilde
q-1]ir_0}\r]^{\frac1{r_0}}\sup_{i\in\zz}2^i\lf\|\sum_{j\in\nn}
\mathbf{1}_{B_{i,j}}\r\|_{X}\\
&\lesssim\lf[\max\lf\{\alpha^{\frac{n}{2}-\frac{n}{q}},1\r\}
\r]^{\widetilde q}\alpha^{\frac{n}{r_0r}}\lambda^{-1}
\sup_{i\in\zz}2^i\lf\|\sum_{j\in\nn}\mathbf{1}_{B_{i,j}}\r\|_{X}.
\end{align*}
This shows that
\begin{equation*}
\lambda\mathrm{II_{1}}\lesssim
\lf[\max\lf\{\alpha^{\frac{n}{2}-\frac{n}{q}},1\r\}
\r]^{q/p_0}
\alpha^{\frac{n}{r_0r}}\sup_{i\in\zz}2^i\lf\|\sum_{j\in\nn}\mathbf{1}_{B_{i,j}}\r\|_{X}.
\end{equation*}

Using Lemma \ref{Lezj} again and an argument similar to the
estimation of \eqref{Eq415}, we also obtain
$$
\mathrm{II_2}\lesssim\lf\|\mathbf1_{\{x\in\rn:\ \sum_{i=i_0}
^{\infty}\sum_{j\in\nn}\lambda_{i,j}\ca^{(\alpha)}
(a_{i,j})(x)>\lambda\}}\r\|_X
\lesssim\lf\|\mathbf1_{\cup_{i=i_0}^\infty\cup_{j\in\nn}\alpha B_{i,j}}\r\|_X
\lesssim\alpha^{\frac{n}{r_0}}\lambda^{-1}\sup_{i\in\zz}2^i\lf\|
\sum_{j\in\nn}\mathbf1_{B_{i,j}}\r\|_X.
$$

Combining the estimates for $\mathrm{II_1}$ and $\mathrm{II_2}$,
we conclude that \eqref{Eq619} holds true and hence complete
the proof of Theorem \ref{Thca}.
\end{proof}

Now we prove Theorem \ref{Thgx}.
\begin{proof}[Proof of Theorem \ref{Thgx}]
By Theorem \ref{Tharea} and the fact that, for any $f\in\cs'(\rn)$,
$S(f)\le g_\lambda^\ast(f)$,
we easily obtain the sufficiency of Theorem \ref{Thgx} and still
need to show the necessity.

To this end, let $f\in WH_X(\rn)$.
By Lemma \ref{Le64}, we know that $f$
vanishes weakly at infinity. Moreover, for any $x\in\rn$, we have
\begin{align*}
g_\lambda^*(f)(x)&\le\lf\{\int_0^{\infty}\int_{|x-y|<t}\lf(\frac{t}{t+|x-y|}\r)
^{\lambda n}|f\ast\phi_t(y)|^{2}\,\frac{dy\,dt}{t^{n+1}}
+\sum_{m=0}^\infty\int_0^{\infty}\int_{2^mt\le|x-y|<2^{m+1}t}
\cdots\r\}^{\frac{1}{2}}\\
&\le\ca^{(1)}\lf(F\r)(x)+\sum_{m=0}^\infty2^{-\frac{\lambda nm}{2}}
\ca^{(2^{m+1})}\lf(F\r)(x),
\end{align*}
where $F(x,t):=f\ast\phi_t(x)$ for any $x\in\rn$ and $t\in(0,\infty)$.

By the fact  $\lambda\in(\max\{\frac{2}{\min\{1,p_-\}},1-\frac2{\max\{1,p_+\}}+\frac{2}{\min\{1,p_-\}}\},\infty)$,
we conclude that there exist $r_0\in(0,\min\{1,p_-\})$ and $p_0\in(\max\{1,p_+\},\infty)$
such that \eqref{Eqdm} holds true and $\lambda\in(\max\{\frac{2}{r_0},
1-\frac2{p_0}+\frac{2}{r_0}\},\infty)$.
Then there exist $r\in(0,1)$ and $q\in(p_0,\infty)$ such that
$
\lambda\in(\max\{\frac{2}{r_0r}, 2(\frac12-\frac1q)\frac q{p_0}
+\frac{2}{r_0r}\},\infty).
$
By this, Remark \ref{Re213}, Definition \ref{Debf}(i),
Theorems \ref{Thca} and \ref{Tharea}, we choose a $\nu\in(0,1)$
to conclude that
\begin{align*}
\lf\|g_\lambda^*(f)\r\|_{WX}^\nu&=\lf\|\lf[g_\lambda^*(f)\r]^\nu\r\|_{(WX)^{1/\nu}}
\lesssim\lf\|\lf[\ca^{(1)}\lf(F\r)\r]^\nu\r\|_{(WX)^{1/\nu}}
+\sum_{m=0}^\infty2^{-\frac{\lambda nms}{2}}\lf\|\lf[\ca^{(2^{m+1})}
\lf(F\r)\r]^\nu\r\|_{(WX)^{1/\nu}}\\
&\lesssim\lf\|\ca^{(1)}\lf(F\r)\r\|_{WX}^\nu+\sum_{m=0}^\infty2^{-\frac{\lambda nm\nu}{2}}\lf[\max\{2^{\frac{mn\nu}{2}-\frac{mn\nu}{q}},1\}\r]^{\frac q{p_0}}
2^\frac{mn\nu}{r_0r}\lf\|\ca^{(1)}\lf(F\r)\r\|_{WX}^\nu\\
&\lesssim\|S(f)\|_{WX}^\nu\sim\|f\|_{WH_X(\rn)}^\nu,
\end{align*}
where, in the penultimate step,
we used the fact that $\ca^{(1)}\lf(F\r)=S(f)$.
This finishes the proof of Theorem \ref{Thgx}.
\end{proof}

\section{Real interpolation between $H_{X}(\rn)$ and $L^\infty(\rn)$}\label{s4}

In this section, we first obtain a
decomposition for any distribution of the weak Hardy space $WH_{X^{{1}/{(1-\theta)}}}(\rn)$ into
two parts and then we prove that the real interpolation intermediate space
$(H_{X}(\rn),L^\infty(\rn))_{\theta,\infty}$ between the
Hardy space $H_{X}(\rn)$
and the Lebesgue space $L^\infty(\rn)$ is the weak Hardy space $WH_{X^{{1}/{(1-\theta)}}}(\rn)
$, where $\theta\in (0, 1)$.

We first recall some basic notions about the
theory of the real interpolation (see, for instance, \cite{bl}).
For any two quasi-Banach spaces $X_0$ and $X_1$,
the pair $(X_0,\,X_1)$ is said to be \emph{compatible} if there exists a Hausdorff
space $\mathbb{X}$ such that $A_0\subset\mathbb{X}$ and $A_1\subset\mathbb{X}$.
For any compatible pair $(X_0,\,X_1)$ of quasi-Banach spaces, let
\begin{align*}
X_0+X_1:=\{a\in\mathbb{X}:\ \exists\, a_0\in A_0\ \textup{and}\ a_1\in A_1
\textup{such}\ \textup{that}\ a=a_0+a_1\}.
\end{align*}
For any $t\in(0,\infty)$, the \emph{Peetre K-functional} $K(t,f;X_0,X_1)$
on $X_0+X_1$ is defined by setting, for any $f\in X_0+X_1$,
\begin{equation*}
K(t,f;X_0,X_1):=\inf\lf\{\|f_0\|_{X_0}+t\|f_1\|_{X_1}:
\ f=f_0+f_1,\ f_i\in X_i,\ i\in\{0,1\}\r\}.
\end{equation*}
Then, for any $\theta\in(0,1)$ and $q\in(0,\infty]$, the \emph{real interpolation space}
$(X_0,\,X_1)_{\theta,q}$ between $X_0$ and $X_1$ is defined by setting
$$(X_0,\,X_1)_{\theta,q}:=\lf\{f\in X_0+X_1:\ \|f\|_{\theta,q}<\infty\r\},$$
where, for any $f\in X_0+X_1$,
\begin{equation}\label{222}
\|f\|_{\theta,q}:=
\begin{cases}
\displaystyle\lf[\int_0^\infty\{t^{-\theta}
K(t,f;X_0,X_1)\}^q\,\frac{dt}{t}\r]^{\frac{1}{q}}\ &\text{when}\ q\in(0,\infty),\\
\sup\limits_{t\in(0,\infty)}t^{-\theta}K(t,f;X_0,X_1)\ &\text{when}\ q=\infty.
\end{cases}
\end{equation}

By borrowing some ideas from the proof of \cite[Theorem 4.1]{kv}, we have the following
useful result,
which
plays a key role in the proof of Theorem \ref{inte} below
and itself is of independent interest.

\begin{theorem}\label{de1}
Let $X$ be a ball quasi-Banach function space and $\theta\in(0,1)$.
Assume that there exists some $r\in(0,\infty)$ such that $X^{1/r}$ is a ball Banach function space.
Then it holds true that
$$
(X,L^\infty(\rn))_{\theta,\infty}=(WX)^{{1}/{(1-\theta)}}.
$$
\end{theorem}

\begin{proof}
We first show that
\begin{equation}\label{yib0}
(WX)^{{1}/{(1-\theta)}}\subset(X,L^\infty(\rn))_{\theta,\infty}.
\end{equation}
To this end, let $f\in (WX)^{1/(1-\theta)}$.
For any $t\in(0,\infty)$, let
$$
\alpha(t):=\inf\lf\{\mu\in(0,\infty):\ \lf\{\sum_{j=0}
^{\infty}\lf[2^j\lf\|\mathbf1_{\{x\in\rn:\ f(x)>2^j\mu\}}\r\|_{X}
\r]^r\r\}^{\frac{1}{r}}\le t\r\}.
$$
We claim that, for any $t\in(0,\infty)$,
\begin{equation}\label{yibia}
K(t,f;X,L^\infty(\rn))\lesssim t\alpha(t).
\end{equation}
Indeed, by Definition \ref{Debf}(i) and
the assumption that $X^{1/r}$ is a ball Banach function space, we conclude that,
for any $t\in(0,\infty)$,
\begin{align}\label{88ab}
&K(t,f;X,L^\infty(\rn))\\ \noz
&\quad=\inf\lf\{\|f_0\|_{X}+t\|f_1\|_{L^\infty(\rn)}:
\ f=f_0+f_1,\ f_0\in X,\ f_1\in L^\infty(\rn)\r\}\\ \noz
&\quad=\inf_{\alpha\in(0,\infty)}\lf\{\lf\| \max\{|f|-\alpha,0\}\r\|_{X}
+t\lf\|\min\{|f|,\alpha\}\r\|_{L^\infty(\rn)}\r\}\\ \noz
&\quad\lesssim\inf_{\alpha\in(0,\infty)}\lf\{\lf\|f\mathbf1_{\{x\in\rn:
\ |f(x)|>\alpha\}}\r\|_{X}+t\alpha\r\}\\ \noz
&\quad\lesssim\inf_{\alpha\in(0,\infty)}\lf\{\lf[ \sum_{i=0}^{\infty}\lf\| |f|^r
\mathbf1_{\{x\in\rn:\ 2^{i}\alpha<|f(x)|\le2^{i+1}\alpha\}}
\r\|_{X^{1/r}} \r]^{1/r}+t\alpha\r\}\\ \noz
&\quad\lesssim\inf_{\alpha\in(0,\infty)}\lf\{\lf[ \sum_{i=0}^{\infty}(2^{i}\alpha)^{r}\lf\|
\mathbf1_{\{x\in\rn:\ |f(x)|>2^{i}\alpha\}} \r\|_{X}^r \r]^{1/r}+t\alpha\r\}.
\end{align}
Let $\alpha:=\alpha(t)$. Then it is easy to know that
$$
\lf\{ \sum_{i=0}^{\infty}2^{ir}\lf\|
\mathbf1_{\{x\in\rn:\ |f(x)|>2^{i}\alpha(t)\}} \r\|_{X}^r \r\}^{1/r}\le t.
$$
From this and \eqref{88ab}, we deduce that the claim \eqref{yibia} holds true.
Using this claim and the fact that, for any $k\in\zz$ and $t\in(0,\infty)$ satisfying
$2^k<\alpha(t)$,
$$
\lf\{ \sum_{i=0}^{\infty}2^{ir}\lf\|
\mathbf1_{\{x\in\rn:\ |f(x)|>2^{i}2^k\}} \r\|_{X}^r \r\}^{1/r}> t,
$$
 we conclude that
\begin{align}\label{ll60}
\sup_{t\in(0,\infty)}t^{-\theta}K(t,f;X,L^\infty(\rn))
&\lesssim\sup_{t\in(0,\infty)}t^{1-\theta}\alpha(t)\sim\sup_{k\in\zz}
\dsup_{\gfz{t\in(0,\infty)}{2^k<\alpha(t)\le2^{k+1}}}t^{1-\theta}2^k\\ \noz
&\lesssim\sup_{k\in\zz}
\lf\{ \sum_{i=0}^{\infty}2^{ir}\lf\|
\mathbf1_{\{x\in\rn:\ |f(x)|>2^{i}2^k\}} \r\|_{X}^r \r\}^{\frac{1-\theta}{r}}2^k=:\mathrm{I}.
\end{align}

To estimate $\mathrm{I}$, we consider the following two cases.
If $\frac{1-\theta}{r}\in(0,1]$, we have
\begin{align}\label{ll70}
\mathrm{I}&\lesssim\sup_{k\in\zz}
\sum_{i=0}^{\infty}2^{i(1-\theta)}\lf\|
\mathbf1_{\{x\in\rn:\ |f(x)|>2^{i}2^k\}} \r\|_{X}^{1-\theta}2^k
\lesssim\sup_{k\in\zz}
\sum_{i=0}^{\infty}2^{i+k}2^{-\theta i}\lf\|
\mathbf1_{\{x\in\rn:\ |f(x)|>2^{i+k}\}}
\r\|_{X}^{1-\theta}\\\noz
&\lesssim
\sum_{i=0}^{\infty}2^{-\theta i}\sup_{l\in\zz}2^{l}\lf\|
\mathbf1_{\{x\in\rn:\ |f(x)|>2^{l}\}}
\r\|_{X}^{1-\theta}\sim
\sup_{l\in\zz}2^{l}\lf\|
\mathbf1_{\{x\in\rn:\ |f(x)|>2^{l}\}}
\r\|_{X^{1/(1-\theta)}}\lesssim\lf\|f
\r\|_{(WX)^{1/(1-\theta)}}.
\end{align}
If $\frac{1-\theta}{r}\in(1,\infty)$, then, by the H\"older inequality, we find that, for any
$\epsilon \in (0, \frac{\theta}{1-\theta})$,
\begin{align}\label{ll80}
\mathrm{I}&\lesssim\sup_{k\in\zz}2^k
\sum_{i=0}^{\infty}2^{i(1+\epsilon)(1-\theta)}\lf\|
\mathbf1_{\{x\in\rn:\ |f(x)|>2^{i}2^k\}} \r\|_{X}^{1-\theta}
\lesssim
\sup_{l\in\zz}2^{l}\lf\|
\mathbf1_{\{x\in\rn:\ |f(x)|>2^{l}\}}
\r\|_{X}^{1-\theta}\\ \noz
&\sim
\sup_{l\in\zz}2^{l}\lf\|
\mathbf1_{\{x\in\rn:\ |f(x)|>2^{l}\}}
\r\|_{X^{1/(1-\theta)}}\lesssim\lf\|f
\r\|_{(WX)^{1/(1-\theta)}}.
\end{align}
Thus, by \eqref{ll60}, \eqref{ll70} and \eqref{ll80}, we conclude that
\begin{align*}
\sup_{t\in(0,\infty)}t^{-\theta}K(t,f;H_{X}(\rn),L^\infty(\rn))
\lesssim\lf\|f
\r\|_{(WX)^{1/(1-\theta)}},
\end{align*}
which, together with \eqref{222}, implies
$f\in(X,L^\infty(\rn))_{\theta,\infty}$ and hence completes the proof of \eqref{yib0}.

Next, we show that
\begin{equation}\label{ydd0}
(WX,L^\infty(\rn))_{\theta,\infty}\subset (WX)^{1/(1-\theta)}.
\end{equation}
To this end, let $f\in (WX,L^\infty(\rn))_{\theta,\infty}$.
For any $t\in(0,\infty)$, define
$$
\beta(t):=\sup\lf\{\mu\in(0,\infty):\
\lf\|\mathbf1_{\{x\in\rn:\ |f(x)|>\mu\}}\r\|_{X}
\geq t\r\}.
$$
We claim that, for any $t\in(0,\infty)$,
\begin{equation}\label{yibib}
K(t,f;WX,L^\infty(\rn))\gtrsim t\beta(t).
\end{equation}
Indeed, for any $t\in(0,\infty)$,
\begin{align}\label{888a}
&K(t,f;WX,L^\infty(\rn))\\ \noz
&\quad=\inf\lf\{\|f_0\|_{WX}+t\|f_1\|_{L^\infty(\rn)}:
\ f=f_0+f_1,\ f_0\in WX,\ f_1\in L^\infty(\rn)\r\}\\ \noz
&\quad=\inf_{\alpha\in(0,\infty)}\lf\{\lf\|\max\{|f|-\alpha,0\}\r\|_{WX}
+t\lf\|\min\{|f|,\alpha\}\r\|_{L^\infty(\rn)}\r\}.
\end{align}
When $\alpha\in[\frac{\beta(t)}{2},\infty)$, it is easy to see that
\begin{align}\label{888b}
\lf\|\max\{|f|-\alpha,0\}\r\|_{WX}
+t\lf\|\min\{|f|,\alpha\}\r\|_{L^\infty(\rn)}\geq
t\lf\|\min\{|f|,\alpha\}\r\|_{L^\infty(\rn)}\geq\frac{t\beta(t)}{2}.
\end{align}
In the case of $\alpha\in(0,\frac{\beta(t)}{2})$,
by the definition of $\beta(t)$, we know that
\begin{align*}
\lf\|\max\{|f|-\alpha,0\}\r\|_{WX}
+t\lf\|\min\{|f|,\alpha\}\r\|_{L^\infty(\rn)}&\geq
\lf\|\max\{|f|-\alpha,0\}\r\|_{WX}\\ \noz
&\geq\frac{\beta(t)}{2}\lf\|
\mathbf1_{\{x\in\rn:\ |f(x)|>\alpha+\frac{\beta(t)}{2}\}} \r\|_{X}\geq\frac{t\beta(t)}{2}.
\end{align*}
From this, \eqref{888a} and \eqref{888b}, we deduce that the claim \eqref{yibib} holds true.
Using this claim,
 we find that
\begin{align}\label{llzz}
\sup_{t\in(0,\infty)}t^{-\theta}K(t,f;WX,L^\infty(\rn))
&\gtrsim\sup_{t\in(0,\infty)}t^{1-\theta}\beta(t)\sim
\sup_{k\in\zz}
\dsup_{\gfz{t\in(0,\infty)}{2^{k-1}<t\le2^{k}}}t^{1-\theta}\beta(t)\\ \noz
&\gtrsim\sup_{k\in\zz}
2^{k(1-\theta)}\beta(2^{k}).
\end{align}
For any $\alpha\in(0,\infty)$, let $h(\alpha):=\|\mathbf1_{\{x\in\rn:\ |f(x)|>\alpha\}}\|_{X}$.
Then, by \eqref{llzz}, we obtain
\begin{align*}
\|f\|_{(WX)^{1/(1-\theta)}}&=\sup_{\alpha\in(0,\infty)}
\alpha [h(\alpha)]^{1-\theta}
\lesssim\sup_{k\in\zz}
\dsup_{\gfz{\alpha\in(0,\infty)}{2^{k}<
h(\alpha)\le2^{k+1}}}\alpha [h(\alpha)]^{1-\theta}\\\noz
&\lesssim\sup_{k\in\zz}2^{k(1-\theta)}\beta(2^{k})\lesssim
\sup_{t\in(0,\infty)}t^{-\theta}K(t,f;WX,L^\infty(\rn)).
\end{align*}
This, combined with \eqref{222}, further implies that
$\|f\|_{(WX)^{1/(1-\theta)}}\lesssim\|f\|_{(WX,L^\infty(\rn))_{\theta,\infty}}$
and hence finishes the proof of \eqref{ydd0}.
Finally, by \eqref{yib0} and \eqref{ydd0}, we conclude that
\begin{align*}
\|\cdot\|_{(X,L^\infty(\rn))_{\theta,\infty}}\sim\|\cdot\|_{(WX)^{1/(1-\theta)}},
\end{align*}
which completes the proof of Theorem \ref{de1}.
\end{proof}

We now recall the notion
of the Hardy type space $H_X(\rn)$ (see \cite[Definition 2.22]{SHYY}).

\begin{definition}\label{DeHX}
Let $X$ be a ball quasi-Banach function space.
The \emph{Hardy space} $H_X(\rn)$ associated with $X$ is defined to be the set of all $f\in\cs'(\rn)$ such that
$$
\lf\|f\r\|_{H_X(\rn)}:=\lf\|M_N^0(f)\r\|_{X}<\infty,
$$
where $M_N^0(f)$ is as in \eqref{EqMN0} with $N\in\nn$ sufficiently large.
\end{definition}

Now, we establish a decomposition of any distribution
$f\in WH_{X^{{1}/{(1-\theta)}}}(\rn)$ into two parts in Lemma \ref{de}, which plays a key
role in the proof of Theorem \ref{inte} below.
We begin with some
notation.

Let $X$ be a ball quasi-Banach function space satisfying Assumption \ref{a2.15} for some $p_-\in(0,\infty)$.
Assume that there exist both $\vartheta_0\in(1,\infty)$ such that $X$ is $\vartheta_0$-concave
and $r_1\in(0,\infty)$
such that $\cm$ in \eqref{mm} is bounded on $(WX)^{1/r_1}$.
Let $d\geq\lfloor n(1/\min\{\frac{p_-}{\vartheta_0},r_0\}-1)\rfloor$.
Assume that
there exist $r_0\in(0,\underline{p})$ and $p_0\in(r_0,\infty)$
such that $X^{1/r_0}$ is a ball Banach function space and \eqref{Eqdm} holds true.
Assume that $\psi\in\cs(\rn)$ satisfies $\supp\psi\subset B(\vec 0_n,1)$ and $\int_\rn\psi(x)x^\gamma\,dx=0$
for any $\gamma\in\zz_+^n$ with $|\gamma|\le d$.
Then, by Lemma \ref{Le47},
we know that there exists a
$\phi\in\cs(\rn)$ such that the support of $\widehat\phi$ is compact and away from the origin and,
for any $x\in\rn\setminus\{\vec 0_n\}$,
$$
\int_0^\infty\widehat\psi(tx)\widehat\phi(tx)\,\frac{dt}{t}=1.
$$

Let $\eta$ be a function on $\rn$ such that
$\widehat\eta(\vec 0_n):=1$ and, for any $x\in\rn\setminus\{\vec 0_n\}$,
$$
\widehat\eta(x):=\int_1^\infty\widehat\psi(tx)\widehat\phi(tx)\,\frac{dt}{t}.
$$
Then, by \cite[p.\,219]{C1977}, we know that
such an $\eta$ exists and $\widehat\eta$ is infinitely differentiable,
has compact support and equals $1$ near the origin.
Moreover, for any $t_0,\ t_1 \in(0,\infty)$ and $x \in \rn\setminus\{\vec{0_n}\}$,
\begin{equation}\label{hard}
\int_{t_0}^{t_1}\widehat\psi(tx)\widehat\phi(tx)\,\frac{dt}{t}=
\widehat\eta(t_0x)-\widehat\eta(t_1x).
\end{equation}

Let $x_0:=(2,\ldots,2)\in\rn$ and  $f\in\cs'(\rn)$ vanish weakly at infinity.
For any $x\in\rn$ and $t\in(0,\infty)$, let
$\widetilde\phi(x):=\phi(x-x_0)$, $\widetilde\psi(x):=\psi(x+x_0)$, $F(x,t):=f\ast\widetilde\phi_t(x)$
and $G(x,t):=f\ast\eta_t(x)$.
Then, using \eqref{hard}, we have, for any $t_0,\ t_1\in(0,\infty)$ and
$x\in\rn\setminus\{\vec{0_n}\}$,
$$
\int_{t_0}^{t_1}\int_{\rn}F(y,t)\widetilde\psi(x-y)\,\frac{dy\,dt}{t}=G(x,t_0)-G(x,t_1)
$$
and, by Lemma \ref{Le47} and the fact that $f\in\cs'(\rn)$ vanishes weakly at infinity,
we know that
$$
f=\lim_{\substack{\epsilon\to0^+\\ A\to\infty}}
\int_\epsilon^A \int_{\rn}F(y,t)\widetilde\psi(x-y)\,\frac{dy\,dt}{t}\quad\text{in}\quad\cs'(\rn).
$$
Now, for any $x\in\rn$,
let
\begin{equation}\label{m47}
M_\triangledown(f)(x):=\sup_{\{t\in(0,\infty),\ |y-x|\le3(|x_0|+1)t\}}[|F(y,t)|+|G(y,t)|].
\end{equation}

\begin{lemma}\label{de}
Let $X$ be a ball quasi-Banach function space satisfying Assumption \ref{a2.15} for some $p_-\in(0,\infty)$.
Assume that there exist both $\vartheta_0\in(1,\infty)$ such that $X$ is $\vartheta_0$-concave
and $r_1\in(0,\infty)$
such that $\cm$ in \eqref{mm} is bounded on $(WX)^{1/r_1}$.
Assume that
there exist $r_0\in(0,\underline{p})$ and $p_0\in(r_0,\infty)$
such that $X^{1/r_0}$ is a ball Banach function space and \eqref{Eqdm} holds true.
Assume that $\theta\in(0,1)$. Let $\alpha\in(0,\infty)$
and $f\in WH_{X^{1/(1-\theta)}}(\rn)$.
Then there exist $g_\alpha\in L^\infty(\rn)$ and $b_\alpha\in\cs'(\rn)$
such that $f =g_\alpha + b_\alpha$ in $\cs'(\rn)$,
$\|g_\alpha\|_{L^\infty(\rn)}\le c_1\alpha$ and
\begin{equation}\label{pr8-1}
\|b_\alpha\|_{H_{X}(\rn)}\le c_2\|M_\triangledown(f)\mathbf1_{\{x\in\rn:\ M_\triangledown(f)(x)>\alpha\}}\|_{X}<\infty,
\end{equation}
where $M_\triangledown$ is the same as in \eqref{m47}, and
$c_1$ and $c_2$ are two positive constants independent of $f$ and $\alpha$.
\end{lemma}

\begin{proof}
Let $f \in WH_{X^{1/(1-\theta)}}(\rn)$ and, for any $i\in\zz$,
$$
\Omega_i:=\lf\{x\in\rn:\ M_\triangledown(f)(x)>2^i\r\}.
$$
Then $\Omega_i$ is open and,
by \cite[Theorem 3.2(ii)]{zwyy}, we further find that
$$\sup_{i\in\zz}2^i\|\mathbf1_{\Omega_i}\|_{X^{1/(1-\theta)}}
\le\|M_\triangledown(f)\|_{(WX)^{1/(1-\theta)}}\sim\|f\|_{WH_{X^{{1}/{(1-\theta)}}}(\rn)}.$$

Since $\Omega_i$ is a proper open subset of $\rn$, from the Whitney decomposition theorem (see, for instance, \cite[p.\,463]{G1}),
we deduce that there exists a sequence of cubes, $\{Q_{i,j}\}_{j\in\nn}$, such that, for any $i\in\zz$,
\begin{enumerate}
\item[(i)] $\bigcup_{j\in\nn} Q_{i,j}=\Omega_i$ and $\{Q_{i,j}\}_{j\in\nn}$ have disjoint interiors;
\item[(ii)] for any $j\in\nn$, $\sqrt nl_{Q_{i,j}}\le\dist(Q_{i,j},\Omega_i^\complement)
\le4\sqrt nl_{Q_{i,j}}$,
here and in the remainder of this proof, $l_{Q_{i,j}}$ denotes the \emph{side length} of the cube $Q_{i,j}$ and $\dist(Q_{i,j},\Omega_i^\complement)
:=\inf\{|x-y|:\ x\in Q_{i,j},\ y\in\Omega_i^\complement\}$;
\item[(iii)] for any $j,\ k\in\nn$, if the boundaries of two cubes $Q_{i,j}$ and $Q_{i,k}$ touch,
then $\frac14\le\frac{l_{Q_{i,j}}}{l_{Q_{i,k}}}\le4$;
\item[(iv)] for any given $j\in\nn$, there exist at most $12^n$ different cubes $\{Q_{i,k}\}_k$ that touch $Q_{i,j}$.
\end{enumerate}
For any $\epsilon\in(0,1)$, $i\in\zz$, $j\in\nn$ and $x\in\rn$, let
$$
\dist\lf(x,\Omega_i^\complement\r):=\inf\lf\{|x-y|:\ y\in\Omega_i\r\},
$$
$$
\widetilde\Omega_i:=\lf\{(x,t)\in\rr_+^{n+1}:=\rn\times(0,\infty):\
0<2t(|x_0|+1)<\dist\lf(x,\Omega_i^\complement\r)\r\},
$$
$$
\widetilde Q_{i,j}:=\lf\{(x,t)\in\rr_+^{n+1}:\ x\in Q_{i,j},\ (x,t)\in\widetilde\Omega_i\setminus\widetilde\Omega_{i+1}\r\}
$$
and
$$
b_{i,j}^\epsilon(x):=\int_\epsilon^{1/\epsilon}\int_\rn\mathbf{1}_{\widetilde Q_{i,j}}(y,t)F(y,t)\widetilde\psi_t(x-y)\,\frac{dy\,dt}{t}.
$$
For any $i\in\zz$, let
$$
\Omega_i^*:=\lf\{(y,t)\in\rr_+^{n+1}:\ y\in\Omega_i,\
(y,t)\in\widetilde\Omega_i\setminus\widetilde\Omega_{i+1}\r\}
$$
and, for any $\epsilon\in(0,\infty)$ and $x\in\rn$,
$$
b_i^\epsilon(x):=\int_\epsilon^{1/\epsilon}\int_\rn
\mathbf1_{\Omega^*_i}(y,t)F(y,t)\widetilde\psi_t(x-y)\,dy\,\frac{dt}{t}.
$$

Then, by the proof of \cite[Proposition 2.1]{zyy}, we know that, for any $\epsilon\in (0,\infty)$ and $i\in\zz$,
$
\|b_i^\epsilon\|_{L^\infty(\rn)}\lesssim2^i
$
with the implicit positive constant independent of $\epsilon$, $i$ and $f$,
and
there exist
$\{b_{i}\}_{i\in\zz}\subset L^\infty(\rn)$ and a sequence $\{\epsilon_k\}_{k\in\nn}\subset(0,\infty)$
such that $\epsilon_k\to0$ as $k\to\infty$ and, for any $i\in\zz$ and $g\in L^1(\rn)$,
\begin{equation}\label{ccc}
\lim_{k\to\infty}\langle b_{i}^{\epsilon_k},g\rangle=\langle b_{i},g\rangle,\ \|b_{i}\|_{L^\infty(\rn)}\lesssim2^i
\end{equation}
and
\begin{equation*}
\lim_{k\to\infty}\sum_{i\in\zz}b_{i}^{\epsilon_k}
=\sum_{i\in\zz}b_{i}\quad\text{in}\quad\cs'(\rn).
\end{equation*}

For any given $\alpha\in(0,\infty)$, let $i_0 \in\zz$ be such that $2^{i_0}\le\alpha<2^{i_0+1}$. Let
$$g_\alpha :=\sum_{i=-\infty}^{i_0}b_i.$$
Then, by \eqref{ccc} , we conclude that $g_\alpha \in L^\infty(\rn)$ and
$\|g_\alpha\|_{L^\infty(\rn)}\lesssim2^{i_0}\sim\alpha$.

Let $r_0$, $p_-$ and $\vartheta_0$ be as in this lemma.
Let $d\geq\lfloor n(1/\min\{\frac{p_-}{\vartheta_0},r_0\}-1)\rfloor$.
From $d\geq\lfloor n(\vartheta_0/p_--1)\rfloor$ and the fact that
$X$ satisfy
Assumption \ref{a2.15} for some $p_-\in(0,\infty)$, we deduce that there exists a
$p\in(0,p_-)$ such that $\cm$ in \eqref{mm} is bounded on
$X^{1/(\vartheta_0 p)}$ and $d\geq \lfloor n(1/p-1)\rfloor$.
By this and an argument similar to that used in the proof of \cite[Theorem 4.2]{zwyy},
we conclude that there exist
positive constants $C_1$ and $C_2$,
$\{b_{i,j}\}_{i>i_0,j\in\nn}\subset L^\infty(\rn)$ and a sequence
$\{\epsilon_{k_l}\}_{l\in\nn}\subset\{\epsilon_k\}_{k\in\nn}$
such that $k_l\to\infty$ as $l\to\infty$ and,
for any $i\in\zz\cap(i_0,\infty),\,j\in\nn$ and $g\in L^1(\rn)$,
\begin{equation*}
\lim_{l\to\infty}\langle b_{i,j}^{\epsilon_{k_l}},g\rangle=\langle b_{i,j},g\rangle,
\end{equation*}
$\supp b_{i,j}\subset C_1Q_{i,j}$, $\|b_{i,j}\|_{L^\infty(\rn)}\le C_22^i$ and,
for any $\gamma\in\zz_+^n$ with $|\gamma|\le d$, where $d\geq\lfloor n(\frac{1}{p}-1)\rfloor$,
$\int_\rn b_{i,j}(x)x^\gamma\,dx=0$. Furthermore,
\begin{equation}\label{qq}
\lim_{l\to\infty}\sum_{i=i_0+1}^{\infty}\sum_{j\in\nn}b_{i,j}^{\epsilon_{k_l}}
=\sum_{i=i_0+1}^{\infty}\sum_{j\in\nn}b_{i,j}\quad\text{in}\quad\cs'(\rn).
\end{equation}
Then we let $b_\alpha:=\sum_{i=i_0+1}^{\infty}\sum_{j\in\nn}b_{i,j}$
in $\cs'(\rn)$
and prove that
\eqref{pr8-1} holds true.
For any $i\in\zz\cap(i_0,\infty)$ and $j\in\nn$, let
$$
a_{i,j}:=\frac{b_{i,j}}{C_22^i\|\mathbf{1}_{B_{i,j}}\|_{X}}
\quad\text{and}\quad\lambda_{i,j}:=C_22^i\lf\|\mathbf{1}_{B_{i,j}}\r\|_{X}.
$$
Then, using the properties of $b_{i,j}$, we are easy to show that $a_{i,j}$ is
an $(X,\infty,d)$-atom.
By this, $d\geq\lfloor n(1/r_0-1)\rfloor$, the assumption on $X$ in this lemma
and \cite[Theorem 3.6]{SHYY}, we find that
\begin{align}\label{q13}
\|b_\alpha\|_{H_X(\rn)}\lesssim\lf\| \lf[\sum_{i=i_0+1}^{\infty}\sum_{j\in\nn}
\lf(\frac{\lambda_j}{\|\mathbf{1}_{C_1Q_{i,j}}\|_{X}}\r)^{r_0}
\mathbf{1}_{C_1Q_{i,j}}\r]^{1/r_0} \r\|_{X}.
\end{align}
From the fact that, for any given $s\in(0,r_0)$, $\mathbf{1}_{C_1Q_{i,j}}\lesssim
 [\cm(\mathbf{1}_{Q_{i,j}})]^{\frac1s}$, property (i) of the aforementioned Whitney
decomposition and Assumption \ref{a2.15}, we deduce that
\begin{align}\label{qqqq}
&\lf\| \lf[\sum_{i=i_0+1}^{\infty}\sum_{j\in\nn}
\lf(\frac{\lambda_j}{\|\mathbf{1}_{C_1Q_{i,j}}\|_{X}}\r)^{r_0}
\mathbf{1}_{C_1Q_{i,j}}\r]^{1/r_0} \r\|_{X}\\ \noz
&\quad\lesssim
\lf\| \lf[\sum_{i=i_0+1}^{\infty}\sum_{j\in\nn}
\lf(2^i\lf[\cm(\mathbf{1}_{Q_{i,j}})\r]^{\frac1s}\r)^{r_0}
\r]^{1/r_0} \r\|_{X}\lesssim
\lf\| \lf[\sum_{i=i_0+1}^{\infty}\sum_{j\in\nn}
2^{ir_0}\mathbf{1}_{Q_{i,j}}
\r]^{1/r_0} \r\|_{X}\\ \noz
&\quad\sim\lf\| \lf[\sum_{i=i_0+1}^{\infty}
2^{ir}\mathbf{1}_{\Omega_i}
\r]^{1/r_0} \r\|_{X}\sim\lf\| M_\triangledown(f)(x)
\mathbf1_{\{x\in\rn:\ M_\triangledown(f)(x)>\alpha\}} \r\|_{X}.
\end{align}
Moreover, by the definition of $WH_{X^{1/(1-\theta)}}(\rn)$, Definition \ref{Debf}(i),
the assumption that $X^{1/r_0}$ is a ball Banach function space and \cite[Theorem 3.2(ii)]{zwyy},
we conclude that, for any given $\theta\in(0,1)$ as in this theorem,
\begin{align}\label{q11}
&\lf\| M_\triangledown(f)
\mathbf1_{\{x\in\rn:\ M_\triangledown(f)(x)>\alpha\}} \r\|_{X}\\ \noz
&\quad\lesssim\lf\{ \sum_{i=0}^{\infty}\lf\| \lf[M_\triangledown(f)\r]^{r_0}
\mathbf1_{\{x\in\rn:\ 2^{i}\alpha<M_\triangledown(f)(x)\le2^{i+1}\alpha\}}
\r\|_{X^{1/r_0}} \r\}^{1/r_0}\\\noz
&\quad\lesssim
\lf\{ \sum_{i=0}^{\infty}(2^{i}\alpha)^{r_0}\lf\|
\mathbf1_{\{x\in\rn:\ M_\triangledown(f)(x)>2^{i}\alpha\}} \r\|_{X}^{r_0} \r\}^{1/r_0}\\\noz
&\quad\lesssim
\lf\{ \sum_{i=0}^{\infty}(2^{i}\alpha)^{-\frac{\theta r_0}{1-\theta}}\lf[2^{i}\alpha\lf\|
\mathbf1_{\{x\in\rn:\ M_\triangledown(f)(x)
>2^{i}\alpha\}} \r\|_{X^{1/(1-\theta)}}\r]^{\frac{r_0}{1-\theta}}
\r\}^{1/r_0}\\\noz
&\quad\lesssim\lf\|
 M_\triangledown(f)
 \r\|_{WX^{1/(1-\theta)}}^{\frac{1}{1-\theta}}
\lf\{ \sum_{i=0}^{\infty}(2^{i}\alpha)^{-\frac{\theta r_0}{1-\theta}}
\r\}^{1/r_0}
\sim\lf\|
f
 \r\|_{WH_{X^{1/(1-\theta)}}(\rn)}^{\frac{1}{1-\theta}}
\alpha^{-\frac{\theta }{1-\theta}}<\infty.
\end{align}
Combining \eqref{q13}, \eqref{qqqq} and \eqref{q11}, we know that
$b_\alpha\in H_{X}(\rn)$ and
$$\|b_\alpha\|_{H_{X}(\rn)}\le C_2\|M_\triangledown(f)\mathbf1_{\{x\in\rn:\ M_\triangledown(f)(x)>\alpha\}}\|_{X}<\infty.$$
This finishes the proof of \eqref{pr8-1}.

Finally, we prove that $f =g_\alpha + b_\alpha$ in $\cs'(\rn)$.
For any $\zeta\in\cs(\rn)$, by the Lebesgue dominated convergence theorem
and $\sum_{j\in\nn}\mathbf{1}_{\widetilde Q_{i,j}}=1$, we have
\begin{align*}
\lf\langle\sum_{j\in\nn}b_{i,j}^\epsilon,\zeta\r\rangle
&=\int_\rn\zeta(x)\sum_{j\in\nn}\int_\epsilon^{1/\epsilon}\int_\rn\mathbf{1}_{\widetilde Q_{i,j}}(y,t)F(y,t)
\widetilde\psi_t(x-y)\,\frac{dy\,dt}{t}\,dx\\
&=\int_\rn\zeta(x)\int_\epsilon^{1/\epsilon}\int_\rn \mathbf{1}_{\Omega_i^*}F(y,t)
\widetilde\psi_t(x-y)\,\frac{dy\,dt}{t}\,dx
\end{align*}
and hence
\begin{align}\label{2177}
b_{i}^\epsilon=\sum_{j\in\nn}b_{i,j}^\epsilon\quad\text{in}\quad\cs'(\rn).
\end{align}
Then, from \eqref{2177}, \eqref{qq} and
the same argument as that used in the estimation of \cite[(4.10)]{zwyy}, we
deduce that
\begin{align*}
f&=\lim_{l\to\infty}
\sum_{i\in\zz}\sum_{j\in\nn}b_{i,j}^{\epsilon_{k_l}}=\lim_{l\to\infty}
\lf[\sum_{i=-\infty}^{i_0}\sum_{j\in\nn}b_{i,j}^{\epsilon_{k_l}}+
\sum_{i=i_0+1}^{\infty}\sum_{j\in\nn}b_{i,j}^{\epsilon_{k_l}}\r]\\
&=\lim_{l\to\infty}\lf[
\sum_{i=-\infty}^{i_0}b_{i}^{\epsilon_{k_l}}+
\sum_{i=i_0+1}^{\infty}\sum_{j\in\nn}b_{i,j}^{\epsilon_{k_l}}\r]
=\sum_{i=-\infty}^{i_0}b_{i}+
\sum_{i=i_0+1}^{\infty}\sum_{j\in\nn}b_{i,j}=g_{\alpha}+b_{\alpha}
\quad\text{in}\quad\cs'(\rn).
\end{align*}
This finishes the proof of Lemma \ref{de}.
\end{proof}

\begin{theorem}\label{inte}
Let $X$ satisfy all the same assumptions as in Lemma \ref{de} and $\theta\in(0,1)$. Then it holds true that
$$
(H_{X}(\rn),L^\infty(\rn))_{\theta,\infty}=WH_{X^{1/(1-\theta)}}(\rn).
$$
\end{theorem}

\begin{proof}
We first show that
\begin{equation}\label{yib}
WH_{X^{1/(1-\theta)}}(\rn)\subset(H_{X}(\rn),L^\infty(\rn))_{\theta,\infty}.
\end{equation}
Let $f\in WH_{X^{1/(1-\theta)}}(\rn)$ and use the same notation
as in Lemma \ref{de}. Fix $r\in(0,\underline{p})$.
For any $t\in(0,\infty)$, let
$$
\alpha(t):=\inf\lf\{\mu\in(0,\infty):\ \lf\{\sum_{j=0}
^{\infty}\lf[2^j\lf\|\mathbf1_{\{x\in\rn:\ M_\triangledown(f)(x)>2^j\mu\}}\r\|_{X}
\r]^r\r\}^{\frac{1}{r}}\le t\r\}.
$$
We claim that, for any $t\in(0,\infty)$,
\begin{equation}\label{yibi}
K(t,f;H_{X}(\rn),L^\infty(\rn))\lesssim t\alpha(t).
\end{equation}
Indeed, by Lemma \ref{de}, we know that, for any $\alpha\in(0,\infty)$,
there exist $g_{\alpha}\in L^\infty(\rn)$
and $b_{\alpha}\in H_{X}(\rn)$
such that $f = g_{\alpha} + b_{\alpha}$ in $\cs'(\rn)$,
$\|g_\alpha\|_{L^\infty(\rn)}\le c_1\alpha$, and
$$
\|b_\alpha\|_{H_{X}(\rn)}\le c_2\|M_\triangledown(f)\mathbf1_{\{x\in\rn:\ M_\triangledown(f)(x)>\alpha\}}\|_{X}<\infty.
$$
Then, from this, Definition \ref{Debf}(i) and the assumption that $X^{1/r}$ is a ball Banach function space,
it follows that, for any $t\in(0,\infty)$,
\begin{align}\label{88}
&K(t,f;H_{X}(\rn),L^\infty(\rn))\\ \noz
&\quad=\inf\lf\{\|f_0\|_{H_{X}(\rn)}+t\|f_1\|_{L^\infty(\rn)}:
\ f=f_0+f_1,\ f_0\in H_{X}(\rn),\ f_1\in L^\infty(\rn)\r\}\\ \noz
&\quad\le\inf_{\alpha\in(0,\infty)}\lf\{\|b_{\alpha}\|_{H_{X}(\rn)}+t\|g_{\alpha}\|_{L^\infty(\rn)}:
\ g_{\alpha}\ \text{and}\ b_{\alpha}\ \text{are}\ \text{as}\ \text{in}
\ \text{Lemma}\ \ref{de}\r\}\\ \noz
&\quad\lesssim\inf_{\alpha\in(0,\infty)}\lf\{\|M_\triangledown(f)\mathbf1_{\{x\in\rn:\ M_\triangledown(f)(x)>\alpha\}}\|_{X}+t\alpha\r\}\\ \noz
&\quad\lesssim\inf_{\alpha\in(0,\infty)}\lf\{\lf[ \sum_{i=0}^{\infty}\lf\|
\lf\{M_\triangledown(f)\r\}^r
\mathbf1_{\{x\in\rn:\ 2^{i}\alpha<M_\triangledown(f)(x)\le2^{i+1}\alpha\}} \r\|_{X^{1/r}} \r]^{1/r}+t\alpha\r\}\\ \noz
&\quad\lesssim\inf_{\alpha\in(0,\infty)}\lf\{\lf[ \sum_{i=0}^{\infty}(2^{i}\alpha)^{r}\lf\|
\mathbf1_{\{x\in\rn:\ M_\triangledown(f)(x)>2^{i}\alpha\}} \r\|_{X}^r \r]^{1/r}+t\alpha\r\}.
\end{align}
Let $\alpha:=\alpha(t)$. Then it is easy to see that
$$
\lf\{ \sum_{i=0}^{\infty}2^{ir}\lf\|
\mathbf1_{\{x\in\rn:\ M_\triangledown(f)(x)>2^{i}\alpha(t)\}} \r\|_{X}^r \r\}^{1/r}\le t.
$$
By this and \eqref{88}, we conclude that the claim \eqref{yibi} holds true.
Using this claim and the fact that, for any $k\in\zz$ and $t\in(0,\infty)$ satisfying
$2^k<\alpha(t)$,
$$
\lf\{ \sum_{i=0}^{\infty}2^{ir}\lf\|
\mathbf1_{\{x\in\rn:\ M_\triangledown(f)(x)>2^{i}2^k\}} \r\|_{X}^r \r\}^{1/r}> t,
$$
 we find that
\begin{align}\label{ll5}
\sup_{t\in(0,\infty)}t^{-\theta}K(t,f;H_{X}(\rn),L^\infty(\rn))
&\lesssim\sup_{t\in(0,\infty)}t^{1-\theta}\alpha(t)\sim\sup_{k\in\zz}
\dsup_{\gfz{t\in(0,\infty)}{2^k<\alpha(t)\le2^{k+1}}}t^{1-\theta}2^k\\\noz
&\lesssim\sup_{k\in\zz}
\lf\{ \sum_{i=0}^{\infty}2^{ir}\lf\|
\mathbf1_{\{x\in\rn:\ M_\triangledown(f)(x)>2^{i}2^k\}} \r\|_{X}^r \r\}^{\frac{1-\theta}{r}}2^k=:\mathrm{I}.
\end{align}
If $\frac{1-\theta}{r}\le1$, we have
\begin{align}\label{ll4}
\mathrm{I}&\lesssim\sup_{k\in\zz}
\sum_{i=0}^{\infty}2^{i(1-\theta)}\lf\|
\mathbf1_{\{x\in\rn:\ M_\triangledown(f)(x)>2^{i}2^k\}} \r\|_{X}^{1-\theta}2^k
\lesssim\sup_{k\in\zz}
\sum_{i=0}^{\infty}2^{i+k}2^{-\theta i}\lf\|
\mathbf1_{\{x\in\rn:\ M_\triangledown(f)(x)>2^{i+k}\}}
\r\|_{X}^{1-\theta}\\\noz
&\lesssim
\sum_{i=0}^{\infty}2^{-\theta i}\sup_{l\in\zz}2^{l}\lf\|
\mathbf1_{\{x\in\rn:\ M_\triangledown(f)(x)>2^{l}\}}
\r\|_{X}^{1-\theta}\sim
\sup_{l\in\zz}2^{l}\lf\|
\mathbf1_{\{x\in\rn:\ M_\triangledown(f)(x)>2^{l}\}}
\r\|_{X^{1/(1-\theta)}}\\
&\lesssim\lf\|M_\triangledown(f)
\r\|_{X^{1/(1-\theta)}}.\noz
\end{align}
If $\frac{1-\theta}{r}>1$, then, by the H\"older inequality and Definition \ref{Debf}(i), we find that, for any
given
$\epsilon \in (0, \frac{\theta}{1-\theta})$,
\begin{align}\label{ll3}
\mathrm{I}&\lesssim\sup_{k\in\zz}2^k
\sum_{i=0}^{\infty}2^{i(1+\epsilon)(1-\theta)}\lf\|
\mathbf1_{\{x\in\rn:\ M_\triangledown(f)(x)>2^{i}2^k\}} \r\|_{X}^{1-\theta}
\lesssim
\sup_{l\in\zz}2^{l}\lf\|
\mathbf1_{\{x\in\rn:\ M_\triangledown(f)(x)>2^{l}\}}
\r\|_{X}^{1-\theta}\\ \noz
&\sim
\sup_{l\in\zz}2^{l}\lf\|
\mathbf1_{\{x\in\rn:\ M_\triangledown(f)(x)>2^{l}\}}
\r\|_{X^{1/(1-\theta)}}\lesssim\lf\|M_\triangledown(f)
\r\|_{X^{1/(1-\theta)}}.
\end{align}

Finally, from \eqref{ll5}, \eqref{ll4} and \eqref{ll3}, it follows that
\begin{align*}
\sup_{t\in(0,\infty)}t^{-\theta}K(t,f;H_{X}(\rn),L^\infty(\rn))
\lesssim\lf\|M_\triangledown(f)
\r\|_{X^{{1}/{(1-\theta)}}},
\end{align*}
which, together with \eqref{222} and \cite[Theorem 3.2(ii)]{zwyy},
further implies that $$\|f\|_{(H_{X}(\rn),L^\infty(\rn))_{\theta,\infty}}\lesssim\|f
\|_{WH_{X^{1/(1-\theta)}}(\rn)}$$ and hence completes the proof of \eqref{yib}.

Next, we show that
\begin{equation}\label{ydd}
(H_{X}(\rn),L^\infty(\rn))_{\theta,\infty}\subset WH_{X^{1/(1-\theta)}}(\rn).
\end{equation}
To this end, for any $f\in\cs'(\rn)$, let
$T( f ):= M_N^0(f)$, where $M_N^0$ is as in \eqref{EqMN0} with $N$ sufficiently
large (see Definition \ref{DewSH}).
To prove \eqref{ydd}, we first claim that the operator $T$ is bounded from the space
$(H_{X}(\rn),L^\infty(\rn))_{\theta,\infty}$ to the space
$(X,L^\infty(\rn))_{\theta,\infty}$. Indeed, let $g\in(H_{X}(\rn),L^\infty(\rn))_{\theta,\infty}$.
Then, from the definition of $(H_{X}(\rn),L^\infty(\rn))_{\theta,\infty}$, it follows that
there exist $g_0\in H_{X}(\rn)$ and $g_1\in L^\infty(\rn)$ such that
\begin{equation}\label{yee}
\sup_{t\in(0,\infty)}t^{-\theta}[\|g_0\|_{H_{X}(\rn)}+t\|g_1\|_{L^\infty(\rn)}]
\lesssim\|g\|_{(H_{X}(\rn),L^\infty(\rn))_{\theta,\infty}}.
\end{equation}
Since $T(g)\le T(g_0)+T(g_1)$ and $T$ is bounded from $L^\infty(\rn)$ to $L^\infty(\rn)$
and also from $H_{X}(\rn)$ to $X$, we deduce that $T(g_0)\in X$ and $T(g_1)\in L^\infty(\rn)$.
Thus, we have $$T(g)=T(g)\mathbf1_{E_0}+T(g)\mathbf1_{E_1\setminus E_0}\in X+L^\infty(\rn),$$
where $E_0:=\{x\in\rn:\ \frac12T(g)(x)\le T(g_0)(x)\}$
and $E_1:=\{x\in\rn:\ \frac12T(g)(x)\le T(g_1)(x)\}$.
From this and \eqref{yee}, we deduce that
\begin{align*}
\|T(g)\|_{(X,L^\infty(\rn))_{\theta,\infty}}&\lesssim
\sup_{t\in(0,\infty)}t^{-\theta}[\|T(g)\mathbf1_{E_0}\|_{X}+
t\|T(g)\mathbf1_{E_1\setminus E_0}\|_{L^\infty(\rn)}]\\ \noz
&\lesssim
\sup_{t\in(0,\infty)}t^{-\theta}[\|T(g_0)\|_{X}+t\|T(g_1)\|_{L^\infty(\rn)}]\\ \noz
&\lesssim
\sup_{t\in(0,\infty)}t^{-\theta}[\|g_0\|_{H_{X}(\rn)}+t\|g_1\|_{L^\infty(\rn)}]
\lesssim\|g\|_{(H_{X}(\rn),L^\infty(\rn))_{\theta,\infty}}.
\end{align*}
Therefore, the above claim holds true.
Form this claim and Lemma 3.1, it follows that, for any
$f\in(H_{X}(\rn),L^\infty(\rn))_{\theta,\infty}$,
$T ( f )\in (WX)^{1/(1-\theta)}$, namely,
$f\in WH_{X^{1/(1-\theta)}}(\rn)$.
Thus, \eqref{ydd} holds true.
This finishes the proof of Theorem \ref{inte}.
\end{proof}

\section{Applications}\label{s5}

In this section, we apply all above results to five concrete examples
of ball quasi-Banach function spaces, namely, Morrey spaces, mixed-norm Lebesgue spaces,
variable Lebesgue spaces, weighted Orlicz spaces
and Orlicz-slice spaces. Observe that, among the five examples, only variable Lebesgue spaces are
quasi-Banach function spaces, while the other four examples are not.

\subsection{Morrey spaces}

\begin{definition}
Let $0<q\le p<\infty$.
The \emph{Morrey space $M_q^p(\rn)$} is defined
to be the set of all measurable functions $f$ such that
$$
\|f\|_{M_q^p(\rn)}:=\sup_{B\in\BB}|B|^{1/p-1/q}\|f\|_{L^q(B)}<\infty,
$$
where $\BB$ is as in \eqref{Eqball} (the set of all balls of $\rn$).
\end{definition}

Recall that, due to the applications in elliptic partial differential equations,
the Morrey space $M_q^p(\rn)$ with $0<q\le p<\infty$
were introduced by Morrey \cite{m38} in 1938.
In recent decades,
there exists an increasing interest in applications
of Morrey spaces to various areas of analysis,
such as partial differential equations, potential theory and
harmonic analysis (see, for instance, \cite{a15,a04,cf,JW,ysy}).

Now, we recall the notion of the weak Morrey space $WM_q^p(\rn)$.

\begin{definition}
Let $0<q\le p<\infty$.
The \emph{weak Morrey space $WM_q^p(\rn)$}
is defined to be the set of all measurable functions $f$ such that
$$
\|f\|_{WM_q^p(\rn)}:=\sup_{\alpha\in(0,\infty)}\,\lf\{\alpha\lf\|\mathbf1_{\{x\in\rn:\ |f(x)|>\alpha\}}\r\|_{M_q^p(\rn)}\r\}<\infty.
$$
\end{definition}

\begin{remark}
Let $0<q\le p<\infty$.
\begin{itemize}
\item[(i)] The weak Morrey space $WM_q^p(\rn)$
is just the weak Morrey space $\mathcal{M}_u^{q,\infty}(\rn)$ in \cite{H17} with
$u(B):=|B|^{1/q-1/p}$ for any $B\in\BB$, where $\BB$ is as in \eqref{Eqball}.
\item[(ii)] Observe that, as was pointed out in \cite[p.\,86]{SHYY}, $M_q^p(\rn)$
is not a quasi-Banach function space, but it does be a ball
quasi-Banach function space as in Definition \ref{Debqfs}.
\end{itemize}
\end{remark}

Next, we recall the notions of both the weak Morrey Hardy
space $WHM_q^p(\rn)$ and the Morrey Hardy
space $HM_q^p(\rn)$ as follows.

\begin{definition}
Let $0<q\le p<\infty$.
\begin{itemize}
\item[(i)] The \emph{Morrey Hardy
space $HM_q^p(\rn)$} is defined to be the set of all $f\in\cs'(\rn)$
such that $\|f\|_{HM_q^p(\rn)}:=\
\|M_N^0(f)\|_{M_q^p(\rn)}<\infty$,
where $M_N^0(f)$ is as in \eqref{EqMN0} with $N$ sufficiently large.

\item[(ii)] The \emph{weak Morrey Hardy
space $WHM_q^p(\rn)$} is defined to be the set of all $f\in\cs'(\rn)$
such that $\|f\|_{WHM_q^p(\rn)}:=\
\|M_N^0(f)\|_{WM_q^p(\rn)}<\infty$,
where $M_N^0(f)$ is as in \eqref{EqMN0} with $N$ sufficiently large.
\end{itemize}
\end{definition}

\begin{remark}
Let $1<q\le p<\infty$. By \cite[Remark 7.8]{zwyy},
we know that $WHM_q^p(\rn)=WM_q^p(\rn)$ with equivalent norms.
\end{remark}

Let $0<q\le p<\infty$. Then
$M_q^p(\rn)$ satisfies Assumption \ref{a2.17} (see \cite[Theorem 3.2]{H17}) and
Assumption \ref{a2.15} (see, for instance, \cite{cf,H15}).
Applying  \cite[Lemma 5.7]{H15} and \cite[Theorem 4.1]{ST},
we can easily show that,
for any given $r_0\in(0,q)$ and $p_0\in(q,\infty]$, \eqref{Eqdm} with $X:=M_q^p(\rn)$
holds true.
Thus, all the assumptions of main theorems in Sections \ref{s3} and \ref{s4} are satisfied.
Using
Theorems
\ref{Tharea}, \ref{Thgf} and \ref{Thgx},
we obtain
the following characterizations of $WHM_q^p(\rn)$, respectively, in terms of the
Lusin area function, the Littlewood--Paley $g$-function and the Littlewood--Paley $g_\lambda^*$-function.

\begin{theorem}\label{Mth3}
Let $0<q\le p<\infty$.
Then $f\in WHM_q^p(\rn)$ if and only if either
of the following two items holds true:
\begin{itemize}
\item[\rm{(i)}] $f\in\cs'(\rn)$, $f$ vanishes weakly at infinity and $S(f)\in WM_q^p(\rn)$,
where $S(f)$ is as in \eqref{eq61}.
\item[\rm{(ii)}] $f\in\cs'(\rn)$, $f$ vanishes weakly at infinity and $g(f)\in WM_q^p(\rn)$,
where $g(f)$ is as in \eqref{eq63}.
\end{itemize}
Moreover, for any $f\in WHM_q^p(\rn)$,
$\|f\|_{WHM_q^p(\rn)}\sim\|S(f)\|_{WM_q^p(\rn)}\sim\|g(f)\|_{WM_q^p(\rn)}$,
where the positive equivalence constants are independent of $f$.
\end{theorem}

\begin{theorem}\label{Mth4}
Let $0<q\le p<\infty$.
Let
$\lambda\in(\max\{\frac{2}{\min\{1,q\}},1-\frac2{\max\{1,q\}}+\frac{2}{\min\{1,q\}}\},\infty)$.
Then $f\in WHM_q^p(\rn)$ if and only if $f\in\mathcal{S}'(\rn)$,
$f$ vanishes weakly at infinity and $g_\lambda^*(f)\in WM_q^p(\rn)$,
where $g_\lambda^*(f)$ is as in \eqref{eq62}.
Moreover, for any $f\in WHM_q^p(\rn)$,
$$
\|f\|_{WHM_q^p(\rn)}\sim\lf\|g_\lambda^*(f)\r\|_{WM_q^p(\rn)},
$$
where the positive equivalence constants are independent of $f$.
\end{theorem}

\begin{remark}
Let $0<q\le p<\infty$ and $p=q$. Then we know that
$M_q^p(\rn)=L^q(\rn)$ and $WM_q^p(\rn)(\rn)=WH^q(\rn)$,
where $WH^q(\rn)$ denotes the classical weak Hardy space.
In this case,
when $q\in(0,1]$, the range of
$\lambda$ in Theorem \ref{Mth4} becomes $\lambda\in(2/q,\infty)$
which is known to be the best possible and was proved in \cite[Theorem 4.13]{LYJ} as a special case.
\end{remark}

Applying Theorem \ref{inte}, we find that
$WHM_{q/(1-\theta)}^{p/(1-\theta)}(\rn)$
is just the real interpolation intermediate space between $HM_q^p(\rn)$
and $L^\infty(\rn)$ with $\theta\in(0,1)$ as follows.

\begin{theorem}\label{Mth5}
Let $0<q\le p<\infty$ and $\theta\in(0,1)$. Then
$$
(HM_q^p(\rn),L^\infty(\rn))_{\theta,\infty}=WHM_{q/(1-\theta)}^{p/(1-\theta)}(\rn).
$$
\end{theorem}

\begin{remark}
\begin{itemize}
\item[\rm{(i)}] Let $0<q\le p<\infty$ and $p=q$.
Then
we know that $M_q^p(\rn)=L^q(\rn)$ and $WM_q^p(\rn)=WL^q(\rn)$. In this case,
Theorem \ref{Mth5} coincides with \cite[Theorem 1]{FRS},
which states that
$$
(H^p(\rn),L^\infty(\rn))_{\theta,\infty}=WH^{p/(1-\theta)}(\rn),\ \theta\in(0,1).
$$
\item[\rm{(ii)}] In \cite{H17}, Ho obtained the atomic characterizations of
the weak Hardy--Morrey spaces and, using these atomic characterizations,
Ho \cite{H17} further established the Hardy inequalities on the weak Hardy--Morrey spaces.
However, all these results stated in this subsection are new.
\end{itemize}
\end{remark}

\subsection{Mixed-norm Lebesgue spaces}

\begin{definition}
Let $\vec{p}:=(p_1,\ldots,p_n)\in(0,\infty]^n$.
The \emph{mixed-norm Lebesgue space $L^{\vec{p}}(\rn)$} is defined
to be the set of all measurable functions $f$ such that
$$
\|f\|_{L^{\vec{p}}(\rn)}:=\lf\{\int_{\rr}\cdots\lf[\int_{\rr}|f(x_1,\ldots,x_n)|^{p_1}\,dx_1\r]
^{\frac{p_2}{p_1}}\cdots\,dx_n\r\}^{\frac{1}{p_n}}<\infty
$$
with the usual modifications made when $p_i=\infty$ for some $i\in\{1,\ldots,n\}$.
\end{definition}

The study of mixed-norm Lebesgue spaces $L^{\vec{p}}(\rn)$ with
$\vec{p}\in (0,\infty]^n$ originated from Benedek and Panzone
\cite{bp} in the early 1960's, which can be traced back to H\"ormander \cite{H60}.
Later on, in 1970, Lizorkin \cite{l70} further developed both the theory of
multipliers of Fourier integrals and estimates of convolutions
in the mixed-norm Lebesgue spaces.
Particularly,
in order to meet the requirements arising in the study of the boundedness of operators,
partial differential equations and some other fields, the real-variable theory of mixed-norm
function spaces, including mixed-norm Morrey spaces,
mixed-norm Hardy spaces, mixed-norm Besov
spaces and mixed-norm Triebel--Lizorkin spaces, has rapidly been developed
in recent years (see, for instance,
\cite{cgn17bs,gjn17,tn,hy,hlyy,hlyy19}).

From the definition
of $\|\cdot\|_{L^{\vec{p}}(\rn)}$, it is easy to deduce that
the mixed-norm Lebesgue space $L^{\vec{p}}(\rn)$
is a ball quasi-Banach function space.
Let $\vec{p}:=(p_1,\ldots,p_n)\in[1,\infty]^n$. Then, for any $f\in L^{\vec{p}}(\rn)$
and $g\in L^{\vec{p'}}(\rn)$, it is easy to know that
$$\int_{\rn}|f(x)g(x)|\, dx \le \|f\|_{L^{\vec{p}}(\rn)}\|g\|_{L^{\vec{p'}}(\rn)},
$$
where $\vec{p'}$
denotes the conjugate vector of $\vec{p}$, namely, for any $i\in{1,\ldots,n}$, $1/p_i + 1/p_i'
= 1$.
This implies that $L^{\vec{p}}(\rn)$ with $\vec{p}\in[1,\infty]^n$ is a ball
Banach function space.
But, the mixed-norm Lebesgue space $L^{\vec{p}}(\rn)$ is not a quasi-Banach function space
(see, for instance, \cite[Remark 7.20]{zwyy}).

In what follows, for any $\vec{p}:=(p_1,\ldots,p_n)\in(0,\infty)^n$, we always let
$p_-:= \min\{p_1, \ldots , p_n\}$ and  $p_+ := \max\{p_1, \ldots , p_n\}$.

Now, we recall the notions of the weak mixed-norm Lebesgue space $WL^{\vec{p}}(\rn)$,
the weak mixed-norm Hardy space $WH^{\vec{p}}(\rn)$
and the mixed-norm Hardy
space $H^{\vec{p}}(\rn)$ as follows.

\begin{definition}
Let $\vec{p}\in(0,\infty)^n$.
The \emph{weak mixed-norm Lebesgue space $WL^{\vec{p}}(\rn)$}
is defined to be the set of all measurable functions $f$ such that
$$
\|f\|_{WL^{\vec{p}}(\rn)}:=\sup_{\alpha\in(0,\infty)}\,\lf\{\alpha\lf\|\mathbf1_{\{x\in\rn:\ |f(x)|>\alpha\}}\r\|_{L^{\vec{p}}(\rn)}\r\}<\infty.
$$
\end{definition}

\begin{definition}
Let $\vec{p}\in(0,\infty)^n$.
\begin{itemize}
\item[\rm{(i)}]
The \emph{mixed-norm Hardy
space $H^{\vec{p}}(\rn)$} is defined to be the set of all $f\in\cs'(\rn)$
such that $\|f\|_{H^{\vec{p}}(\rn)}:=\
\|M_N^0(f)\|_{L^{\vec{p}}(\rn)}<\infty$,
where $M_N^0(f)$ is as in \eqref{EqMN0} with $N$ sufficiently large.
\item[\rm{(ii)}]
The \emph{weak mixed-norm Hardy space $WH^{\vec{p}}(\rn)$}
is defined to be the set of all $f\in\cs'(\rn)$
such that $\|f\|_{WH^{\vec{p}}(\rn)}:=\
\|M_N^0(f)\|_{WL^{\vec{p}}(\rn)}<\infty$,
where $M_N^0(f)$ is as in \eqref{EqMN0} with $N$ sufficiently large.
\end{itemize}
\end{definition}

\begin{remark}
Let $\vec{p}\in(1,\infty)^n$. By \cite[Remark 7.28]{zwyy},
we know that $WH^{\vec{p}}(\rn)=WL^{\vec{p}}(\rn)$ with equivalent norms.
\end{remark}

Let $\vec{p}\in(0,\infty)^n$. Then
$L^{\vec{p}}(\rn)$ satisfies Assumption \ref{a2.17} (see \cite[Theorem 7.19]{zwyy}) and
Assumption \ref{a2.15} (see \cite[Lemma 3.7]{hlyy}).
Applying  \cite[Lemma 3.5]{hlyy} and \cite[Theorem 1.a]{bp},
we can easily show that,
for any given $r_0\in(0,p_-]$ and $p_0\in(p_+,\infty]$, \eqref{Eqdm} with $X:=L^{\vec{p}}(\rn)$
holds true.
Thus, all the assumptions of main theorems in Sections \ref{s3} and \ref{s4} are satisfied.
Using Theorems
\ref{Tharea}, \ref{Thgf} and \ref{Thgx},
we obtain
the following characterizations of $WH^{\vec{p}}(\rn)$, respectively, in terms of the
Lusin area function, the Littlewood--Paley $g$-function and the Littlewood--Paley $g_\lambda^*$-function.

\begin{theorem}\label{mixMth3}
Let $\vec{p}\in(0,\infty)^n$.
Then $f\in WH^{\vec{p}}(\rn)$ if and only if either
of the following two items holds true:
\begin{itemize}
\item[\rm{(i)}] $f\in\cs'(\rn)$, $f$ vanishes weakly at infinity and $S(f)\in WL^{\vec{p}}(\rn)$,
where $S(f)$ is as in \eqref{eq61}.
\item[\rm{(ii)}] $f\in\cs'(\rn)$, $f$ vanishes weakly at infinity and $g(f)\in WL^{\vec{p}}(\rn)$,
where $g(f)$ is as in \eqref{eq63}.
\end{itemize}
Moreover, for any $f\in WH^{\vec{p}}(\rn)$,
$\|f\|_{WH^{\vec{p}}(\rn)}\sim\|S(f)\|_{WL^{\vec{p}}(\rn)}\sim\|g(f)\|_{WL^{\vec{p}}(\rn)}$,
where the positive equivalence constants are independent of $f$.
\end{theorem}

\begin{theorem}\label{mixMth4}
Let $\vec{p}\in(0,\infty)^n$.
Let
$\lambda\in(\max\{\frac{2}{\min\{1,p_-\}},1-\frac2{\max\{1,p_+\}}+\frac{2}{\min\{1,p_-\}}\},\infty)$.
Then $f\in WH^{\vec{p}}(\rn)$ if and only if $f\in\mathcal{S}'(\rn)$,
$f$ vanishes weakly at infinity and
$g_\lambda^*(f)\in WL^{\vec{p}}(\rn)$,
where $g_\lambda^*(f)$ is as in \eqref{eq62}.
Moreover, for any $f\in WH^{\vec{p}}(\rn)$,
$$
\|f\|_{WH^{\vec{p}}(\rn)}\sim\lf\|g_\lambda^*(f)\r\|_{WL^{\vec{p}}(\rn)},
$$
where the positive equivalence constants are independent of $f$.
\end{theorem}

Applying Theorem \ref{inte}, we find that
$WH^{\vec{p}/(1-\theta)}(\rn)$
is just the real interpolation intermediate space between $H^{\vec{p}}(\rn)$
and $L^\infty(\rn)$ with $\theta\in(0,1)$ as follows.

\begin{theorem}\label{mixMth5}
Let $\vec{p}\in(0,\infty)^n$ and $\theta\in(0,1)$. Then
$$
(H^{\vec{p}}(\rn),L^\infty(\rn))_{\theta,\infty}=WH^{\vec{p}/(1-\theta)}(\rn).
$$
\end{theorem}

\subsection{Variable Lebesgue spaces}
Let $p(\cdot):\ \rn\to[0,\infty)$ be a measurable function. Then the \emph{variable Lebesgue space
$L^{p(\cdot)}(\rn)$} is defined to be the set of all measurable functions $f$ on $\rn$ such that
$$
\|f\|_{L^{p(\cdot)}(\rn)}:=\inf\lf\{\lambda\in(0,\infty):\ \int_\rn[|f(x)|/\lambda]^{p(x)}\,dx\le1\r\}<\infty.
$$
If $p(x)\geq1$ for almost every $x\in\rn$, then $\|f\|_{L^{p(\cdot)}(\rn)}$ is
a Banach function space (see, for instance, \cite[Theorem 3.2.13]{DHR}).
We refer the reader to \cite{N1,N2,KR,CUF,DHR} for more details on variable Lebesgue spaces.

For any measurable function $p(\cdot):\ \rn\to(0,\infty)$, let
$$
 p_-:=\underset{x\in\rn}{\essinf}\,p(x)\quad\text{and}\quad p_+:=\underset{x\in\rn}{\esssup}\,p(x).
$$
Then $p(x)/p_-\geq1$ for almost every $x\in\rn$. Therefore, for any $B\in\BB$ with $\BB$ as in \eqref{Eqball},
$$
\lf\|\mathbf{1}_B\r\|_{L^{p(\cdot)}(\rn)}=\lf\|\mathbf{1}_B\r\|_{L^{p(\cdot)/p_-}(\rn)}^{1/p_-}<\infty
$$
(see, for instance, \cite[Lemma 3.2.6]{DHR}). Thus, whenever $p(\cdot):\rn\to(0,\infty)$,
$L^{p(\cdot)}(\rn)$ is a ball quasi-Banach function space.

A measurable function $p(\cdot):\ \rn\to(0,\infty)$ is said to be \emph{globally
log-H\"older continuous} if there exists a $p_{\infty}\in\rr$ such that, for any
$x,\ y\in\rn$,
$$
|p(x)-p(y)|\lesssim\frac{1}{\log(e+1/|x-y|)}
$$
and
$$
|p(x)-p_\infty|\lesssim\frac{1}{\log(e+|x|)},
$$
where the implicit positive constants are independent of $x$ and $y$.

Let $p(\cdot):\ \rn\to(0,\infty)$ be a globally
log-H\"older continuous function satisfying
$0<p_-\le p_+<\infty$.
Then $L^{p(\cdot)}(\rn)$ satisfies Assumption \ref{a2.15}
(see, for instance, \cite{CUF,CUW}) and Assumption \ref{a2.17}
(see \cite[Proposition 3.4]{YYYZ}). Furthermore, from \cite[Lemma 2.16]{CUF},
we deduce that,
for any given $r_0\in(0,\min\{1,p_-\})$ and $p_0\in(\max\{1,p_+\},\infty]$,
\eqref{Eqdm} with $X:=L^{p(\cdot)}(\rn)$ holds true.
Thus, all the assumptions of main theorems in Sections \ref{s3} and \ref{s4} are satisfied.
If $X:=L^{p(\cdot)}(\rn)$, then $WH_X(\rn)$ is the variable weak Hardy space studied in \cite{YYYZ}.
In this case, Yan et al. \cite{YYYZ} and Zhuo et al. \cite{zyy}
obtained all main theorems in Sections \ref{s3} and \ref{s4}
for the variable weak Hardy space $WH^{p(\cdot)}(\rn)$.
However, \emph{there exists a gap} in the proof of the Lusin area function characterization
of the variable weak Hardy space $WH^{p(\cdot)}(\rn)$ in
\cite[(6.5)]{YYYZ}.
Using Theorems
\ref{Tharea}, \ref{Thgf} and \ref{Thgx},
we obtain the following characterizations of $WH^{p(\cdot)}(\rn)$, respectively, in terms of the
Lusin area function, the Littlewood--Paley $g$-function and the Littlewood--Paley $g_\lambda^*$-function
under some even \emph{weaker assumptions} on the Littlewood--Paley functions.

\begin{theorem}\label{vth3}
Let $p(\cdot):\ \rn\to(0,\infty)$ be a globally
log-H\"older continuous function satisfying
$0<p_-\le p_+<\infty$.
Then $f\in WH^{p(\cdot)}(\rn)$ if and only if either
of the following two items holds true:
\begin{itemize}
\item[\rm{(i)}] $f\in\cs'(\rn)$, $f$ vanishes weakly at infinity and $S(f)\in WL^{p(\cdot)}(\rn)$,
where $S(f)$ is as in \eqref{eq61}.
\item[\rm{(ii)}] $f\in\cs'(\rn)$, $f$ vanishes weakly at infinity and $g(f)\in WL^{p(\cdot)}(\rn)$,
where $g(f)$ is as in \eqref{eq63}.
\end{itemize}
Moreover, for any $f\in WH^{p(\cdot)}(\rn)$,
$\|f\|_{WH^{p(\cdot)}(\rn)}\sim\|S(f)\|_{WL^{p(\cdot)}(\rn)}\sim\|g(f)\|_{WL^{p(\cdot)}(\rn)}$,
where the positive equivalence constants are independent of $f$.
\end{theorem}

\begin{theorem}\label{vth4}
Let $p(\cdot):\ \rn\to(0,\infty)$ be a globally
log-H\"older continuous function satisfying
$0<p_-\le p_+<\infty$.
Let
$\lambda\in(\max\{\frac{2}{\min\{1,p_-\}},1-\frac2{\max\{1,p_+\}}+\frac{2}{\min\{1,p_-\}}\},\infty)$.
Then $f\in WH^{p(\cdot)}(\rn)$ if and only if $f\in\mathcal{S}'(\rn)$,
$f$ vanishes weakly at infinity and
$g_\lambda^*(f)\in WL^{p(\cdot)}(\rn)$,
where $g_\lambda^*(f)$ is as in \eqref{eq62}.
Moreover, for any $f\in WH^{p(\cdot)}(\rn)$,
$$
\|f\|_{WH^{p(\cdot)}(\rn)}\sim\lf\|g_\lambda^*(f)\r\|_{WL^{p(\cdot)}(\rn)},
$$
where the positive equivalence constants are independent of $f$.
\end{theorem}

We point out that, when $p_-\in(0,1]$,
the $g_\lambda^\ast$-function characterization in Theorem \ref{vth4}
widens the range of $\lambda\in(1+\frac{2}{\min\{2,p_-\}},\infty)$
in \cite{YYYZ} into $\lz\in(\max\{\frac{2}{\min\{1,p_-\}},
1-\frac2{\max\{1,p_+\}}+\frac{2}{\min\{1,p_-\}}\},\infty)$.
Recall that, when $p_-\in(1,\infty)$, from \cite[Proposition 3.4]{YYYZ} and \cite[Theorem 3.4]{zwyy}, it is
easy to deduce that $WL^{p(\cdot)}(\rn)=WH^{p(\cdot)}(\rn)$ with equivalent norms, which was also
proved in \cite{zyy}.

\subsection{Weighted Orlicz spaces}

If $X$ is a weighted Orlicz space with Muckenhoupt weights, then $WH_X(\rn)$ is a
weighted weak Orlicz--Hardy space.
The weighted weak Orlicz--Hardy space $WH_\omega^\Phi(\rn)$, which is defined via the weighted
Orlicz space $L_\omega^\Phi(\rn)$, was studied in \cite{ccyy,yy11}.

Recall that an $A_p(\rn)$-\emph{weight} $\omega$, with $p\in[1,\infty)$, is a
locally integrable and nonnegative function satisfying that, for any given $p\in(1,\infty)$,
\begin{equation*}
\sup_{B\in\BB}\lf[\frac1{|B|}\int_B\omega(x)\,dx\r]\lf[\frac1{|B|}
\int_B\lf\{\omega(x)\r\}^{\frac1{1-p}}\,dx\r]^{p-1}<\infty
\end{equation*}
and, for $p=1$,
$$
\sup_{B\in\BB}\frac1{|B|}\int_B\omega(x)\,dx\lf[\lf\|\omega^{-1}\r\|_{L^\infty(B)}\r]<\infty,
$$
where $\BB$ is as in \eqref{Eqball}. Then define $A_\infty(\rn):=\bigcup_{p\in[1,\infty)}A_p(\rn)$.

Now, we recall the notions of both Orlicz functions and Orlicz spaces (see, for instance, \cite{RR}).

\begin{definition}\label{or}
A function $\Phi:\ [0,\infty)\ \to\ [0,\infty)$ is called an \emph{Orlicz function} if it is
non-decreasing and satisfies $\Phi(0)= 0$, $\Phi(t)>0$ whenever $t\in(0,\infty)$, and $\lim_{t\to\infty}\Phi(t)=\infty$.
\end{definition}

An Orlicz function $\Phi$ as in Definition \ref{or} is said to be
of \emph{lower} (resp., \emph{upper}) \emph{type} $p$ with
$p\in(-\infty,\infty)$ if
there exists a positive constant $C_{(p)}$, depending on $p$, such that, for any $t\in[0,\infty)$
and $s\in(0,1)$ [resp., $s\in [1,\infty)$],
\begin{equation*}
\Phi(st)\le C_{(p)}s^p \Phi(t).
\end{equation*}
A function $\Phi:\ [0,\infty)\ \to\ [0,\infty)$ is said to be of
 \emph{positive lower} (resp., \emph{upper}) \emph{type} if it is of lower
 (resp., upper) type $p$ for some $p\in(0,\infty)$.

\begin{definition}\label{fine}
Let $\Phi$ be an Orlicz function with positive lower type $p_{\Phi}^-$ and positive upper type $p_{\Phi}^+$.
Let $\omega\in A_\infty(\rn)$.
The \emph{weighted Orlicz space $L_\omega^\Phi(\rn)$} is defined
to be the set of all measurable functions $f$ such that
 $$\|f\|_{L_\omega^\Phi(\rn)}:=\inf\lf\{\lambda\in(0,\infty):\ \int_{\rn}\Phi\lf(\frac{|f(x)|}{\lambda}\r)\omega(x)\,dx\le1\r\}<\infty.$$
\end{definition}

\begin{remark}
\begin{itemize}
\item[(i)] Let $p\in(0,\infty)$ and $\Phi(t):=t^p$
for any $t\in[0,\infty)$. It is easy to see that, in this case,
$L^\Phi_\omega(\rn)$ coincides with the weighted Lebesgue space $L_\omega^p(\rn)$
(see, for instance, \cite{AJ,QY}).
\item[(ii)] By \cite[p.\,86]{SHYY}, we know that
a weighted Orlicz space $L_\omega^\Phi(\rn)$ with an
$A_\infty(\rn)$-weight $\omega$ may \emph{not be} a quasi-Banach function space.
For example, let $\Phi(t):=t^2$ for any $t\in[0,\infty)$
and $\omega_0:=1+|x|$ for any $x\in\rr$. In this case,
$L^\Phi_\omega(\rr)=L_{\omega_0}^2(\rr)$.
Let $E:=\cup_{l\in\nn}[2^l,2^l+2^{-l}]$. Then it is easy to know that $|E|<\infty$
and $\|\mathbf{1}_{E}\|_{X}=\infty$. Thus, $L_{\omega_0}^2(\rr)$ is not a quasi-Banach function space.
On the other hand, by the definition of $A_\infty(\rn)$, it is easy to prove that
$L_\omega^\Phi(\rn)$ is a ball quasi-Banach function space.
\end{itemize}
\end{remark}

Let $\Phi$ be an Orlicz function with positive lower type $p_{\Phi}^-$ and positive upper type $p_{\Phi}^+$.
Let $\omega\in A_\infty(\rn)$. Then
$L_\omega^\Phi(\rn)$ satisfies Assumption \ref{a2.17} (see \cite[Theorem 2.8]{LYJ}) and
Assumption \ref{a2.15} (see \cite[Theorem 2.10]{LYJ}).
Applying \cite[p.\,110, Theorem 7]{RR} and \cite[Lemma 2.16]{ZYYW},
we know that, when $p_{\Phi}^-\in(1,\infty)$, $(L_\omega^\Phi(\rn))^*$ is
isomorphic and homeomorphic to $L_\omega^\Psi(\rn)$, where, for any $y\in[0,\infty)$,
$$\Psi(y):=\sup\lf\{xy-\Phi(x):\ x\in[0,\infty)\r\}.$$
From this and \cite[Proposition 7.8]{SHYY}, we can easily deduce that,
for any given $r_0\in(0,p_{\Phi}^-)$, $p_0\in(p_{\Phi}^+,\infty]$
and $\omega\in A_{(p_{\Phi}^+/r_0)'/(p_0/r_0)'}(\rn)$, \eqref{Eqdm}
holds true.
Thus, all the assumptions of main theorems in Sections \ref{s3} and \ref{s4} are satisfied.

When $X:=L_\omega^\Phi(\rn)$,
$WH_X(\rn)$ goes back to the weighted weak Orlicz--Hardy space $WH_\omega^\Phi(\rn)$.
When $p_{\Phi}^-\in(1,\infty)$, from \cite[Theorem 2.8]{LYJ} and
\cite[Theorem 3.4]{zwyy}, it is easy to deduce that
$WL_\omega^\Phi(\rn)=WH_\omega^\Phi(\rn)$  with equivalent norms.
When $p_{\Phi}^+\in(0,1]$, we point
out that Liang et al. \cite{LYJ} obtained all main theorems in Section \ref{s3}
for the weighted weak Orlicz--Hardy space $WH_\omega^\Phi(\rn)$ as the special case
of the results in \cite{LYJ}.
However, \emph{there exists a gap} in the proof of the Lusin area function characterization
in lines 14 and 16 of \cite[p.\,662]{LYJ}.
Using Theorems
\ref{Tharea}, \ref{Thgf} and \ref{Thgx},
we obtain the following characterizations of $WH_\omega^\Phi(\rn)$, respectively, in terms of the
Lusin area function, the Littlewood--Paley $g$-function and the Littlewood--Paley $g_\lambda^*$-function
under some even weaker assumptions on the Littlewood--Paley functions.

\begin{theorem}\label{wth3}
Let $\Phi$ be an Orlicz function with positive lower type $p_{\Phi}^-$ and positive upper type $p_{\Phi}^+$.
Let $\omega\in A_\infty(\rn)$.
Then $f\in WH_\omega^\Phi(\rn)$ if and only if either
of the following two items holds true:
\begin{itemize}
\item[\rm{(i)}] $f\in\cs'(\rn)$, $f$ vanishes weakly at infinity and $S(f)\in WL_\omega^\Phi(\rn)$,
where $S(f)$ is as in \eqref{eq61}.
\item[\rm{(ii)}] $f\in\cs'(\rn)$, $f$ vanishes weakly at infinity and $g(f)\in WL_\omega^\Phi(\rn)$,
where $g(f)$ is as in \eqref{eq63}.
\end{itemize}
Moreover, for any $f\in WH_\omega^\Phi(\rn)$,
$\|f\|_{WH_\omega^\Phi(\rn)}\sim\|S(f)\|_{WL_\omega^\Phi(\rn)}\sim\|g(f)\|_{WL_\omega^\Phi(\rn)}$,
where the positive equivalence constants are independent of $f$.
\end{theorem}

\begin{theorem}\label{wth4}
Let $\Phi$ be an Orlicz function with positive lower type $p_{\Phi}^-$ and positive upper type $p_{\Phi}^+$.
Let $\omega\in A_\infty(\rn)$.
Let
$\lambda\in(\max\{\frac{2}{\min\{1,p_{\Phi}^-\}},
1-\frac2{\max\{1,p_{\Phi}^+\}}+\frac{2}{\min\{1,p_{\Phi}^-\}}\},\infty)$.
Then $f\in WH_\omega^\Phi(\rn)$ if and only if $f\in\mathcal{S}'(\rn)$,
$f$ vanishes weakly at infinity and
$g_\lambda^*(f)\in WL_\omega^\Phi(\rn)$,
where $g_\lambda^*(f)$ is as in \eqref{eq62}.
Moreover, for any $f\in WH_\omega^\Phi(\rn)$,
$$
\|f\|_{WH_\omega^\Phi(\rn)}\sim\lf\|g_\lambda^*(f)\r\|_{WL_\omega^\Phi(\rn)},
$$
where the positive equivalence constants are independent of $f$.
\end{theorem}

In what follows, for any given Orlicz function $\Phi$
and any given $\theta\in(0,\infty)$, and for any $t\in[0,\infty)$, let
\begin{equation}\label{phi1}
\Phi_{1/\theta}(t):=\Phi(t^{1/\theta}).
\end{equation}

\begin{definition}
Let $\Phi$ be an Orlicz function with positive lower type $p_{\Phi}^-$ and positive upper type $p_{\Phi}^+$.
Let $\omega\in A_\infty(\rn)$.
The \emph{weighted Orlicz--Hardy
space $H_\omega^\Phi(\rn)$} is defined to be the set of all $f\in\cs'(\rn)$
such that $\|f\|_{H_\omega^\Phi(\rn)}:=\
\|M_N^0(f)\|_{L_\omega^\Phi(\rn)}<\infty$,
where $M_N^0(f)$ is as in \eqref{EqMN0} with $N$ sufficiently large.
\end{definition}

Applying Theorem \ref{inte} to weighted Orlicz spaces, we find that
$WH_\omega^{\Phi_{1/(1-\theta)}}(\rn)$
is just the real interpolation intermediate space between $H_\omega^\Phi(\rn)$
and $L^\infty(\rn)$ with $\theta\in(0,1)$.

\begin{theorem}\label{th5}
Let $\Phi$ be an Orlicz function with positive lower type $p_{\Phi}^-$ and positive upper type $p_{\Phi}^+$.
Let $\omega\in A_\infty(\rn)$ and
$\theta\in(0,1)$. Then
$$
(H_\omega^\Phi(\rn),L^\infty(\rn))_{\theta,\infty}=WH_\omega^{\Phi_{1/(1-\theta)}}(\rn),
$$
where $\Phi_{1/(1-\theta)}$ is as in \eqref{phi1} with $\theta$ replaced by $1-\theta$.
\end{theorem}

\subsection{Orlicz-slice spaces}

If $X$ is an Orlicz-slice spaces, then $WH_X(\rn)$ is a
weak Orlicz-slice Hardy space.
Let $q,\ t\in(0,\infty)$ and $\Phi$ be an Orlicz function.
Recall that the Orlicz-slice space $(E_\Phi^q)_t(\rn)$ introduced in \cite{ZYYW}
generalizes both the slice space $E_t^p(\rn)$ [in this case, $\Phi(\tau):=\tau^2$ for any $\tau\in[0,\infty)$], which
was originally introduced by Auscher and
Mourgoglou \cite{AM2014} and has been applied to study the classification of weak solutions in the natural classes
for the boundary value problems of a $t$-independent elliptic system in the upper plane,
and $(E_r^p)_t(\rn)$ [in this case, $\Phi(\tau):=\tau^r$ for any $\tau\in[0,\infty)$ with $r\in(0,\infty)$],
which was originally introduced by Auscher and Prisuelos-Arribas \cite{APA}
and has been applied to study the boundedness of operators such
as the Hardy--Littlewood maximal operator, the Calder\'on--Zygmund operator and the Riesz potential.
Moreover, Zhang et al. \cite{ZYYW} introduced the Orlicz-slice Hardy space $(HE_\Phi^q)_t(\rn)$
and obtained real-variable characterizations of $(HE_\Phi^q)_t(\rn)$, respectively,
in terms of various maximal functions, atoms, molecules and Littlewood--Paley functions, and the boundedness
on $(HE_\Phi^q)_t(\rn)$ for convolutional $\delta$-order and non-convolutional $\gamma$-order Calder\'on--Zygmund
operators. Naturally, this new scale of Orlicz-slice Hardy spaces contains the variant of the Hardy-amalgam space
[in this case, $t=1$ and $\Phi(\tau):=\tau^p$ for any $\tau\in [0,\fz)$ with $p\in (0,\fz)$] of Abl\'e and Feuto \cite{AF}
as a special case. Moreover, the results in \cite{ZYYW} indicate that, similarly to the classical
Hardy space $H^p(\rn)$ with $p\in (0,1]$, $(HE_\Phi^q)_t(\rn)$ is a good substitute of
$(E_\Phi^q)_t(\rn)$ in the study on the boundedness of operators.
On another hand, observe that $(E_\Phi^p)_t(\rn)$ when $p=t=1$ goes back to the amalgam space $(L^\Phi,\ell^1)(\rn)$
introduced by Bonami and Feuto \cite{BF}, where
$\Phi(t):=\frac{t}{\log(e+t)}$ for any $t\in[0,\infty),$
and the Hardy space $H_*^\Phi(\rn)$
associated with the amalgam space $(L^\Phi,\ell^1)(\rn)$
was applied by Bonami and Feuto \cite{BF} to study the linear decomposition of the product of
the Hardy space $H^1(\rn)$ and its dual space $\BMO(\rn)$. Another main
motivation to introduce $(HE_\Phi^q)_t(\rn)$ in \cite{ZYYW} exists in that it is
a natural generalization of $H_*^\Phi(\rn)$ in \cite{BF}.
In the last part of this section, we focus on the weak Orlicz-slice Hardy space $(WHE_\Phi^q)_t(\rn)$
built on the Orlicz-slice space $(E_\Phi^q)_t(\rn)$. We first recall some of the useful
properties of Orlicz-slice spaces.

In Definition \ref{fine}, if $\omega\equiv1$, we then denote $L_\omega^\Phi(\rn)$
simply by $L^{\Phi}(\rn)$.

\begin{definition}
Let $t,\ q\in(0,\infty)$ and $\Phi$ be an Orlicz function with positive lower type $p_{\Phi}^-$ and
positive upper type $p_{\Phi}^+$. The \emph{Orlicz-slice space} $(E_\Phi^q)_t(\rn)$
is defined to be the set of all measurable functions $f$
such that
$$
\|f\|_{(E_\Phi^q)_t(\rn)}
:=\lf\{\int_{\rn}\lf[\frac{\|f\mathbf{1}_{B(x,t)}\|_{L^\Phi(\rn)}}
{\|\mathbf{1}_{B(x,t)}\|_{L^\Phi(\rn)}}\r]^q\,dx\r\}^{\frac{1}{q}}<\infty.
$$
\end{definition}

\begin{remark}
By \cite[Lemma 2.28]{ZYYW}, we know that the Orlicz-slice space $(E_\Phi^q)_t(\rn)$
is a ball quasi-Banach function space, but it is not a quasi-Banach function space
(see, for instance, \cite[Remark 7.41(i)]{zwyy})
\end{remark}

Now, we recall the notions of both the weak Orlicz-slice space
$(WE_\Phi^q)_t(\rn)$ and the weak Orlicz-slice Hardy spaces $(WHE_\Phi^q)_t(\rn)$
as follows.

\begin{definition}
\begin{itemize}
\item[(i)] Let $t,\ q\in(0,\infty)$ and $\Phi$ be an Orlicz function with positive lower type $p_{\Phi}^-$ and
positive upper type $p_{\Phi}^+$. The \emph{weak Orlicz-slice space} $(WE_\Phi^q)_t(\rn)$
is defined to be the set of all measurable functions $f$ such that
$$
\|f\|_{(WE_\Phi^q)_t(\rn)}:=\sup_{\alpha\in(0,\infty)}\lf\{\alpha\lf\|
\mathbf{1}_{\{x\in\rn:\ |f(x)|>\alpha\}}\r\|_{(E_\Phi^q)_t(\rn)}\r\}<\infty.
$$

\item[(ii)] Let $t,\ q\in(0,\infty)$, $N\in\nn$ and $\Phi$ be an Orlicz function with positive lower type $p_{\Phi}^-$ and
positive upper type $p_{\Phi}^+$. The \emph{weak Orlicz-slice Hardy space}
$(WHE_\Phi^q)_t(\rn)$ is defined to be the set of all $f\in\cs'(\rn)$
such that $M_N^0(f)\in(WE_\Phi^{q})_t(\rn)$ and, for any $f\in(WHE_\Phi^q)_t(\rn)$, let
$$
\|f\|_{(WHE_\Phi^q)_t(\rn)}:=\lf\|M_N^0(f)\r\|_{(WE_\Phi^{q})_t(\rn)},
$$
where $M_N^0(f)$ is as in \eqref{EqMN0} with $N$ sufficiently large.
\end{itemize}
\end{definition}

\begin{remark}
Let $t\in(0,\infty)$, $q\in(1,\infty)$ and
$\Phi$ be an Orlicz function with positive lower type $p_{\Phi}^-\in(1,\infty)$
and
positive upper type $p_{\Phi}^+$. By \cite[Remark 7.49]{zwyy},
we know that $(WHE_\Phi^q)_t(\rn)=(WE_\Phi^{q})_t(\rn)$ with
equivalent norms.
\end{remark}

We now recall the notion of the Orlicz-slice Hardy space $(HE_\Phi^q)_t(\rn)$ introduced in \cite{ZYYW}.
\begin{definition}
Let $t$, $q\in(0,\infty) $ and $\Phi$ be an Orlicz function with positive lower type $p_{\Phi}^-$ and positive upper type $p_{\Phi}^+$.
Then the \emph{Orlicz-slice Hardy space $(HE_\Phi^q)_t(\rn)$} is defined by setting
$$
(HE_\Phi^q)_t(\rn):=\lf\{f\in\mathcal{S}'(\rn):\ \|f\|_{(HE_\Phi^q)_t(\rn)}:=
\|M(f,\varphi)\|_{(E_\Phi^q)_t(\rn)}<\infty \r\},
$$
where $\varphi\in\mathcal{S}(\rn)$ satisfies
$
\int_{\rn}\varphi(x)\,dx\neq0.
$
In particular, when $\Phi(s):=s^r$ for any $s\in[0,\fz)$ with any given $r\in(0,\fz)$,
the Hardy-type space $(HE_r^q)_t(\rn):=(HE_{\Phi}^q)_t(\rn)$ is called
the \emph{slice Hardy space}.
\end{definition}

Let $t,\ q\in(0,\infty)$ and $\Phi$ be an Orlicz function with positive lower type $p_{\Phi}^-$ and
positive upper type $p_{\Phi}^+$.
Then $(E_\Phi^q)_t(\rn)$ satisfies Assumption \ref{a2.15}
(see \cite[Lemma 4.3]{ZYYW}) and Assumption \ref{a2.17}
(see \cite[Proposition 7.40]{zwyy}). Furthermore, from \cite[Lemmas 4.4]{ZYYW},
we deduce that,
for any given $r_0\in(0,\min\{1,p_{\Phi}^-\})$ and $p_0\in(\max\{1,p_{\Phi}^+\},\infty)$,
\eqref{Eqdm} with $X:=(E_\Phi^q)_t(\rn)$ holds true.
Thus, all the assumptions of main theorems in Sections \ref{s3} and \ref{s4} are satisfied.
Using Theorems
\ref{Tharea}, \ref{Thgf} and \ref{Thgx},
we immediately obtain the following characterizations of $(WHE_\Phi^q)_t(\rn)$ in terms of the
Lusin area function, the Littlewood--Paley $g$-function and
the Littlewood--Paley $g_\lambda^*$-function.

\begin{theorem}\label{fgtt}
Let $t,\ q\in(0,\infty)$  and
$\Phi$ be an Orlicz function with positive lower type $p_{\Phi}^-$ and
positive upper type $p_{\Phi}^+$. Then $f\in(WHE_\Phi^q)_t(\rn)$ if and only if either
of the following two items holds true:
\begin{itemize}
\item[\rm{(i)}] $f\in\cs'(\rn)$, $f$ vanishes weakly at infinity and $S(f)\in (WE_\Phi^q)_t(\rn)$,
where $S(f)$ is as in \eqref{eq61}.
\item[\rm{(ii)}] $f\in\cs'(\rn)$, $f$ vanishes weakly at infinity and $g(f)\in (WE_\Phi^q)_t(\rn)$,
where $g(f)$ is as in \eqref{eq63}.
\end{itemize}
Moreover, for any $f\in(WHE_\Phi^q)_t(\rn)$,
$\|f\|_{(WHE_\Phi^q)_t(\rn)}\sim\|S(f)\|_{(WE_\Phi^q)_t(\rn)}\sim\|g(f)\|_{(WE_\Phi^q)_t(\rn)}$,
where the positive equivalence constants are independent of $f$ and $t$.
\end{theorem}

\begin{theorem}\label{gx}
Let $t,\ q\in(0,\infty)$ and
$\Phi$ be an Orlicz function with positive lower type $p_{\Phi}^-$ and
positive upper type $p_{\Phi}^+$. Let $\lambda\in(\max\{\frac{2}{\min\{1,p_{\Phi}^-,\ q\}},
1-\frac2{\max\{1,p_\Phi^+,q\}}+\frac{2}{\min\{1,p_{\Phi}^-,\ q\}}\},\infty)$.
Then $f\in (HE_\Phi^q)_t(\rn)$ if and only if $f\in\mathcal{S}'(\rn)$,
$f$ vanishes weakly at infinity and $g_\lambda^*(f)\in (WE_\Phi^q)_t(\rn)$,
where $g_\lambda^*(f)$ is as in \eqref{eq62}.
Moreover, for any $f\in(WHE_\Phi^q)_t(\rn)$,
$$
\|f\|_{(WHE_\Phi^q)_t(\rn)}\sim\lf\|g_\lambda^*(f)\r\|_{(WE_\Phi^q)_t(\rn)},
$$
where the positive equivalence constants are independent of $f$ and $t$.
\end{theorem}

Applying Theorem \ref{inte} to Orlicz-slice spaces, we find that $(WHE_{\Phi_{1/(1-\theta)}}^{q/(1-\theta)})_t(\rn)$
is just the real interpolation intermediate space between $(HE_\Phi^q)_t(\rn)$
and $L^\infty(\rn)$ with $\theta\in(0,1)$ as follows.

\begin{theorem}\label{inte2}
Let $t$, $q\in(0,\infty) $ and $\Phi$ be an Orlicz function with
positive lower type $p_{\Phi}^-$ and positive upper type $p_{\Phi}^+$.
Assume that $\theta\in(0,1)$. Then
$$
((HE_\Phi^q)_t(\rn),L^\infty(\rn))_{\theta,\infty}=
(WHE_{\Phi_{1/(1-\theta)}}^{q/(1-\theta)})_t(\rn),
$$
where $\Phi_{1/(1-\theta)}$ is as in \eqref{phi1} with $\theta$ replaced by $1-\theta$.
\end{theorem}

\emph{Acknowledgements}.
The authors would like to thank Professor Sibei Yang for many
discussions on the subject of this article.



\bigskip

\noindent Songbai Wang

\smallskip

\noindent College of Mathematics and Statistics, Hubei Normal University,
Huangshi 435002, People's Republic of China

\smallskip

\noindent{\it E-mail}: \texttt{haiyansongbai@163.com}

\bigskip

\noindent Dachun Yang (Corresponding author), Wen Yuan and Yangyang Zhang

\smallskip

\noindent  Laboratory of Mathematics and Complex Systems
(Ministry of Education of China),
School of Mathematical Sciences, Beijing Normal University,
Beijing 100875, People's Republic of China

\smallskip

\noindent {\it E-mails}: \texttt{dcyang@bnu.edu.cn} (D. Yang)

\noindent\phantom{{\it E-mails:}} \texttt{wenyuan@bnu.edu.cn} (W. Yuan)

\noindent\phantom{{\it E-mails:}} \texttt{yangyzhang@mail.bnu.edu.cn} (Y. Zhang)

\end{document}